\documentclass[11pt]{article}
\pdfoutput=1

\usepackage[margin=1in]{geometry}
\usepackage{setspace}

\usepackage{amsmath, amssymb, amsthm, bm, mathtools}
\usepackage{graphicx}
\usepackage{booktabs}
\usepackage{enumitem}
\usepackage{pgfplots}
\pgfplotsset{compat=1.17}
\usepackage{comment}
\usepackage{xcolor}
\usepackage[linktocpage=true, colorlinks=true, linkcolor=red, citecolor=blue,urlcolor=magenta]{hyperref}
\usepackage[utf8]{inputenc} 
\usepackage[T1]{fontenc}    
\usepackage{hyperref}       
\usepackage{url}            
\usepackage{booktabs}       
\usepackage{amsfonts}       
\usepackage{nicefrac}       
\usepackage{microtype}      
\usepackage{xcolor}         
\usepackage{subcaption}
\usepackage{wrapfig}
\usepackage{booktabs}
\usepackage[normalem]{ulem}
\usepackage{comment}
\usepackage{enumitem}
\usepackage{float}
\usepackage{tikz}
\usetikzlibrary{arrows.meta,patterns}
\usepackage{ulem}
\usepackage{tcolorbox}
\usepackage[authoryear]{natbib}

\newcommand{\argmin}{\operatorname*{argmin}}
\newcommand{\argmax}{\operatorname*{argmax}}
\newcommand{\LyapOpt}{\mathrm{LyapOpt}}
\newcommand{\MaxWeight}{\mathrm{MaxWeight}}
\newcommand{\indep}{\perp\!\!\!\perp}

\theoremstyle{plain}
\newtheorem{theorem}{Theorem}
\newtheorem{proposition}{Proposition}
\newtheorem{lemma}{Lemma}
\newtheorem{corollary}{Corollary}
\theoremstyle{definition}

\newtheorem{assumption}{Assumption}
\theoremstyle{remark}
\newtheorem{remark}{Remark}

\usepackage{bbm}
\usepackage{makecell}

\bibpunct[, ]{(}{)}{,}{a}{}{,}%

\title{Finite‑Time Minimax Bounds and an Optimal Lyapunov Policy in Queueing Control}

\author{%
  Yujie Liu \hspace{3em} Vincent Y.~F.~Tan \hspace{3em} Yunbei Xu\\
  National University of Singapore\\
  \texttt{\{yj\string-liu, vtan, yunbei\}@nus.edu.sg}
}
\date{}

\begin{document}

\maketitle

\begin{abstract}
We introduce an original minimax framework for finite-time performance analysis in queueing control and propose a surprisingly simple Lyapunov-based scheduling policy with superior finite-time performance. The framework quantitatively characterizes how the expected total queue length scales with key system parameters, including the capacity of the scheduling set and the variability of arrivals and departures across queues. This characterization provides a systematic quantitative basis for evaluating and comparing scheduling policies in the finite-time regime, including  nonstationary settings under certain assumptions on the model, and shows that the proposed policy  provably and empirically outperforms the classical MaxWeight strategy in finite time.
Within this framework, we establish three main sets of results. First, we derive minimax lower bounds on the expected total queue length for parallel-queue scheduling via a novel Brownian coupling argument. Second, we propose a new policy, $\LyapOpt$, which minimizes the full quadratic Lyapunov drift—capturing both first- and second-order terms—and achieves optimal finite-time performance under the dominated region condition in heavy traffic while retaining classical stability guarantees. Third, we identify a key limitation of the classical MaxWeight policy, which optimizes only the first-order drift: its finite-time performance depends suboptimally on system parameters, leading to substantially larger backlogs in explicitly characterized settings. Together, these results delineate the scope and limitations of classical drift-based scheduling and motivate new queueing-control methods with rigorous finite-time guarantees.
\end{abstract}%


\section{Introduction}
\label{sec:intro}
In this paper, we study the fundamental scheduling problem in parallel queueing systems. We extend the celebrated “generalized switch” model of \citet{stolyar2004maxweight} to certain nonstationary settings and characterize its finite-time behavior. The system comprises $n$ job flows, each with its own queue, competing for a limited pool of potentially shared, time-varying many-server resources, modeled by general scheduling sets in $\mathbb{R}_{+}^n$. At each discrete time step, a central controller chooses how to allocate the available service capacity across the $n$ queues. The objective is to optimize system-wide performance measures such as total queue length, delay, and throughput.

Scheduling plays a central role in modern infrastructures. It underpins cloud and communication platforms such as data centers, wireless networks, and service systems, while also serving as a cornerstone of AI infrastructure by enabling throughput-optimized schedulers for GPU clusters that support large-scale model training and inference \citep{ao2025optimizing,li2025throughput}. Whereas classical analyses focus on long-run performance under stationary environment, contemporary platforms routinely encounter bursty, nonstationary workloads and stringent service-level objectives,  as exemplified by production systems including NVIDIA’s Run:AI and Google’s Borg \citep{nvidia2025runai,verma2015large,qiao2024conserve}. In these regimes, system performance is determined by transient congestion and workload variability, highlighting the need for rigorous finite-time guarantees. Developing scheduling policies with provable delay bounds in these nonstationary, time-critical regimes is therefore essential to ensure responsiveness and resource efficiency that next-generation systems require.

The celebrated MaxWeight scheduling policy \citep{tassiulas1992stability,stolyar2004maxweight} has been widely studied in the queueing literature. It is known to be throughput-optimal \citep{tassiulas1992stability} and to achieve optimality in the diffusion-scaled, heavy-traffic regime,  as established through elegant Lyapunov drift analysis and diffusion approximations \citep{stolyar2004maxweight, mandelbaum2004scheduling, dai2005maximum, dai2008asymptotic}. However, in finite-time regime, its performance remains poorly understood, leaving a crucial gap in both theory and practice. To study finite-time queueing performance in nonstationary, time-critical regimes, we adopt a minimax framework that evaluates policies through their worst-case finite-time performance, thereby providing robustness to workload fluctuations in the finite-time regime. The framework is applicable to settings with workload variability, such as transient demand fluctuations and ramp-up phases \citep{verma2015large,gao2024empirical}, and to time-critical situations involving service-level-agreement violations or worst-case variability planning \citep{milner2008service,bandi2015robust}.

The minimax framework introduced here is closely related to the minimax-regret theory of stochastic adaptive control with unknown plant parameters and stochastic process noise \citep{abbasi2011regret,simchowitz2020naive,ziemann2021uninformative}. A particularly relevant connection is the lower-bound approach of \citet{ziemann2021uninformative}, which captures parameter-dependent problem difficulty through local minimax regret.
The minimax-risk formulation is classical in statistical decision theory, pioneered by \citet{wald1945statistical} and developed in modern treatments \citep{lehmann1998theory,wainwright2019highdimstats,lattimore2020bandit,chen2024assouad}. Also close in spirit to our emphasis on parameter-dependent geometry is the theory of oracle complexity for stochastic optimization, which studies problem classes characterized by key system parameters \citep{agarwal2012information}.
 Our framework evaluates non-asymptotic performance over {\em model classes} rather than a single instance, enabling a unified finite-time analysis. This perspective is distinct from the instance- or problem-dependent approaches of prior works \citep{tassiulas1992stability,  stolyar2004maxweight}. We derive finite-time upper and lower bounds that hold uniformly and optimally over well-defined classes of instances, up to universal constants. These results offering a principled basis for queueing control with explicit, parameter-dependent optimality guarantees at finite horizons.

In this work, we ask the following central questions: 
\begin{itemize}
    \item How can we formulate a minimax framework for the scheduling problem in queueing systems?  Within this framework, what is the minimum achievable queue length by time $T$?
    \item Can MaxWeight attain this minimum?  If not, what alternative scheduling policies can possibly achieve it, and under what conditions?
\end{itemize}

\subsection{Contributions}
We  provide  rigorous and definitive answers to our central questions: 
\begin{enumerate}[leftmargin=*]
\item \emph{Proposal of a Minimax Framework For Queueing.} We develop what is, to our knowledge, the first finite-time minimax framework for scheduling in queueing systems that captures the dependence on key problem parameters (Section~\ref{sec:minimax_framework}). We focus on finite-time performance metrics such as the expected total queue length and the expected sum of squared queue lengths. The model class comprises capacity-constrained scheduling sets that may vary over time, nonstationary arrival processes subject to mean and variance constraints, and departure processes with bounded variance, thereby capturing a broad range of uncertainties in queueing systems. The minimax objective is to find a scheduling policy that minimizes the worst-case expected total queue length over all instances in this class. This viewpoint sidesteps classical instance-specific assumptions, enables explicit finite-time analysis (including nonstationary settings under certain assumptions on the model). It also complements classical asymptotic tools such as fluid and diffusion approximations by explicitly revealing how the queue length scales with the underlying system parameters in the finite-time regime, a level of detail that asymptotic methods inherently obscure.

\item \emph{Derivation of  Finite-Time Lower Bounds.} Within the minimax framework, we derive the first explicit finite-time minimax lower bound that applies uniformly to all scheduling policies in the generalized switch model for parallel queueing systems  (Section~\ref{sec:lower}). In the heavy-traffic regime, the bound grows on the order of $\Omega(\sqrt{T})$ over a horizon $T$, while in the interior regime it scales as $\Omega((1-\rho)^{-1})$, where $\rho$ denotes the traffic intensity;
other system parameters characterized in the bounds include the number of queues, the arrival variance, the departure variance, and the capacity parameters of the scheduling sets. This lower bound thus provides a fundamental benchmark for all scheduling policies. Its proof leverages classical minimax lower-bound theory to construct hard yet analytically tractable instances within the model class, and then analyzes the resulting stochastic queueing process through an innovative Brownian-coupling argument from applied probability. This novel, blended proof strategy may have broader applicability for establishing lower bounds in other Lyapunov-optimization problems.

\item \emph{Design and Analysis of an Optimal Lyapunov Policy.}  Building on the finite-time minimax framework, we design and analyze $\LyapOpt$, a novel scheduling policy that minimizes the \emph{full} Lyapunov drift, explicitly incorporating both the first- and second-order terms of the quadratic Lyapunov function (Section~\ref{sec:policy}). 
 We show that, under the dominated region condition, $\LyapOpt$ attains the minimax lower bound up to universal constants, thereby establishing its minimax optimality in finite time. In the absence of this condition, we derive a general upper bound and provide numerical evidence supporting the strong performance of $\LyapOpt$. By including the second-order term,  $\LyapOpt$  \emph{adapts} the scheduling decision to the \emph{geometry} of the queue length vector, rather than defaulting to extreme points as in traditional drift-based methods like MaxWeight.  This enhancement improves the finite-time 
 performance, particularly in scenarios with asymmetric scheduling sets---for example, when servers have significantly different service rates. Moreover, $\LyapOpt$ retains the classical stability guarantees proven for MaxWeight whenever the traffic intensity is strictly less than 1 (Section~\ref{sec:stability}), since it directly minimizes the one-step Lyapunov drift and the standard stability arguments for MaxWeight carry over verbatim.
 

\item \emph{Identification of a   Key Limitation of MaxWeight.} Within the finite-time minimax framework, we show that MaxWeight is not minimax optimal   (Section~\ref{sec:mw-gap}).  In particular, its expected total queue length  exceeds the minimax lower bound by a factor that depends explicitly on the capacity of the scheduling set. In highly asymmetric scheduling sets, even under deterministic arrivals and departures, MaxWeight exhibits a $\sqrt{T}$ growth for some range of the finite time $T$, whereas $\LyapOpt$ maintains a constant total queue length independent of $T$ (Section~\ref{sec:mw-gap}). This finding does not contradict MaxWeight's classical optimality in the diffusion-scaled, heavy-traffic regime; rather, it underscores that traditional drift-based methods inevitably overlook finer geometric structures which become especially important when considering the more practically relevant finite-time regime. 
\end{enumerate}

\noindent Beyond the technical contributions highlighted above, two broader perspectives merit careful attention.
\paragraph{Bridging Queueing Theory and Nonasymptotic Learning.}
Our work develops a finite-time minimax framework for the generalized switch model, a canonical model in queueing control, drawing on established minimax traditions in statistics, control theory, and operations research. The framework connects the nonasymptotic, lower-bound perspective of learning theory with the structural models of queueing control and offers a complementary approach to studying real-time decision making in queueing systems.
\paragraph{Regret and Queue Length.}
A common approach to finite-time analysis in reinforcement learning, adaptive control, and operations research evaluates a policy through regret relative to an \emph{oracle} that knows the system. We likewise study the minimax performance of policies without prior knowledge of the system, but analyze total queue length directly rather than using regret as an intermediate quantity. This distinction is primarily methodological rather than fundamental. Three features are central in our framework: (i) the structured generalized-switch queueing model; (ii) the explicit identification of the system parameters and problem classes that determine the relevant statistical difficulty; and (iii) an absolute comparison of the scheduling policy with the empty-queue baseline, rather than only an oracle-relative comparison. In the setting considered here, the reflected queueing recursion makes the oracle benchmark comparatively direct to control, whereas comparison with the empty-queue baseline provides a direct guarantee on queue accumulation itself.

\subsection{Related Works}\label{sec:related}

\paragraph{Stability and Diffusion Optimality of MaxWeight.} The MaxWeight policy (also referred to as the backpressure policy in multi-hop queueing networks) was introduced by \citet{tassiulas1992stability} and has since been extensively studied (see, e.g., \citet{andrewsscheduling2004, Tassiulas01012000, stolyar2004maxweight, mandelbaum2004scheduling, dai2005maximum, dai2008asymptotic}). Much of this literature has established the policy's throughput optimality across a wide range of queueing models, demonstrating that MaxWeight stabilizes the queue length process whenever stability is achievable \citep{tassiulas1992stability, Tassiulas01012000, andrewsscheduling2004, dai2005maximum}.
Beyond throughput guarantees, another stream of research has shown that the MaxWeight  policy minimizes the diffusion-scaled workload process among all scheduling policies in the heavy-traffic regime, both in parallel queueing systems \citep{stolyar2004maxweight, mandelbaum2004scheduling} and in fairly general stochastic processing networks \citep{dai2008asymptotic}. These diffusion-scaled analyses are a major step toward understanding the asymptotic efficiency of MaxWeight and its connection to reflected Brownian motion models. However, because they focus on diffusion-scaled dynamics, the resulting characterizations capture only coarse parameter dependence in finite time—typically through first- and second-order moments—while other key quantities (such as the capacity parameters of the scheduling sets) do not appear explicitly. As a result, such asymptotic approximations cannot reveal how performance depends on these system parameters over a fixed, finite time horizon.

\paragraph{Minimax Bounds in Queueing Theory and Reinforcement Learning.}
Queueing control provides a canonical class of structured dynamic programming problems. A large literature in queueing theory develops transient analyses, delay and backlog bounds, heavy-traffic approximations, and policy-independent performance lower bounds. In particular, prior work establishes time-uniform transient upper bounds on expected queue size for input-queued switches \citep{shah2016queuesize,xu2020improved}, as well as steady-state upper and lower bounds that are optimal over broad policy classes and can be interpreted as minimax results for interior-regime classes \citep{shah2014optimal}. Our contribution differs in scope: we formulate an explicit prescribed-horizon minimax criterion and, in the heavy-traffic regime, derive lower bounds that display their dependence on both the horizon $T$ and key system parameters. Under additional structural conditions, we also obtain finite-time upper bounds matching the leading $\sqrt{T}$ dependence. Furthermore, a substantial literature studies the nonasymptotic statistical performance of reinforcement-learning, stochastic-control, and adaptive-control methods through finite-sample guarantees, regret bounds, and minimax lower bounds
\citep{jaksch2010near,azar2017minimax,abbasi2011regret,
simchowitz2020naive,ziemann2021uninformative}.

\paragraph{Existing Finite-Time Analysis in Queueing.} Understanding queueing behavior in the finite-time regime remains challenging, even for basic models 
such as the $M/M/1$ queue \citep{abate1987transient}. Much of the existing literature focuses on uncontrolled systems, examining the convergence rate to steady state—typically via coupling arguments for the $M/M/1$ queue \citep{robert2013stochastic} or spectral methods for the $GI/M/n$ queue in the heavy traffic regime \citep{gamarnik2013rate,van2011transient}. Finite-time 
analyses for controlled systems is far less developed. There are a  few results  for routing policies such as Join-the-Shortest-Queue (see, e.g., \citep{luczak2006maximum,ma2025convergence}). \citet{luczak2006maximum} use mean-field and coupling techniques to characterize the convergence rate of the maximum queue length to equilibrium in large-scale systems, while \citet{ma2025convergence} establish the convergence rate for the two-server symmetric case based on the original system rather than its diffusion-scaled counterpart. In contrast, our framework directly analyzes finite-time performance in the original system, without relying on diffusion-scaled regimes or making any reference to steady-state behavior. 

\paragraph{Parameter Learning in Queueing.} Some strands of online learning research in queueing systems study scheduling problems in which either  service statistics~\citep{krishnasamy2018learning,krishnasamy2021learning,yang2023learning,stahlbuhk2021learning,freund2023quantifying,liang2018minimizing,dai2024adversarial} or the utility function~\citep{dai2024adversarial,hsu2022integrated} are initially unknown and must be learned online. A significant body of work establishes upper bounds on the time-averaged queue length of learning algorithms within the finite-time regime~\citep{yang2023learning,dai2024adversarial}, with aims such as ensuring system stability or evaluating utility maximization relative to benchmark policies~\citep{dai2024adversarial}. Another common performance metric in the literature is the queue length regret, which quantifies the excess backlog incurred by a learning algorithm relative to an oracle policy such as the MaxWeight or $c\mu$ policies (see, e.g., \citet{krishnasamy2018learning,krishnasamy2021learning,stahlbuhk2021learning,freund2023quantifying,liang2018minimizing}). Despite these advances, little is known about the optimal policies in the finite-time regime when system parameters are already known—beyond the classical result that the $c\mu$ rule is optimal in the single-server case~\citep{hirayama1989further}. To address  this gap in the literature, we go beyond the parameter learning paradigm and propose a minimax formulation with the expected total queue length as the performance metric, thereby elucidating the fundamental finite-time limits of scheduling policies and identifying a minimax-optimal strategy.

\paragraph{Lyapunov Drift Methods and Related Policies.} Lyapunov drift methods are widely used to establish stability and characterize steady-state behavior in queueing networks \citep{eryilmaz2012asymptotically,maguluri2016heavy}. Many control policies can, in fact, be interpreted as solutions that optimize the one-step Lyapunov drift. The MaxWeight policy directly minimizes the first-order term of the one-step quadratic Lyapunov drift, whereas the drift-plus-penalty policy \citep{neely2010stochastic} minimizes the same term augmented with auxiliary objectives such as delay or energy consumption, thereby ensuring stability while simultaneously optimizing a desired performance criterion. Going beyond the first-order term, the policies in \cite{yu2018new} and \cite{wang2017towards} address the second-order term in the Lyapunov drift by approximating it using the previous-slot net input and incorporating a proximal regularization term to control the resulting approximation error. In contrast, our proposed $\LyapOpt$ policy directly minimizes the {\em full Lyapunov drift}; this includes both the first- and second-order terms. In particular, the second-order term drives the departure to track the instantaneous arrival rate, thereby directly optimizing  the  finite-time performance. This refinement highlights that optimizing only the first-order term, as done in traditional Lyapunov drift-based methods, is too coarse to achieve finite-time optimality. 

The preceding discussion reviewed related works from several perspectives. After introducing our model and the minimax framework in Section~\ref{sec:setup},  we present a comprehensive and detailed comparison of our approach with the most closely related studies  in Table~\ref{tab:comparison} in Section~\ref{sec:minimax_framework}. 

\subsection{Outline of Paper}
The rest of the paper is organized as follows. Section~\ref{sec:setup} presents the model setup and establishes the minimax framework. Section~\ref{sec:lower} develops the minimax finite-time lower bound. In Section~\ref{sec:upper}, we propose the $\LyapOpt$ policy, establish its minimax optimality under the dominated region condition in the finite-time regime, and characterize a limitation of the MaxWeight policy. Section~\ref{sec:stability} further proves the stability of $\LyapOpt$. Section~\ref{sec:exp} presents numerical experiments comparing $\LyapOpt$ with MaxWeight. Section~\ref{sec:extension} extends the results by deriving the minimax finite-time lower bound under specific conditions and considering additional model classes. Finally, Section~\ref{sec:conclusion} concludes the paper and discusses directions for future research.

\section{Problem Setup of Queueing Control and Minimax Framework}
\label{sec:setup}

\subsection{The Nonstationary Generalized Switch Model}\label{sec: problem_setup}

\paragraph{Notation.}
Let $\mathbb{N} = \{1, 2, \ldots\}$ be the set of natural numbers, $\mathbb{N}_0 := \mathbb{N} \cup\{0\}$, $\mathbb{R}^n_+ := \{x \in \mathbb{R}^n : x_i \ge 0,\ 1\leq i \leq n\}$ and $\mathbb{R}^n_{++} := \{x \in \mathbb{R}^n : x_i > 0,\ 1\leq i \leq n\}$. Let $\mathbb{Z}_+^n := \{x \in \mathbb{Z}^n : x_i \ge 0,\ 1\leq i \leq n\}$ be the set of integer-valued vectors. For any positive integer $n$, let $[n]:=\{1,\ldots,n\}$.
For any set $U \subseteq \mathbb{R}_+^n$, let $\mathrm{conv}(U)$ denote its convex hull, and define $\Pi(U) = \big\{ \gamma \in \mathbb{R}_+^n : \exists\, u \in \text{conv}(U)\ \text{with}\ \gamma \le u \big\}.$ That is, $\Pi(U)$ consists of all nonnegative vectors that are componentwise bounded above by some convex combination of elements of $U$.
We use $\mathbf{0}$ to denote the zero vector, with the dimension understood from context. 
For $x,y \in \mathbb{R}^n$, the componentwise maximum and minimum are defined by $(\max\{x,y\})_i = \max\{x_i,y_i\}$ and $(\min\{x,y\})_i = \min\{x_i,y_i\}$ for $1 \le i \le n$. $\langle x, y \rangle$ denotes the inner product of $x$ and $y$. Independence of random elements $X$ and $Y$ is written  as $X \indep Y$, while conditional independence of $X$ and $Y$ given a $\sigma$-algebra $\mathcal{G}$ is written $X \indep Y \mid \mathcal{G}$. We adopt the convention $\frac{c}{0}=+\infty$ for any constant $c>0$.


\paragraph{Probability space.}
We work in a complete probability space $(\Omega,\mathcal{F},\mathbb{P})$. 
Expectation and variance under $\mathbb{P}$ are denoted as $\mathbb{E}[\cdot]$ and $\text{Var}(\cdot)$, respectively, and ``a.s.'' is a shorthand for almost surely. 
A stochastic process $X:\mathbb{N}_0\times\Omega\to E$ is denoted as $\{X(t)\}_{t\in \mathbb{N}_0}$, where $X(t)=X(t,\omega)$ refers to the random variable at time $t\in \mathbb{N}_0$. 

\paragraph{Queueing Model and Dynamics.}
We consider the discrete-time  {\em  nonstationary generalized switch model} that generalizes the framework of \citet{stolyar2004maxweight} to allow for time-varying arrivals and scheduling sets. This model consists of  $n$ parallel infinite-capacity queues, where jobs arrive exogenously and depart after a single service. 
The queue length process is $Q:\mathbb{N}_0\times\Omega\to\mathbb{R}^n_+$ with $Q(0)=\mathbf{0}$ almost surely. The arrival process is $A:\mathbb{N}\times\Omega\to\mathbb{R}^n_+$, where $A_i(t)$ is the number of jobs arriving to queue $i$ for $1\leq i\leq n$. We consider the scheduling sequence (also referred to as a {\em switch}) $\mathcal{D}:=\{\mathcal{D}_t\}_{t\in\mathbb{N}_0}$, where each scheduling set $\mathcal{D}_t\subset\mathbb{R}^n_+$ is compact. 
An element of $\mathcal{D}_t$, termed a \emph{schedule}, specifies the expected number of jobs that can depart from each queue at time $t$. At each time $t$, the controller chooses $d(t)\in\mathcal{D}_t$ according to a (possibly randomized) policy, and the realized departure vector $D(t)$ is a random realization consistent with $d(t)$, as formalized in~\eqref{definition_d} in the following.
The queue lengths then evolve according to
\begin{align}
  Q(0) &=\textbf{0};\nonumber\\
  Q(t+1) &= \max\{Q(t) - D(t), \mathbf{0}\} + A(t+1), \ t \in \mathbb{N}. \label{queue_recursion}
\end{align}

\paragraph{Arrival Process.} In our model, we allow the arrivals to be nonstationary. That is, the distribution of $A(t)$ may vary with time $t$, while the sequence $\{A(t)\}_{t\in \mathbb{N}}$ remains independent across $t$. Moreover, for each fixed $t$, the components $\{A_i(t)\}_{i=1}^n$ may exhibit arbitrary dependence across queues indexed by $1\le i\le n$.
Let $\lambda(t):=\mathbb{E}[A(t)]\in\mathbb{R}^n_+$ be the arrival rate vector at time $t$.

\paragraph{Policy.}
  Let the history (i.e., the information available) prior to choosing $d(t)$  at time $t $ be
\begin{align*}
  \mathcal{F}_{t}: =\sigma\{D(0), A(1),  \cdots, D(t-1), A(t)\}.  
\end{align*}
A \emph{policy} $\Phi$ is defined as a sequence of scheduling rules, $\Phi=\{\phi_t\}_{t \in \mathbb{N}_0}$.
Each $\phi_t$ specifies, based on the history $\mathcal{F}_{t}$, a probability distribution over the scheduling set $\mathcal{D}_t$.
Formally, given $\mathcal{F}_{t}$, the schedule $d(t)\in\mathcal{D}_t$ is sampled according to
\[
\mathbb{P}\big(d(t)\in \cdot \,\big|\, \mathcal{F}_{t}\big)=\phi_t(\cdot).
\]

\paragraph{Departure Mechanism and Capacity Region.}
The departure variability is modeled by a \emph{random departure field} $S= \{S_t\}_{t\in \mathbb{N}_0}$, where
\[
S_t: \mathcal{D}_t\times\Omega \to \mathbb{R}^n_+,
\qquad (d,\omega)\mapsto S(d,\omega),
\]
 is jointly measurable with the natural product $\sigma$-algebra. Write $S_d(t): =S_t(d,\omega)$. For each $t \in \mathbb{N}_0$, the $\sigma$-algebra generated by $\{S_d(t):d\in\mathcal{D}_t\}$ is independent of $\mathcal{F}_{t}$. Moreover, the random departure field $S$ is assumed to be independent of the arrival process $A$.
The field satisfies \emph{mean consistency}:
\[
\mathbb{E}\![S_d(t)]=d \quad \text{for all } t\in\mathbb{N}_0,\ d\in\mathcal{D}_t.
\]
At time $t$, once the policy selects $d(t)\in\mathcal{D}_t$, the departure vector is realized by
\begin{align}\label{definition_d}
    D(t):=S_{d(t)}(t).
\end{align}
Consequently, 
\[
D(t)\ \perp\!\!\!\perp\ \mathcal{F}_{t} \mid d(t),\quad \mathbb{E}\!\left[D(t)\mid \mathcal{F}_{t},d(t)\right]=d(t) \ \text{a.s.}
\]


Our goal is to minimize the queue length in the finite-time regime. To achieve this, we develop a minimax framework that captures the fundamental finite-time performance limits of the scheduling problem and provides a benchmark against which all scheduling policies can be evaluated.


\subsection{Relationship Between the Nonstationary Generalized Switch Model and Other Common Models}
To clarify the scope of the nonstationary generalized switch model considered in this paper, we briefly illustrate how several canonical scheduling systems can be represented within the scheduling-set abstraction and discuss which settings are covered by the model and which fall outside it.

\paragraph{Single- and Multi-server models}
The scheduling decisions in multi-server models can be represented through the scheduling set $\mathcal{D}_t$. Two illustrative examples are shown in
Table~\ref{tab:service-matrix}. The number in the row for queue $i$ and the column for server $j$ specifies the service rate of server $j$ for queue $i$ (a zero indicates that server $j$ cannot serve queue $i$). Each server  serves at most one
queue per slot. If service is not infinitely divisible, then in the single-server example (Table~\ref{tab:service-matrix}(a)), the instantaneous scheduling set is $\mathcal{D}_1=\{(1,0),\,(0,2)\}$. In the two-server example (Table~\ref{tab:service-matrix}(b)), the
instantaneous scheduling set is $\mathcal{D}_2=\{(1,0),\,(0,1.5),\,(2,0),\,(1,1.5)\}$. In contrast, if the service is infinitely divisible, then the corresponding scheduling sets become $\Pi(\mathcal{D}_1)$ and $\Pi(\mathcal{D}_2)$, respectively. Figure~\ref{fig:scheduling-sets}
depicts the sets $\mathcal{D}_1$ and $\mathcal{D}_2$ together with their associated regions
$\Pi(\mathcal{D}_1)$ and $\Pi(\mathcal{D}_2)$.

\begin{table}[htbp]
  \centering
  \begin{subtable}{0.48\linewidth}
    \centering
    \caption{Single server}
    \label{tab:single}
    \begin{tabular}{@{}lc@{}}
      \toprule
      & Server 1 \\ \midrule
      Queue 1 & 1 \\
      Queue 2 & 2 \\ \bottomrule
    \end{tabular}
  \end{subtable}\hfill
  \begin{subtable}{0.48\linewidth}
    \centering
    \caption{Two servers}
    \label{tab:multi}
    \begin{tabular}{@{}lcc@{}}
      \toprule
      & Server 1 & Server 2 \\ \midrule
      Queue 1 & 1 & 1 \\
      Queue 2 & 0 & 1.5 \\ \bottomrule
    \end{tabular}
  \end{subtable}
  \caption{Service tables}
  \label{tab:service-matrix}
\end{table}

\begin{figure}[H]
\centering

\begin{minipage}[t]{0.45\textwidth}
\centering
\begin{tikzpicture}[scale=1.2, >=Stealth]
\path[use as bounding box] (-0.4,-0.55) rectangle (2.75,1.85);

\draw[thick,->] (-0.05,0) -- (2.55,0) node[right] {$Q_1$};
\draw[thick,->] (0,-0.05) -- (0,1.75) node[above] {$Q_2$};

\coordinate (O) at (0,0);
\coordinate (A) at (0,1);
\coordinate (B) at (2,0);

\fill[gray!35] (O) -- (A) -- (B) -- cycle;

\draw[thick] (A) -- (B);

\filldraw[blue] (A) circle (1.3pt);
\filldraw[blue] (B) circle (1.3pt);

\node[left] at (0,1) {\small $1$};
\node[below] at (2,0) {\small $2$};

\draw[->, thick]
  (1.25,1.15) .. controls (1.05,0.95) and (0.85,0.78) .. (0.6,0.58);
\node at (1.55,1.38) {$\Pi(\mathcal{D}_1)$};

\end{tikzpicture}
\end{minipage}
\hfill
\begin{minipage}[t]{0.45\textwidth}
\centering
\begin{tikzpicture}[scale=1.2, >=Stealth]
\path[use as bounding box] (-0.4,-0.55) rectangle (2.95,1.85);

\draw[thick,->] (-0.05,0) -- (2.75,0) node[right] {$Q_1$};
\draw[thick,->] (0,-0.05) -- (0,1.75) node[above] {$Q_2$};

\coordinate (O) at (0,0);
\coordinate (A) at (0,1.25);
\coordinate (B) at (1,1.25);
\coordinate (C) at (2,0);
\coordinate (D) at (1,0);

\fill[gray!35] (O) -- (A) -- (B) -- (C) -- cycle;

\draw[thick] (A) -- (B) -- (C);
\draw[dashed] (B) -- (D);

\filldraw[blue] (A) circle (1.3pt);
\filldraw[blue] (B) circle (1.3pt);
\filldraw[blue] (C) circle (1.3pt);
\filldraw[blue] (D) circle (1.3pt);

\node[left] at (0,1.25) {\small $1.5$};
\node[below] at (1,0) {\small $1$};
\node[below] at (2,0) {\small $2$};

\draw[->, thick]
  (1.95,1.15) .. controls (1.75,0.98) and (1.52,0.82) .. (1.02,0.58);
\node at (2.18,1.35) {$\Pi(\mathcal{D}_2)$};
\end{tikzpicture}
\end{minipage}

\caption{Scheduling sets induced by the service matrices in Table~\ref{tab:service-matrix}.}
\label{fig:scheduling-sets}
\end{figure}

\paragraph{Input-queued switch}
In an input-queued switch \citep{mckeown1999achieving}, queues correspond to input--output pairs, and each slot admits a matching between inputs and outputs. Each feasible matching induces a binary departure vector indicating which queues are served in that slot. In this case, the scheduling set $\mathcal{D}_t$ is the set of departure vectors induced by all feasible matchings. 
\paragraph{Scope of the model}
Our framework is appropriate for systems with parallel multi-class queues in which queue lengths are observable and the control action at each time consists solely of selecting a schedule. Within this class of systems, and under suitable conditions, it also accommodates nonstationary arrivals and time-varying scheduling sets, thereby allowing the available service resources to vary over time. It is worth emphasizing that our notion of nonstationarity is deliberately formulated in a maximally general (and necessarily abstract) way, allowing arbitrary arrivals and scheduling sets; this stands in contrast to much of the prior literature, which models nonstationarity through specific Markov process assumptions and particular structural patterns. However, it does not directly cover multi-hop networks \citep{dai2005maximum, dai2008asymptotic}, nor does it apply to models involving other forms of control, such as routing \citep{yu2018new, wang2017towards}.

\subsection{Minimax Criteria}\label{sec:minimax_framework}
The minimax criterion is a standard tool for characterizing the intrinsic difficulty of problems in statistics and machine learning \citep{wald1945statistical}. In this work, we bring this learning-theoretic perspective to queueing control, seeking to connect queueing theory with learning-theoretic methodology and to motivate the use of statistical decision-theoretic tools in a canonical class of structured control problems.


\paragraph{Performance Metrics.} 
We consider two standard performance metrics: the total queue length (see~\citet{neely2010stochastic}), which captures the overall system backlog, and the sum of squares of queue lengths (see~\citet{stolyar2004maxweight}), which reflects the degree of imbalance across queues. These are formally defined as
\begin{align*}
 \mathbb{E}\!\left[\sum_{i=1}^{n} Q_i(T)\right] \quad \text{and} \quad
\mathbb{E}\!\left[\sum_{i=1}^{n} Q_i(T)^2\right] ,
\end{align*}
respectively. In this work, we primarily present results for the total queue length, while the corresponding matching bounds for the sum of squares objective are provided in Appendix~\ref{app_sec:square_queue_length}.
 Note that this is a dynamic programming problem, where the objective at time $t$ is a function of all past decisions $d(0), d(1), \ldots,d(t-1)$. A central challenge in dynamic programming and control is to develop solutions that are both computationally tractable and statistically efficient.

We adopt total queue length as the performance metric because the objective of this paper is to characterize the system’s transient congestion level in the finite-time regime. In comparison, the time-average queue length $\frac{1}{T}\sum_{t=1}^T\mathbb{E}\!\left[\sum_{i=1}^{n} Q_i(t)\right]$ aggregates congestion over time and is therefore coarser for this purpose. Delay \citep{shortle2018fundamentals} is also an important job-level metric, but it is less convenient in our finite-time scheduling setting because it depends on future scheduling decisions. Tail-risk metrics \citep{asmussen2003applied,stolyar2001largest}, such as delay-violation probabilities, are likewise important, but they are intended to capture rare-event reliability rather than {\em typical} transient congestion, which is the focus of the present work.

\paragraph{Model Class.}
The queueing system under consideration is specified by a scheduling sequence $\mathcal{D}$, an arrival process $A$, and a random departure field
$S$.
Adopting a minimax perspective,  we define the following model class that includes capacity-constrained scheduling sequence, arrival processes subject to mean and variance constraints, and departure processes with bounded variance.
\begin{align*}
     \mathcal{M}^ {\rho}(B,C_{\mathrm{a}},C_{\mathrm{d}}) = \left\{ (\mathcal{D},A,S)\,: \, \parbox[c]{3.3in}{$\frac{1}{n}\sum_{i=1}^{n} d_i^2 \leq B^2, \,\,\, \forall d\in\mathcal{D}_t,    t\in \mathbb{N}_0;$ \vspace{0.04 in} \\ $\lambda(t)\in \rho\Pi(\mathcal{D}_t)$, $\frac{1}{n}\sum_{i=1}^{n} \mathrm{Var}(A_i(t)) \leq C_{\mathrm{a}}^2,\,\,\, \forall\, t\in \mathbb{N};$ \vspace{0.04 in} \\  $\frac{1}{n}\sum_{i=1}^{n} \mathrm{Var}\!\left((S_d(t))_i\right) \leq C_{\mathrm{d}}^2, 
    \,\,\, \forall\, d\in\mathcal{D}_t,  t\in \mathbb{N}_0$ }
      \right\}
\end{align*}
where the scheduling sequence $\mathcal{D}$, the arrival process $A$, and the random departure field
$S$ are those defined in Section~\ref{sec: problem_setup}. The parameter $\rho \in (0,1]$ denotes the {\em traffic intensity}; $B>0$ constrains the capacity of the scheduling set; and $C_{\mathrm{a}}\geq 0$ and $C_{\mathrm{d}}\geq 0$ respectively control the variability of the arrivals and departures across queues. We refer to $\rho \uparrow 1$  as the {\em  heavy-traffic regime} and $\rho \in (0,1)$ as the {\em interior} (of the capacity region) regime. 



Different domination conditions relating arrival rates to scheduling capacity have been studied in prior work.
\citet{stolyar2004maxweight} and~\citet{neely2010stochastic} model time-varying scheduling sets as finite-state Markov chains and require arrival rates to be dominated by schedules averaged under the stationary distribution, whereas \citet{lim2013stability} and \citet{liang2018minimizing} consider adversarial settings in which it suffices that the time-averaged arrival rate vector lies in the interior of the time-averaged capacity region over a given time window. By contrast, we allow the scheduling sequence to be time-varying and impose the domination constraint  $\lambda(t)\in \rho \Pi(\mathcal{D}_t)$ {\em at each time step}~$t$. This stronger {\em pointwise condition} enables a finer characterization of finite-time performance than window-based assumptions, but it does not capture transient overload episodes that may arise under highly nonstationary workloads. Extending the present analysis from the pointwise condition to a window-based condition considered, for example, in \citet{yang2023learning} appears to be technically challenging. A more in-depth  discussion is provided in Appendix~\ref{app:disc_pointwise_con}.

We aim to characterize the fundamental minimax expected total queue length at time $T$:  
\begin{align}\label{minmax_frame}
  \inf_{\Phi} \sup_{(\mathcal{D},A,S) \in \mathcal{M}^ {\rho}(B,C_{\mathrm{a}}, C_{\mathrm{d}})} \mathbb{E}^\Phi_{ (A, S)}\left[\sum_{i=1}^{n} Q_i(T)\right],
\end{align}
which we denote by $\Gamma^ {\rho}(B,C_{\mathrm{a}}, C_{\mathrm{d}},T)$.
To our knowledge, this work introduces the first explicit finite-time minimax formulation for characterizing how the fundamental limits of queueing control depend on key problem parameters, providing a principled approach to quantifying the difficulty of structured dynamic scheduling problems.

Table~\ref{tab:comparison} provides a structured comparison of most related works on scheduling in queueing systems. ``Resource Type'' specifies the fundamental scheduling unit, such as individual servers or the more abstract scheduling set $\mathcal{D}_t$.  ``Arrival Pattern'' characterizes the arrival process, represented by $A=\{A(t)\}_{t\in\mathbb{N}_0}$, distinguishing between stationary arrivals, which are i.i.d.\ over time, and nonstationary arrivals, which may vary over time. ``Unknown Parameters'' lists the system characteristics that are not assumed to be known a priori and must be inferred. Finally, ``Performance Metric'' specifies the criterion used to evaluate policies; for our setting, this is defined precisely in~\eqref{minmax_frame}.

\begin{table}[htbp]
\centering
\setlength{\tabcolsep}{4pt}
\begin{small}
\begin{tabular}{|c|c|c|c|c|c|}
\hline
& \makecell{Resource \\ Type}
& \makecell{Arrival \\ Pattern} 
& \makecell{Unknown \\ Parameters} 
& \makecell{Performance \\ Metric} 
\\ \hline

\cite{krishnasamy2021learning} 
& Server
& Stationary 
& \makecell{Arrival rate \\ Service rate}
& \makecell{Queue length \\ regret}  
 \\ \hline

\cite{yang2023learning} 
& Server 
& Nonstationary
& \makecell{Arrival rate \\ Service rate}
& \makecell{Time-averaged \\ queue length}  
\\ \hline

\cite{stolyar2004maxweight} 
& \makecell{Scheduling \\ set}
& Stationary 
& Arrival rate
& \makecell{Workload in \\diffusion-scaled,\\ heavy-traffic regime}  
\\ \hline

Ours 
& \makecell{Scheduling \\ set}
& \makecell{Nonstationary  }
& Arrival rate
& \makecell{Finite-time \\ queue length} 
 \\ \hline
\end{tabular}
\end{small}
\caption{Comparison of most related works on scheduling in queueing systems.}
\label{tab:comparison}
\end{table}

As shown in Table~\ref{tab:comparison}, our framework represents a departure from traditional online learning approaches which typically focus on estimating unknown service rates to ensure stability. Additionally, our framework extends beyond classical diffusion-scaled optimality results. It addresses the {\em finite-time} scheduling problem and aims to identify the minimax-optimal policy in a nonstationary environment under pointwise conditions, i.e., $\lambda(t)\in \rho \Pi(\mathcal{D}_t)$  at each time step  $t$.

\section{Minimax Lower Bound}
\label{sec:lower}

In this section, we derive a finite-time minimax lower bound on the expected total queue length for the system introduced in Section \ref{sec:setup}. Formally, we seek a lower bound for the quantity \eqref{minmax_frame}. To this end, we employ a minimax approach together with Brownian coupling techniques to construct hard yet analytically tractable instances within the model class and analyze their behavior.  The resulting bound explicitly quantifies how queue lengths scale with  the  time horizon $T$, the traffic intensity $\rho$, the capacity parameter $B$ imposed by the scheduling sequences, and the variance parameters $C_{\mathrm{a}}$ and $C_{\mathrm{d}}$ on the arrival processes and random departure fields.

Intuitively, this lower bound reveals the challenging scenarios an adversary can construct under the prescribed capacity and variance constraints, thereby establishing fundamental performance limits for any scheduling policy in parallel queueing systems.

\subsection{The Bound and Its Interpretation}
The following theorem formalizes this lower bound for arbitrary scheduling policies within the model class $\mathcal{M}^ {\rho}(B,C_{\mathrm{a}}, C_{\mathrm{d}})$.

\begin{theorem}[Minimax Lower Bound]\label{thm_lower_bound}
For any scheduling policy, and for scheduling sequences, arrival processes and random departure fields within the model class $\mathcal{M}^ {\rho}(B,C_{\mathrm{a}}, C_{\mathrm{d}})$, there exist an absolute constant $c_1>0$ and a finite constant $\psi_1(n,B,C_{\mathrm{a}},C_{\mathrm{d}})>0$ such that, for
\begin{small}
\begin{align*}
\rho \in
\begin{cases}
\left[\big(1-\tfrac{\sqrt{C_{\mathrm{a}}^2+C_{\mathrm{d}}^2}}{B\psi_1(n,B,C_{\mathrm{a}},C_{\mathrm{d}})}\big)\vee \frac12, 1\right],
& \text{if } C_{\mathrm{a}}^2 + C_{\mathrm{d}}^2 > 0,\\
(0, 1], & \text{if } C_{\mathrm{a}}^2 + C_{\mathrm{d}}^2 = 0,
\end{cases}
\end{align*}
\end{small}
the following lower bound holds for all $T \geq \psi_1(n,B,C_{\mathrm{a}},C_{\mathrm{d}})^2\mathbf{1}_{\{C_{\mathrm{a}}^2+C_{\mathrm{d}}^2>0\}}+1$, 
\begin{align}\label{ieq:lower_bound_dependent_rho}
  \Gamma^ {\rho}(B,C_{\mathrm{a}}, C_{\mathrm{d}},T) \geq 
  c_1\min\bigg \{ n\sqrt{(C_{\mathrm{a}}^2+C_{\mathrm{d}}^2)(T-1)},
   \frac{n(C_{\mathrm{a}}^2+C_{\mathrm{d}}^2)}{B(1-\rho)} \bigg\} +n\rho B. 
\end{align}
\end{theorem}
\begin{remark}\label{rem:psi1-scaling} 
The finite constant $\psi_1(n,B,C_{\mathrm a},C_{\mathrm d})$ can in fact be chosen independent of $n$ and  an explicit  expression in the remaining parameters $ B,C_{\mathrm a},C_{\mathrm d}$ can be found in Appendix~\ref{appendix_rem_scaling}. In addition, an   approximate expression when $\frac{B}{\sqrt{C_{\mathrm a}^2+ C_{\mathrm d}^2 }}$ is either large or small is also presented therein.
\end{remark}


Theorem~\ref{thm_lower_bound} demonstrates that, for any scheduling policy and any system within the model class $\mathcal{M}^\rho(B, C_{\mathrm{a}}, C_{\mathrm{d}})$, the expected total queue length scales at least on the order of $\Omega\big(\min\{\sqrt{T}, (1-\rho)^{-1}\}\big)$, up to problem-dependent constants. 

The first term inside the $\min$ in~\eqref{ieq:lower_bound_dependent_rho} is active in the heavy traffic regime where $\rho \uparrow 1$. It reflects the intrinsic stochastic fluctuations over a finite horizon. When $\rho$ assumes its extremal value of $1$, the lower bound scales as $n \sqrt{(C_{\mathrm{a}}^2 + C_{\mathrm{d}}^2) T}$, and this scaling persists for all $T$. 

The second term in the $\min$ becomes dominant in the interior regime, when $\rho$ is strictly less than~$1$ and $T$ is sufficiently large. It captures the long-term congestion effect and depends explicitly  on the traffic intensity parameter. As $\rho$ approaches~$1$, the $\sqrt{T}$ term dominates for sufficiently large but finite horizons, indicating that the lower bound grows with a rate of at least   $n \sqrt{(C_{\mathrm{a}}^2 + C_{\mathrm{d}}^2) T}$. Beyond this point, the lower bound effectively plateaus, with a value determined by the system parameters $B, C_{\mathrm{a}}, C_{\mathrm{d}}$, and the traffic intensity $\rho$. This dependence of the queue length on the reciprocal of $1-\rho$ in the interior regime  has also been observed in previous work \citep{eryilmaz2012asymptotically}.

Overall, the  two terms in the lower bound together establish a fundamental performance limit that no scheduling policy can surpass.

\subsection{Proof Sketch of Theorem~\ref{thm_lower_bound}}
As the proof of Theorem \ref{thm_lower_bound} contains some interesting and novel elements, we provide a sketch of the four main steps here. The complete details of the proof are deferred to Appendix~\ref{appendix_thm_lower_bound_general}.



\begin{proof}[Proof Sketch of Theorem~\ref{thm_lower_bound}]
\noindent\textbf{(1) Lower bound: reducing vector processes to total sums.}
From the lower bound on the total queue length (as stated in Lemma~\ref{lem: queuelength_lower_bound}), we have 
\begin{align}\label{ieq:thm1_lower}
 \sum_{i=1}^n Q_i(T)
\geq
&\max_{1\leq k\leq T-1}\sum_{t=k}^{T-1}\bigg(\sum_{i=1}^nA_i(t)-\sum_{i=1}^nD_i(t)\bigg)+\sum_{i=1}^nA_i(T).   
\end{align}
In other words, the total queue length is at least the maximum of cumulative net input (total arrivals minus total departures) plus the final arrivals. Hence, bounding $\mathbb{E}[\sum_{i=1}^n Q_i(T)]$ from below reduces to bounding the expected maximum of the total net-input process.

\noindent\textbf{(2) Hard-instance construction: lower-bounding the net input by a drifted random walk.} Consider a time-invariant scheduling set $\mathcal{D}_\star$ with total capacity $M$, i.e., $\max_{d\in \mathcal{D}_\star}\sum_{i=1}^n d_i=M$, where $M$ depends on the parameter $B$.
We construct i.i.d.\ arrivals $\{A(t)\}_{t\in\mathbb{N}}$ over time, where all components $A_i(t)$ are correlated across $i$, such that the aggregate mean satisfies $\mathbb{E}[X(t)]=\rho M$ and the variance satisfies $\text{Var}(X(t))=n^2C_{\mathrm{a}}^2$, with $X(t)=\sum_{i=1}^n A_i(t)$. 
We then introduce an i.i.d.\ total-departure proxy $\{G(t)\}_{t\in\mathbb{N}}$ that dominates any policy’s total departures, with mean $M$ and variance $n^2C_{\mathrm{d}}^2$.
Replacing true total departures by $G(t)$ yields a valid lower bound in \eqref{ieq:thm1_lower}.
The resulting net-input sequence $\{X(t)-G(t)\}$ is i.i.d., and the problem reduces to bound the expected maximum of a drifted random walk with negative drift $-\theta$, where $\theta=(1-\rho)M$ and variance $\sigma^2=n^2(C_{\mathrm{a}}^2+C_{\mathrm{d}}^2)$.

\noindent\textbf{(3) Probabilistic lower bound via Brownian coupling.}
For a  
 Brownian motion with negative drift $Y_t=\sigma W_t-\theta t$, Lemma~\ref{lem:sharp-lb} gives
\begin{align}
    \mathbb{E}\!\left[\sup_{0\leq s\leq t}Y_s\right]
\gtrsim
\min\left\{\sigma\sqrt{t},\ \sigma^2/\theta\right\}. \label{eqn:Ys_lb}
\end{align}
By the Brownian coupling in Lemma~\ref{lem:KMT}, the partial sum of $k$ i.i.d.\ zero-mean increments with variance $\sigma^2$  can be coupled with $\sigma W_k$ so that their maximum difference in the interval $[0,t]$ is $O(\log t)$ in expectation.
Consequently, as formalized in Lemma~\ref{lemma:bound_negative_drift}, the drifted random walk $S(k)=\sum_{t=1}^{k}(X(t)-G(t))$
inherits the same lower bound scaling as in~\eqref{eqn:Ys_lb} (up to constants). Specifically, for all $t \geq \psi^2$ (a finite constant determined by the characteristics of $X(t)-G(t)$) and for all $\theta$ satisfying $0\le \theta \leq \sigma/\psi$, the expected maximum satisfies
\begin{align*}
\mathbb{E}\!\left[\max_{0\leq k\leq t}{S(k)}\right]\gtrsim \min\left\{\sigma\sqrt{t}, \sigma^2/\theta\right\}. 
\end{align*}

\noindent\textbf{(4) Application to the queueing instance.}
 For the rectangular scheduling set $\mathcal{D}_\star=\{x\in \mathbb{R}^n : 0\le x_i\le B \ \text{for} \ 1\leq i\leq n\}$,
we have $M=nB$.
Applying the above bound to the constructed instance  yields the desired lower bound in~\eqref{ieq:lower_bound_dependent_rho}.
\end{proof}

In summary, the proof leverages a classical minimax approach by constructing hard instances with arrival rates close to their maximum admissible limits and with significant fluctuations in both arrivals and departures. Unlike traditional minimax arguments that focus on static parameter estimation, our challenge is to control the evolution of a stochastic queueing process. We demonstrate that the expected total queue length is lower bounded by the expected maximum of a random walk with i.i.d.\ negative-mean increments. To analyze this, we introduce innovative Brownian coupling techniques, coupling the random walk with a drifted Brownian motion. This coupling introduces a small  logarithmic error, which is negligible compared to the scaling of the Brownian motion.
We surmise that this proof strategy may be broadly applicable to establishing lower bounds for other Lyapunov-optimization problems.

Theorem~\ref{thm_lower_bound} reveals that no policy can guarantee a better scaling than $n\sqrt{(C_{\mathrm{a}}^2+C_{\mathrm{d}}^2)T}$ in the finite-time setting. This establishes a fundamental lower bound, serving as a benchmark for evaluating the performance of any scheduling algorithm. Building on this perspective, the next section introduces a policy framework that attains this bound up to universal constants by jointly optimizing the first- and second-order Lyapunov drift, thereby illustrating how the proposed approach bridges the gap in traditional scheduling methods.

\section{Finite‑Time Performance Guarantees} 
\label{sec:upper}
\subsection{Optimal Lyapunov Policy}
\label{sec:policy}
Building on the finite-time minimax framework introduced earlier, we develop a novel scheduling policy that optimizes the full one-step derandomized Lyapunov drift at each decision time. By ``Lyapunov drift'', we refer to the one-step difference of the Lyapunov function.
Unlike traditional drift-based methods that optimize only the first-order term of this drift, our policy anticipates future queue dynamics and directly controls them through an additional second-order objective, yielding more balanced and stable queue evolution.

Recall that $Q(t) \in \mathbb{R}_+^n$ denotes the queue-length vector at time $t$. The $\LyapOpt$ policy
  selects the schedule $d(t)$ as the solution to the following optimization problem:
\begin{align}\label{lyapopt_policy}
d^{\LyapOpt}(t) \in \argmin_{d \in \mathcal{D}_{t}} \sum_{i=1}^n \left( \max\{ Q_i(t) - d_i, 0 \} \right)^2,
\end{align}
with ties broken arbitrarily.
This expression represents a one-step derandomized lookahead of the quadratic Lyapunov function $V(x) = \|x\|_2^2$, capturing the anticipated impact of the chosen schedule $d \in \mathcal{D}_{t}$. Specifically, based on the queue dynamics in~\eqref{queue_recursion}, the one-step Lyapunov drift is given by

\begin{align}
  \Delta V(t) &=\mathbb{E}\left[V\big(Q(t+1)-A(t+1)\big)-V\big(Q(t)-A(t)\big)\mid \mathcal{F}_t\right]\nonumber\\
    &=\mathbb{E}\left[\sum_{i=1}^n(\max\{Q_i(t)-D_i(t),0\})^2 \middle | \mathcal{F}_t\right]-\sum_{i=1}^n(Q_i(t)-A_i(t))^2.\label{eq:drift}
    \end{align}

Replacing the random departure $D(t)$ with its conditional mean $d(t)$ yields a deterministic surrogate objective in \eqref{lyapopt_policy}. As shown in Appendix~\ref{appendix_thm_upper}, this replacement only incurs an additive variance correction term:
\begin{align}\label{ieq:det_d}
    \mathbb{E}\left[\sum_{i=1}^n(\max\{Q_i(t)-D_i(t),0\})^2\middle| \mathcal{F}_t\right]\leq \mathbb{E}\left[\sum_{i=1}^n(\max\{Q_i(t)-d_i(t),0\})^2\middle| \mathcal{F}_t\right]+\sum_{i=1}^n\text{Var}(D_i(t)).
\end{align}

Before presenting the main performance guarantee of this policy, we first develop a general Lyapunov drift analysis applicable to any scheduling policy. 

In the following, for notational simplicity and without risk of ambiguity, we write $\mathbb{E}^\Phi_{(A, S)}$ as $\mathbb{E}$. Then, for any scheduling sequence, arrival processes and random departure fields in $\mathcal{M}^ {\rho}(B,C_{\mathrm{a}}, C_{\mathrm{d}})$, and any scheduling policy $\Phi$, consider the one-step Lyapunov drift as follows:
\begin{align}
   \Delta V(t) &=\mathbb{E}\left[\sum_{i=1}^n(\max\{Q_i(t)-D_i(t),0\})^2 \middle | \mathcal{F}_t\right]-\sum_{i=1}^n(Q_i(t)-A_i(t))^2\nonumber\\
    &\leq \mathbb{E}\left[\sum_{i=1}^n(\max\{Q_i(t)-d_i(t),0\})^2\middle| \mathcal{F}_t\right]+\sum_{i=1}^n\text{Var}(D_i(t))-\sum_{i=1}^n(Q_i(t)-A_i(t))^2 \nonumber\\
    &\leq \mathbb{E}\left[\sum_{i=1}^n(Q_i(t)-d_i(t))^2\middle| \mathcal{F}_t\right]-\sum_{i=1}^n(Q_i(t)-A_i(t))^2+\sum_{i=1}^n\text{Var}(D_i(t))\nonumber\\
   &=f(Q(t),d(t))+r(Q(t),A(t))+ \sum_{i=1}^n\text{Var}(D_i(t)), \label{eqn:Delta_Vt}
\end{align}
where 
\begin{align}
 f(Q(t),d) &:= \underbrace{\mathbb{E}\left[2 \sum_{i=1}^n Q_i(t)\big(\lambda_i(t) - d_i\big)\middle| \mathcal{F}_t\right]}_{\text{first-order term}}+\underbrace{\sum_{i=1}^n\left(d_i^2 - \lambda_i(t)^2\right)}_{\text{second-order term}}, \label{equ:f_drift+quadratic} 
\end{align}
and
\begin{align}\label{eq:variance_A}
 \mathbb{E}[r(Q(t),A(t))]=\sum_{i=1}^n\text{Var}(A_i(t)).   
\end{align}
 The complete arguments for both results \eqref{eqn:Delta_Vt} and \eqref{eq:variance_A} can be found in Appendix~\ref{appendix_thm_upper}. Taking expectations, summing from  $t = 1$ to $T - 1$ and applying Jensen’s and Cauchy--Schwarz inequalities (see Appendix~\ref{appendix_thm_upper} for the detailed derivation), we obtain:
\begin{align}\label{ieq:upp_bound}
   \mathbb{E}\Bigg[\sum_{i=1}^n Q_i(T) \Bigg]\leq 
   n\sqrt{\sum_{t=1}^{T-1} \frac{1}{n}\ \mathbb{E} [f(Q(t),d(t))] +(T-1)(C_{\mathrm{a}}^2+C_{\mathrm{d}}^2)}+\sum_{i=1}^n \mathbb{E}[A_i(T)].
\end{align}
We now present the performance guarantee of the 
$\LyapOpt$ policy.  
\begin{theorem}[Finite-Time Performance of the $\LyapOpt$ Policy]
\label{thm:lyapopt}
Within the model class $\mathcal{M}^{\rho}(B,C_{\mathrm{a}}, C_{\mathrm{d}})$ for $\rho\in (0,1]$, the $\LyapOpt$ policy achieves the following bound on the expected total queue length: 
 \begin{align}
   \mathbb{E}^{\LyapOpt}_{(A, S)}\Bigg[\sum_{i=1}^n Q_i(T) \Bigg]&\leq 
  n \sqrt{\sum_{t=1}^{T-1} \frac{1}{n}\ \mathbb{E}^{\LyapOpt}_{(A, S)}\Big[\min_{d\in \mathcal{D}_{t}}f(Q(t),d)\Big]  + (C_{\mathrm{a}}^2+C_{\mathrm{d}}^2)(T-1)}\nonumber\\
  &\qquad+\sum_{i=1}^n \mathbb{E}_{(A, S)}[A_i(T)] \label{ieq:upp_bound_lyapopt}
 \end{align}
with $f(Q(t),d)$ given by \eqref{equ:f_drift+quadratic}. Moreover, for any fixed $\rho\in (0,1]$, if the arrival rate vector $\lambda(t)$ lies in the dominated region of the scheduling set $\mathcal{D}_t$ for each $t$, i.e.,
\begin{align}\label{con_starhull}
   \lambda(t) \in \mathrm{dom}(\mathcal{D}_t): = \{ x\in \mathbb{R}^n_+ : \exists\, d \in \mathcal{D}_t \ \text{with} \ x\le d \} 
\end{align}
 for $t\in \mathbb{N}$, then the expected total queue length is bounded as
\begin{align}\label{upp_bound:achieve_low_bound}
 \mathbb{E}^{\LyapOpt}_{(A, S)}\Bigg[\sum_{i=1}^n Q_i(T) \Bigg] \leq 
   n \sqrt{ (C_{\mathrm{a}}^2+C_{\mathrm{d}}^2) (T-1)}+\sum_{i=1}^n \mathbb{E}_{(A, S)}[A_i(T)]. 
\end{align}
\end{theorem}

\begin{proof}
The proof starts from a one-step Lyapunov drift inequality under the $\LyapOpt$ policy. 
Since $\LyapOpt$ minimizes the drift surrogate over $\mathcal{D}_t$, the resulting drift admits the upper bound $\min_{d\in \mathcal{D}_{t}} f(Q(t),d)$ plus the residual fluctuation term $r(Q(t),A(t))+ \sum_{i=1}^n\text{Var}(D_i(t))$ appearing in~\eqref{eqn:Delta_Vt}. 
Summing this bound over time and applying a telescoping argument yields \eqref{ieq:upp_bound_lyapopt}. 
Under the dominated region condition, there exists $\bar d(t)\in \mathcal D_t$ with   $\lambda(t)\le \bar d(t)$. By monotonicity, $\sum_{i=1}^n (\max\{Q_i(t)-\bar d_i(t),0\})^2\leq \sum_{i=1}^n (\max\{Q_i(t)-\lambda_i(t),0\})^2$. Then, by the same argument as the one leading to~\eqref{eqn:Delta_Vt} and using $f(Q(t),\lambda(t))=0$, the optimization term  in \eqref{ieq:upp_bound_lyapopt} vanishes, yielding the cleaner bound in~\eqref{upp_bound:achieve_low_bound}.

Now, we provide the details. By inequality~\eqref{ieq:det_d} and the definition of the $\LyapOpt$ policy in~\eqref{lyapopt_policy}, the one-step Lyapunov drift satisfies
\begin{align}
    &\Delta V^\LyapOpt(t) \nonumber\\
    &\leq \mathbb{E}\left[\sum_{i=1}^n(\max\{Q_i(t)-d_i^\LyapOpt(t),0\})^2\middle| \mathcal{F}_t\right]+\sum_{i=1}^n\text{Var}(D_i^\LyapOpt(t))-\sum_{i=1}^n(Q_i(t)-A_i(t))^2\nonumber \\
    &= \min_{d\in \mathcal{D}_{t}}\bigg(\sum_{i=1}^n(\max\{Q_i(t)-d_i,0\})^2 \bigg) +\sum_{i=1}^n\text{Var}(D_i^\LyapOpt(t))-\sum_{i=1}^n(Q_i(t)-A_i(t))^2.\label{intermediate_term}
\end{align}
Then, by~\eqref{eqn:Delta_Vt}, we obtain
\begin{align}
    \Delta V^\LyapOpt(t) \leq \min_{d\in \mathcal{D}_{t}} f(Q(t),d) + r(Q(t),A(t))+nC_{\mathrm{d}}^2. \label{lyapopt_drift}
\end{align}
Applying the same argument as in the proof of~\eqref{ieq:upp_bound} yields~\eqref{ieq:upp_bound_lyapopt}.

Fix any $\rho\in (0,1]$. If $\lambda(t) \in \text{dom}(\mathcal{D}_t)$, then there exists $\bar d(t) \in \mathcal{D}_t$ with $ \lambda(t)\le\bar d(t)$. Substituting $d=\bar d(t)$ into~\eqref{intermediate_term} gives,  
\begin{align*}
   \Delta V^\LyapOpt(t)&\leq \sum_{i=1}^n (\max\{Q_i(t)-\bar d_i(t),0\})^2  +nC_{\mathrm{d}}^2-\sum_{i=1}^n(Q_i(t)-A_i(t))^2\\
   &\leq \sum_{i=1}^n(\max\{Q_i(t)-\lambda_i(t),0\})^2 +nC_{\mathrm{d}}^2-\sum_{i=1}^n(Q_i(t)-A_i(t))^2\\
   &= f(Q(t), \lambda(t)) + r(Q(t),A(t))+nC_{\mathrm{d}}^2\\
   &=r(Q(t),A(t))+nC_{\mathrm{d}}^2,
\end{align*}
where the last equality follows from $f(Q(t), \lambda(t))=0$.
Taking expectations and applying the same telescoping argument as in the proof of~\eqref{ieq:upp_bound} yields the bound in~\eqref{upp_bound:achieve_low_bound}.
\end{proof}

From \eqref{ieq:upp_bound_lyapopt}, it is evident that the key advantage  of the $\LyapOpt$ policy lies in its simultaneous optimization of the first-  and second-order terms by minimizing the full Lyapunov drift. This joint optimization enables the departure vector to align more precisely with the arrival rate vector, particularly when the scheduling set exhibits geometric irregularities.

As shown in~\eqref{upp_bound:achieve_low_bound}, when the arrival rate lies within the dominated region of the scheduling set, the $\LyapOpt$ policy perfectly aligns with the arrival rate vector (see Eqn.~\eqref{equ:f_drift+quadratic}), regardless of the underlying geometry of the scheduling set. Under this setting, $\LyapOpt$ attains, up to constant factors, the first term in the $\min$ in the fundamental minimax lower bound in \eqref{ieq:lower_bound_dependent_rho}. This demonstrates that in the finite-time, heavy-traffic regime, $\LyapOpt$ is optimal up to constants under the dominated region condition in~\eqref{con_starhull}. In contrast, in the same heavy-traffic regime, $\MaxWeight$ is not optimal in finite-time (see Theorem~\ref{thm:lower_bound_maxweight} in the next section); rather it is optimal in the diffusion-scaled regime as demonstrated in~\cite{stolyar2004maxweight}.

\begin{figure}[t]
\centering

\begin{minipage}[t]{0.24\textwidth}
\centering
\begin{tikzpicture}[scale=0.72, >=Stealth]
\path[use as bounding box] (-0.5,-1.0) rectangle (4.8,4.6);

\draw[thick,->] (-0.3,0) -- (4.5,0) node[right] {$Q_1$};
\draw[thick,->] (0,-0.2) -- (0,4.3) node[above] {$Q_2$};

\coordinate (O) at (0,0);
\coordinate (A) at (0,2.8);
\coordinate (C) at (3.0,0);

\fill[pattern=north east lines, pattern color=orange!85!black]
  (O) -- (A) -- (C) -- cycle;

\draw[thick, orange!85!black] (O) -- (A);
\draw[thick, orange!85!black] (O) -- (C);
\draw[thick, blue] (A) -- (C);

\filldraw[blue] (A) circle (2.3pt);
\filldraw[blue] (C) circle (2.3pt);

\node at (2.35,3.15) {$\Pi(\mathcal{D}_t)=\mathrm{dom}(\mathcal{D}_t)$};

\draw[->, thick]
  (2.25,2.95) .. controls (2.05,2.45) and (1.75,1.95) .. (1.35,1.15);

\end{tikzpicture}
\end{minipage}
\hfill
\begin{minipage}[t]{0.24\textwidth}
\centering
\begin{tikzpicture}[scale=0.72, >=Stealth]
\path[use as bounding box] (-0.5,-1.0) rectangle (4.8,4.6);

\draw[thick,->] (-0.3,0) -- (4.5,0) node[right] {$Q_1$};
\draw[thick,->] (0,-0.2) -- (0,4.3) node[above] {$Q_2$};

\coordinate (O) at (0,0);
\coordinate (A) at (0,2.8);
\coordinate (C) at (3.0,0);

\fill[gray!35] (O) -- (A) -- (C) -- cycle;

\draw[thick] (A) -- (C);
\draw[thick, orange!85!black] (O) -- (A);
\draw[thick, orange!85!black] (O) -- (C);

\filldraw[blue] (A) circle (2.3pt);
\filldraw[blue] (C) circle (2.3pt);

\draw[->, thick]
  (2.55,2.85) .. controls (2.25,2.45) and (1.95,1.95) .. (1.35,1.15);
\node at (2.95,3.15) {$\Pi(\mathcal{D}_t)$};

\draw[->, thick]
  (2.15,-0.60) .. controls (1.55,-0.35) and (1.00,0.00) .. (0.95,0.00);
\draw[->, thick]
  (2.15,-0.60) .. controls (1.45,0.35) and (0.25,1.25) .. (0.00,1.00);
\node at (2.50,-0.85) {$\mathrm{dom}(\mathcal{D}_t)$};

\end{tikzpicture}
\end{minipage}
\hfill
\begin{minipage}[t]{0.24\textwidth}
\centering
\begin{tikzpicture}[scale=0.72, >=Stealth]
\path[use as bounding box] (-0.5,-1.0) rectangle (4.8,4.6);

\draw[thick,->] (-0.3,0) -- (4.5,0) node[right] {$Q_1$};
\draw[thick,->] (0,-0.2) -- (0,4.3) node[above] {$Q_2$};

\coordinate (A) at (0,3.0);
\def\s{2.2}
\coordinate (B) at (\s,\s);
\coordinate (C) at (3.4,0);
\coordinate (D) at (0,0);

\fill[gray!25] (D) -- (A) -- (B) -- (C) -- cycle;

\fill[white] (0,0) rectangle (B);

\fill[pattern=north east lines, pattern color=orange!85!black]
  (0,0) rectangle (B);

\draw[thick, orange!85!black] (D) -- (A);
\draw[thick] (A) -- (B) -- (C);
\draw[thick, orange!85!black] (C) -- (D);

\draw[thick, orange!85!black] (0,0) rectangle (B);

\filldraw[blue] (A) circle (2.3pt);
\filldraw[blue] (B) circle (2.3pt);
\filldraw[blue] (C) circle (2.3pt);

\draw[->, thick]
  (1.95,3.85) .. controls (1.25,3.55) and (1.00,2.65) .. (1.20,1.60);
\node at (2.30,4.05) {$\mathrm{dom}(\mathcal{D}_t)$};

\draw[->, thick]
  (3.10,1.70) .. controls (2.75,1.45) and (2.55,1.15) .. (2.55,0.95);
\node at (3.90,1.70) {$\Pi(\mathcal{D}_t)$};

\end{tikzpicture}
\end{minipage}
\hfill
\begin{minipage}[t]{0.24\textwidth}
\centering
\begin{tikzpicture}[scale=0.72, >=Stealth]
\path[use as bounding box] (-0.5,-1.0) rectangle (4.8,4.6);

\draw[thick,->] (-0.3,0) -- (4.5,0) node[right] {$Q_1$};
\draw[thick,->] (0,-0.2) -- (0,4.3) node[above] {$Q_2$};

\coordinate (O) at (0,0);
\coordinate (A) at (0,2.8);
\coordinate (B) at (3.0,2.8);
\coordinate (C) at (3.0,0);

\fill[pattern=north east lines, pattern color=orange!85!black]
  (O) rectangle (B);

\draw[thick, orange!85!black] (O) rectangle (B);

\filldraw[blue] (A) circle (2.3pt);
\filldraw[blue] (B) circle (2.3pt);
\filldraw[blue] (C) circle (2.3pt);

\node at (2.35,3.15) {$\Pi(\mathcal{D}_t)=\mathrm{dom}(\mathcal{D}_t)$};

\draw[->, thick]
  (2.45,3.00) .. controls (2.20,2.45) and (1.95,1.95) .. (1.55,1.40);

\end{tikzpicture}
\end{minipage}

\caption{Illustrations of the relationship between $\Pi(\mathcal{D}_t)$ and $\mathrm{dom}(\mathcal{D}_t)$. Blue points correspond to $\mathcal{D}_t$, orange edges/regions correspond to $\mathrm{dom}(\mathcal{D}_t)$, and gray regions correspond to $\Pi(\mathcal{D}_t)$. The first subfigure shows every point on the northeast boundary of $\Pi(\mathcal{D}_t)$ contained in $\mathcal{D}_t$, so the dominated region condition in \eqref{con_starhull} holds. In the second subfigure, $\mathcal{D}_t$ has only two points and $\mathrm{dom}(\mathcal{D}_t)$ has measure zero. The final two subfigures show cases of finite $\mathcal{D}_t$ where $\mathrm{dom}(\mathcal{D}_t)$ is respectively a strict subset of, and equal to, $\Pi(\mathcal{D}_t)$.}
\label{fig:Pi-dom}
\end{figure}

If every point on the northeast boundary\footnote{The {\em northeast boundary} of a set $S\subset \mathbb{R}^n$ is  $\partial_{\mathrm{NE}}(S):=\{x\in S: S\cap (x+\mathbb{R}_{++}^n)=\emptyset \}$. Equivalently, $\partial_{\mathrm{NE}}(S)$ consists of the maximal elements of $S$ under the coordinatewise order.} of $\Pi(\mathcal{D}_t)$ is contained in $\mathcal{D}_t$, 
then every arrival rate $\lambda \in  \Pi(\mathcal{D}_t)$ satisfies  the dominated region condition in~\eqref{con_starhull}. This condition is naturally satisfied in models with infinitely divisible service capacity, such as stochastic processing networks that allow processor splitting~\citep{dai2008asymptotic} and some wireless system models
with convex achievable rate regions~\citep{eryilmaz2007fair}. In contrast, when $\mathcal{D}_t$ consists of only finitely many  points, condition \eqref{con_starhull} can be quite restrictive: the subset of arrival rates in $\Pi(\mathcal{D}_t)$ satisfying \eqref{con_starhull} may range from a set of measure zero to the whole of $\Pi(\mathcal{D}_t)$, depending sensitively on the geometry of $\mathcal{D}_t$, as illustrated in Figure~\ref{fig:Pi-dom}. In this regime, our analysis yields an explicit upper bound via the per-slot optimization 
\begin{align*}
    \min_{d\in\mathcal{D}_t} f(Q(t),d)
\ \le \ nB^2
\end{align*}
and furthermore, under the  dominated region condition~\eqref{con_starhull}, the final bound no longer involves this optimization term.
This shows that LyapOpt significantly reduces or even eliminates the dependence on the scheduling set capacity parameter $B$, even though it does not lead to a closed-form expression in terms of $B$,  $C_{\mathrm{a}}$, and  $C_{\mathrm{d}}$. In contrast, the lower bound for MaxWeight in Theorem~\ref{thm:lower_bound_maxweight} indicates that MaxWeight’s dependence on $B$ is unavoidable. The numerical results in Subsections~\ref{subsec:add_exp} and~\ref{subsec: exp_nmodel} further suggest that, even in the absence of the dominated region condition, $\LyapOpt$ outperforms MaxWeight across many practically motivated instances.

\paragraph{Computational complexity of $\LyapOpt$:} Since the objective function in \eqref{lyapopt_policy} is convex quadratic, if the scheduling set $\mathcal{D}_t$ is a (non-empty) polyhedral set, then $\LyapOpt$ is a convex quadratic program and hence solvable in polynomial time. As mentioned above, such scheduling sets arise naturally in models with infinitely divisible service capacity. If $\mathcal{D}_t$ is instead a finite set, then the computational complexity depends primarily on the specific structure of the feasible schedules. For an input-queued switch, the special structure of $\mathcal{D}_t$, in which each departure vector has entries in $\{0,1\}$, implies that  $\LyapOpt$ admits an equivalent reformulation as an optimization problem with a linear objective, so that it has the same computational complexity as MaxWeight; in particular, it is solvable in polynomial time~\citep{tarjan1983data}. In contrast, when $\mathcal{D}_t$ is induced by a conflict graph, $\LyapOpt$ is generally hard to solve and may require a mixed-integer programming solver, such as Gurobi or CPLEX. Moreover, MaxWeight itself is computationally intractable in this setting, as it reduces to a maximum-weight independent set problem, which is NP-hard in general; see, e.g., \citet{tassiulas1998linear}.

\subsection{A Limitation of MaxWeight}
\label{sec:mw-gap}

The MaxWeight policy is widely recognized for its throughput optimality and asymptotic guarantees. Nevertheless, its performance in the finite-time regime is much less understood. In this subsection, we revisit established upper-bound results from the literature~\citep{tassiulas1992stability,stolyar2004maxweight,neely2010stochastic,walton2021learning} and present a complementary lower bound that clarifies the necessity of the factors appearing in this bound, thereby providing a more complete understanding of the policy’s fundamental limitations.

\subsubsection{Upper Bound of MaxWeight}

Consider the MaxWeight policy, which selects the schedule according to
\[
d^{\MaxWeight}(t) \in \argmax_{d \in \mathcal{D}_{t}} \langle Q(t), d \rangle.
\]
By optimizing the first-order Lyapunov drift specified in~\eqref{equ:f_drift+quadratic}, MaxWeight gives priority to queues with larger backlogs. The finite-time upper bound on the total queue length under MaxWeight, within the model class $\mathcal{M}^ {\rho}(B,C_{\mathrm{a}}, C_{\mathrm{d}})$, is given below.

\begin{theorem}[Upper Bound of MaxWeight Policy]\label{thm:upper_bound_maxweight}
 Within the model class $\mathcal{M}^ {\rho}(B,C_{\mathrm{a}}, C_{\mathrm{d}})$ for $\rho\in (0,1]$,  the MaxWeight policy satisfies the following upper bound on the expected total queue length:
\begin{align}\label{ieq:upp_bound_maxweight}
   \mathbb{E}^\MaxWeight_{(A, S)}\left[\sum_{i=1}^{n} Q_i(T)\right]\leq 
   n\sqrt{(B^2+C_{\mathrm{a}}^2+C_{\mathrm{d}}^2)(T-1)}+\sum_{i=1}^n \mathbb{E}_{(A, S)}[A_i(T)].
 \end{align}
\end{theorem}
\begin{proof}
    Recall \eqref{equ:f_drift+quadratic}. The MaxWeight policy minimizes the first-order term of $f(Q(t),d
    )$ and ensures that this term remains non-positive at each step, since $\lambda(t) \in \rho\Pi(\mathcal{D}_t)$ and, by definition of $\Pi(\mathcal{D}_t)$, $d^{\MaxWeight}(t)$ maximizes $\langle Q(t), d \rangle$ over all $d \in \Pi(\mathcal{D}_t)$. Besides,
    the second-order term of $f(Q(t),d
    )$, namely, $\sum_{i=1}^n\left(d_i^2 - \lambda_i(t)^2\right)$, is upper bounded by $n B^2$. 
    Substituting into \eqref{ieq:upp_bound}, we obtain~\eqref{ieq:upp_bound_maxweight}.
\end{proof}
The proof of the upper bound of the expected total queue length of the MaxWeight policy in Theorem \ref{thm:upper_bound_maxweight} illustrates a key limitation of MaxWeight and, more broadly, of drift-based approaches. Specifically, these methods focus solely on optimizing the first-order component of the Lyapunov drift, overlooking the effects of the second-order term. This is particularly significant because MaxWeight consistently chooses extreme points of the scheduling set, which may reduce its sensitivity to the nuanced direction and length of the arrival rate vector as captured by the second-order term. As a result, it may potential lead to  queue accumulation. 

\subsubsection{Lower Bound of MaxWeight for Dimension 2}
The upper bound in Theorem \ref{thm:upper_bound_maxweight} suggests that MaxWeight’s performance depends on the number of queues $n$, the variance parameters $C_{\mathrm{a}}$ and $C_{\mathrm{d}}$, and the capacity parameter $B$. In the heavy-traffic regime, while the lower bound in Theorem~\ref{thm_lower_bound} highlights the roles of 
$n$, $C_{\mathrm{a}}$ and $C_{\mathrm{d}}$, the following lower bound further reveals the inherent dependence of the MaxWeight policy on $B$.

\begin{theorem}[Lower Bound of MaxWeight Policy]\label{thm:lower_bound_maxweight}
Let $n=2$, and consider the family of model classes 
$\{\mathcal{M}^{\rho}(B,C_{\mathrm{a}},C_{\mathrm{d}}) : \rho \in (0,1]\}$.
 Let $c_1>0$ and $\psi_1(n, B, C_{\mathrm{a}}, C_{\mathrm{d}})$ be the constants defined in Theorem~\ref{thm_lower_bound}. Then for all  $T \geq \psi_1(n,B,C_{\mathrm{a}},C_{\mathrm{d}})^2\mathbf{1}_{\{C_{\mathrm{a}}^2+C_{\mathrm{d}}^2>0\}}+9$, the MaxWeight policy satisfies
\begin{align}\label{ieq:maxweight_lowbound}
\liminf_{\rho\uparrow 1} \sup_{(\mathcal{D}, A, S) \in \mathcal{M}^ {\rho}(B,C_{\mathrm{a}}, C_{\mathrm{d}})} \mathbb{E}^{\MaxWeight}_{(A, S)}\left[\sum_{i=1}^{n} Q_i(T)\right]\geq 
    c_1\sqrt{(C_{\mathrm{a}}^2+C_{\mathrm{d}}^2)(T-1)}+\frac{BT^{\frac{1}{3}}}{4\sqrt{3}}.
\end{align}
\end{theorem}
\textit{Proof Roadmap of Theorem~\ref{thm:lower_bound_maxweight}.}
The proof consists of two main steps.
First, we construct a family of deterministic instances $\mathcal{I}_{b,q} := (\mathcal{D}, A, S) \in \mathcal{M}^{1}(B,C_{\mathrm{a}},C_{\mathrm{d}})$ with a time-invariant scheduling sequence $\mathcal{D}$. Specifically, this is done by fixing a rational number $b>1$ and an irrational number $\varepsilon>0$ (introduced to avoid tie-breaking). The time-invariant scheduling set $\mathcal{D}_\star $ and instances are chosen to be 
\begin{align}
\mathcal{D}_\star 
&= \Big\{ x (qb, 0) + (1-x)(0, q): \; 0 \le x \le 1 \Big\}, 
     \nonumber \\
A(1) 
&= q \Big(1, \frac{b-1}{b} - \varepsilon \Big); 
A(t) 
= q \Big(1, \frac{b-1}{b} \Big), 
     t \ge 2, \nonumber \\
S_d(t) 
&= d, 
     d \in \mathcal{D}_\star,
\label{exp_qb}
\end{align}
as illustrated in Figure~\ref{fig:maxweight_lowerbound_q}.   Choosing $q=\frac{\sqrt{2}B}{b}$ ensures that $\mathcal{I}_{b,q}\in \mathcal{M}^{1}(B,C_{\mathrm{a}},C_{\mathrm{d}})$.  We then establish a lower bound  on the expected queue length of $\MaxWeight$ for the family $\{\mathcal{I}_{b,\sqrt{2}B/b}: b\ge 6\}$. Second, we show that the expected total queue length under the $\MaxWeight$ policy is continuous in $\rho $ for all $\rho  \in (\rho_0,1]$ for some $\rho_0<1$, where 
$\rho$ is a scalar that scales the arrival process. By combining this continuity result with the lower bound obtained from the deterministic construction above and with Theorem~\ref{thm_lower_bound}, we obtain the desired finite-time lower bound for $\MaxWeight$.  The complete proof is deferred to Appendix~\ref{appendix_thm_lower_maxweight}. 


\begin{figure}[htbp]
\centering
\setlength{\unitlength}{.4mm}
\begin{picture}(160, 115)
\put(0, 10){\vector(1, 0){150}}
\put(10, 0){\vector(0,1){110}}
\put(150,8){ $Q_1$}
\put(-8, 105){ $Q_2$}
\multiput(10,60)(4,0){8}{\line(1,0){2}}
\put(-1, 75){$q$}
\multiput(40,10)(0,4){13}{\line(0,1){2}}
\put(135, 0){$qb$}
\put(39,1){$q$}
\put(1,1){$0$}
\put(-11, 56){$q\frac{b-1}{b}$}
\put(40, 60){\circle*{4}}
\put(75, 45){$\mathcal{D}_\star$}
\put(43, 62){$A(t)$}
\thicklines
\put(10, 75){\line(2, -1){130}}
  \end{picture}
  \caption{Construction of the instance $\mathcal{I}_{b,q}$ }
  \label{fig:maxweight_lowerbound_q}
\end{figure}

The $BT^{\frac{1}{3}}$ term in Theorem~\ref{thm:lower_bound_maxweight} highlights that, in the heavy-traffic regime ($\rho \uparrow 1$), there exists hard instances, especially those with asymmetric scheduling sets, within $\mathcal{M}^ {\rho}(B,C_{\mathrm{a}}, C_{\mathrm{d}})$ for which the performance of the MaxWeight policy exhibits an intrinsic dependence on the scheduling set's capacity parameter $B$. This establishes a fundamental performance limitation of MaxWeight in the finite-time, heavy-traffic regime.

This result does not contradict MaxWeight's classical optimality in the diffusion-scaled, heavy-traffic regime~\citep{stolyar2004maxweight}. First, our analysis characterizes its performance over a {\em model class} rather than a {\em  specific instance} considered in~\citet{stolyar2004maxweight}. Second, even for a fixed instance, as shown in the following corollary, the dependence on $B$ appears only in the finite-time regime and therefore vanishes in the diffusion-scaled regime~\citep{reiman1984open}.

In order to quantify the finite-time gap between $\LyapOpt$ and $\MaxWeight$, we focus on instances with $C_{\mathrm{a}}=C_{\mathrm{d}}=0$. The following corollary analyzes this gap by comparing the lower bound for MaxWeight (from Theorem \ref{thm:lower_bound_maxweight}) with the performance guarantee of $\LyapOpt$. The proof is provided  in Appendix~\ref{app:coro_gap}.
\begin{corollary}[Performance Gap between $\LyapOpt$ and $\MaxWeight$]\label{coro:gap}
    There exists a family of instances in $\mathcal{M}^1(B,0,0)$ with $B\geq 3\sqrt{2}$, for which the expected total queue lengths under $\LyapOpt$ and  $\MaxWeight$ satisfy: 
\begin{align*}
 \sum_{i=1}^2 Q_i^{\LyapOpt}(T) &\leq 2-\frac{1}{\sqrt{2}B}, \qquad   T\geq 1;\\
   \sum_{i=1}^2 Q_i^{\MaxWeight}(T)& \geq 
     \frac{\sqrt{BT}}{2^{\frac{5}{4}}},  \qquad\bigg\lceil \frac{2B^2}{\sqrt{2}B-1}\bigg\rceil \leq T\leq \bigg\lceil \Big(\frac{\sqrt{2}B}{2}-1\Big)^3\bigg\rceil+1.
\end{align*}
\end{corollary}
\begin{proof}[Proof Sketch]
    We consider the instance $\mathcal{I}_{\sqrt{2}B, 1}$ given in~\eqref{exp_qb}, where $b=\sqrt{2}B$ and $q=1$. 
Since $\lambda(t) = A(t) \in \mathcal{D}_\star$, 
we apply the bound in~\eqref{upp_bound:achieve_low_bound} for $\LyapOpt$, together with the bound for MaxWeight derived in the proof of Theorem~\ref{thm:lower_bound_maxweight}, which is stated in \eqref{lower_maxweight:C=0} in Appendix~\ref{appendix_thm_lower_maxweight}.
\end{proof}

\begin{figure}[htbp]
  \centering
  \includegraphics[width=0.4\textwidth]{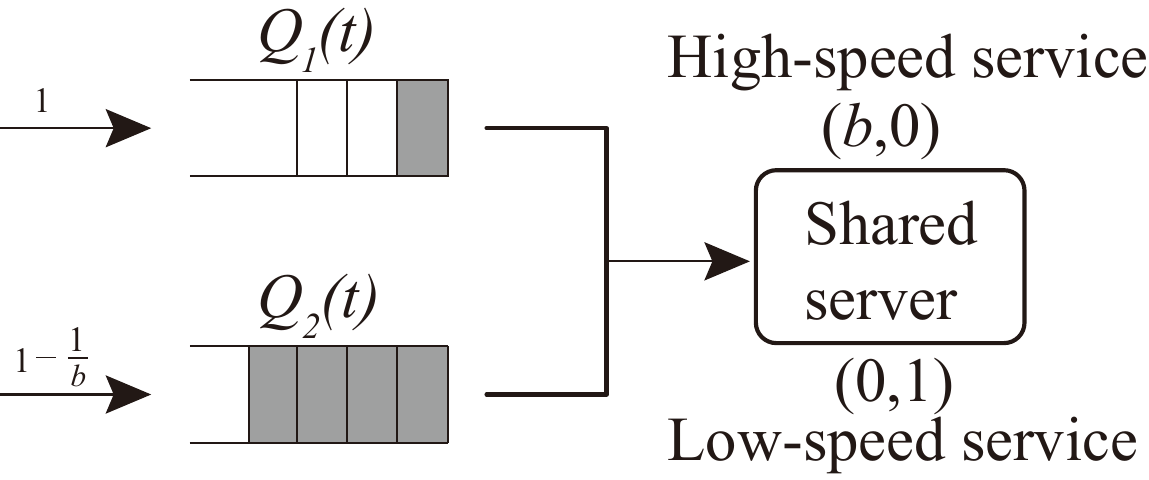}
  \caption{Asymmetric-service shared queueing model
}
  \label{fig:two_queues}
\end{figure}

The model used in the proof of the lower bound in~\eqref{exp_qb} ($q=1$)  and illustrated in Figure~\ref{fig:two_queues} is not contrived.  Instead, it frequently arises in real-world systems such as wireless networks and data centers~\citep{massoulie2002bandwidth, d2023priority}, where two queues---$Q_1(t)$ and $Q_2(t)$---have similar arrival rates $(1,\ 1 - \frac{1}{b})$ with $b \gg 1$, and share a resource whose service shifts between rate $b$ for $Q_1$ and rate $1$ for $Q_2$. MaxWeight tends to over-prioritize $Q_1$ by selecting extreme points, causing backlog in the harder-to-serve $Q_2$ and revealing its finite-time inefficiency. In contrast, $\LyapOpt$ adapts to the arrival rate and avoids backlog in the deterministic setting.

Consequently, the MaxWeight policy, which focuses solely on minimizing the first-order Lyapunov drift, inherently overlooks higher-order terms that are essential for optimal finite-time performance. This oversight results in excessively long queues, especially in scenarios with asymmetric scheduling sets over a finite horizon. Our theoretical findings directly inspire the creation of improved scheduling algorithms that explicitly account for second-order Lyapunov drift terms, leading to better finite-time outcomes.

\section{Stability of LyapOpt in the Interior of the Capacity Region}
\label{sec:stability}
In this section, we demonstrate that $\LyapOpt$  inherits the stability traditionally guaranteed by MaxWeight. Specifically, stability means that the queue length process does not blow up to infinity; it is positive recurrent in the stationary setting and uniformly bounded in expectation over time in the nonstationary setting.

To establish this result, we employ a standard Lyapunov drift analysis. Intuitively, $\LyapOpt$ directly minimizes a derandomized one-step Lyapunov drift, which is therefore no greater than the first-order drift term under MaxWeight, up to additional variance and capacity terms. Consequently, all stability results obtained for MaxWeight through drift analysis also apply to $\LyapOpt$. Formally, under the $\LyapOpt$ policy, within the model class $\mathcal{M}^ {\rho}(B,C_{\mathrm{a}}, C_{\mathrm{d}})$, the one-step Lyapunov drift satisfies the following sequence of inequalities (the detailed proof is provided in Appendix~\ref{app:stability}).
\begin{align*}
    &\mathbb{E}^{\LyapOpt}_{(A,S)}\left[V(Q(t+1))-V(Q(t))\mid Q(t)=x\right]\\
    &\leq \min_{d\in \mathcal{D}_\star} \underbrace{\sum_{i=1}^n\left((\max\{x_i(t)-d_i,0\})^2-x_i(t)^2+2x_i(t)\lambda_i(t+1)+\lambda_i(t+1)^2\right)}_{\text{derandomized one-step Lyapunov drift}}+n(C_{\mathrm{a}}^2+C_{\mathrm{d}}^2)\\
    &\leq \min_{d\in \mathcal{D}_\star} \sum_{i=1}^n\left((x_i(t)-d_i)^2 -x_i(t)^2+2x_i(t)\lambda_i(t+1)+\lambda_i(t+1)^2\right)+n(C_{\mathrm{a}}^2+C_{\mathrm{d}}^2)\\
    &\leq \underbrace{\min_{d\in \mathcal{D}_\star}\sum_{i=1}^n2x_i(t)(\lambda_i(t+1)-d_i)}_{\text{first-order drift term (MaxWeight)}} +2nB^2+n(C_{\mathrm{a}}^2+C_{\mathrm{d}}^2).\\
\end{align*}

We first consider the classical stationary setting, where the scheduling set $\mathcal{D}_t = \mathcal{D}_\star$ for all $t$, 
$A(t)$ is i.i.d.\ across $t$, and $S_d(t)$ is i.i.d.\ across $t$ for each $d \in \mathcal{D}_\star$. In this setting, the induced queue-length process $\{Q(t)\}_{t\in\mathbb{N}_0}$ evolves as a Markov chain. To prove that  $\{Q(t)\}_{t\in\mathbb{N}_0}$ is stable, we apply the classical Foster--Lyapunov criterion (see Lemma \ref{lem:foster-lyapunov} in Appendix~\ref{app:stability}), which is commonly used in analyzing the stability of MaxWeight.

The stability result is stated in the following theorem. To avoid unnecessary complexity, we focus on a countable state space in which all processes take integer values. The general case also holds under mild assumptions using standard arguments. The formal proof is deferred to Appendix~\ref{app:stability}.

\begin{theorem}[Stability of the $\LyapOpt$ Policy in the Stationary Setting]
\label{thm:stability_of_lyapopt}
Suppose that all processes take integer values and are stationary; 
namely, $\mathcal{D}_t = \mathcal{D}_\star$ for all $t$, 
$A(t)$ is i.i.d.\ across $t$, and $S_d(t)$ is i.i.d.\ across $t$ for each $d \in \mathcal{D}_\star$. 
Then, within the model class $\mathcal{M}^{\rho}(B, C_{\mathrm{a}}, C_{\mathrm{d}})$ for $\rho \in (0,1)$, 
the queue length process $\{Q(t)\}_{t \in \mathbb{N}_0}$ under the $\LyapOpt$ policy is stochastically stable, 
in the sense that it reaches the recurrent states with probability~1 and every recurrent state is positive recurrent.
\end{theorem}

This theorem establishes that when the traffic intensity $\rho < 1$, $\LyapOpt$ retains the positive-recurrence property of MaxWeight in the stationary regime. This behavior arises from the inherent nature of $\LyapOpt$, as it minimizes the full one-step Lyapunov drift.

We next establish the stability of $\LyapOpt$ in the nonstationary setting. To this end, we introduce the following assumption, which ensures that for each queue, there exists at least one schedule capable of serving it at a minimum rate of $\alpha$, uniformly over all time slots.
\begin{assumption}[Uniform coordinate coverage] 
\label{assump:ucc}
There exists a constant $\alpha>0$ such that
\begin{equation*}
\label{eq:ucc}
\forall  t\in\mathbb{N}_0, \forall j\in[n], \exists \, d(t)\in \mathcal{D}_t\ \text{such that}\ d_j(t)\geq \alpha.
\end{equation*}
\end{assumption}
With this assumption, we can now establish the stability of $\LyapOpt$ in the nonstationary setting. The proof is presented in Appendix~\ref{app:proof_thm6}.

\begin{proposition}
[Stability of the $\LyapOpt$ Policy in the Nonstationary Setting]
\label{prop:stability_of_lyapopt_adv}
 Under Assumption~\ref{assump:ucc}, within the model class $\mathcal{M}^ {\rho}(B,C_{\mathrm{a}}, C_{\mathrm{d}})$ for $\rho \in (0,1)$,  the queue-length process $ \{Q(t)\}_{t\in\mathbb{N}_0}$ under the $\LyapOpt$ policy is stable, in the sense that for $T\geq 1$,
\begin{align}\label{uni_upper_bound}
    \frac{1}{T}\sum_{t=1}^{T}\mathbb{E}^{\LyapOpt}_{(A, S)}\Bigg[\sum_{i=1}^n Q_i(t) \Bigg]\leq \frac{n^2(B^2 + C_{\mathrm{a}}^2+C_{\mathrm{d}}^2)}{2(1-\rho)\alpha}.
\end{align}
\end{proposition}

Proposition~\ref{prop:stability_of_lyapopt_adv} shows that, in the general nonstationary setting, the $\LyapOpt$ policy ensures that the time-averaged expected queue lengths remain uniformly bounded, a property that also holds for the MaxWeight policy.

Together, Theorem~\ref{thm:stability_of_lyapopt} and Proposition~\ref{prop:stability_of_lyapopt_adv} establish the stability of $\LyapOpt$ in both stationary and nonstationary settings when the traffic intensity $\rho < 1$, using a drift analysis argument similar to that of MaxWeight. Combined with its finite-time minimax optimality under the dominated region condition in~\eqref{con_starhull}, these results demonstrate that $\LyapOpt$ provides a unified framework ensuring both short-term optimality and long-term stability.

\section{Numerical Experiments}
\label{sec:exp}

In this section, we present numerical results that validate the theoretical findings and illustrate the finite-time performance differences between $\LyapOpt$ and MaxWeight across a broad range of settings. 

\subsection{Validation of the Hard Instance} \label{subsec:experiments_gap}
We demonstrate that the $\sqrt{B\, T}$ gap established in  Corollary~\ref{coro:gap}, for the construction $\mathcal{I}_{\sqrt{2}B, 1}$ given in~\eqref{exp_qb},  is evident  for practical values of $T$ and $B$, as shown in Figures~\ref{fig:performance_fixed_b} and~\ref{fig:performance_vs_b}, where $C_{\mathrm{a}}=C_{\mathrm{d}}=0$ and $b=\sqrt{2}B$. However, as shown in Figures~Figures~\ref{fig_app:performance_b10} and~\ref{fig:performance_b20} in Appendix~\ref{sec_appendix: supple_exp}, the total queue length under MaxWeight eventually becomes bounded (i.e., $O(1)$ in $T$) as $T$ increases. 

Moreover, although we do not have a theoretical guarantee for the same construction when the arrival process and random departure field are extended to include variances (i.e., from $C_{\mathrm{a}} = C_{\mathrm{d}} = 0$ to $C_{\mathrm{a}}, C_{\mathrm{d}} > 0$), Figures~\ref{fig_app:performance_B=10_C=1} and~\ref{fig_app:performance_B=10_C=2} indicate that the performance advantage of $\LyapOpt$ over MaxWeight remains substantial.  
\begin{figure}[H]
    \centering
    \begin{subfigure}[t]{0.5\textwidth}
        \includegraphics[width=\linewidth]{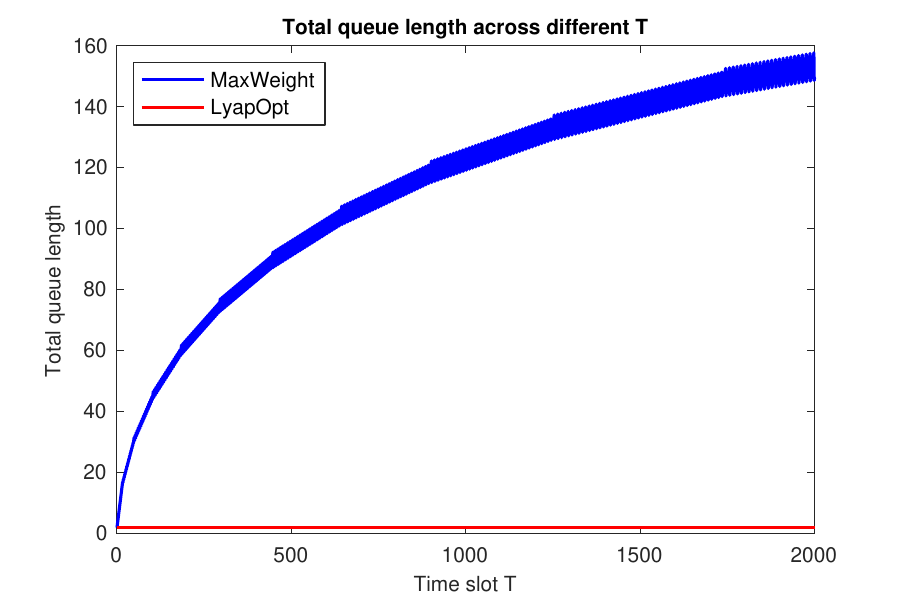}
        \caption{Total queue length when $B=10$, , $C_{\mathrm{a}}=C_{\mathrm{d}}=0$}
        \label{fig:performance_fixed_b}
    \end{subfigure}%
    \hfill
    \begin{subfigure}[t]{0.5\textwidth}
        \includegraphics[width=\linewidth]{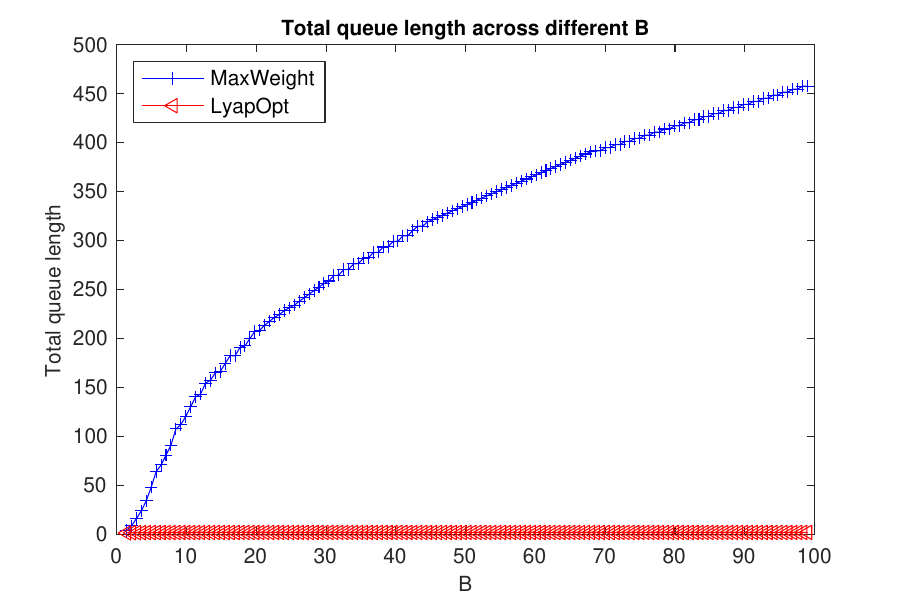}
        \caption{Total queue length when $T=1000$, $C_{\mathrm{a}}=C_{\mathrm{d}}=0$}
        \label{fig:performance_vs_b}
    \end{subfigure}
    \caption{Performance comparison of MaxWeight and LyapOpt policies versus $B$} 
    \label{fig:all_results}
\end{figure}
\begin{figure}[ht]
    \centering
    \begin{subfigure}[t]{0.5\textwidth}
        \includegraphics[width=\linewidth]{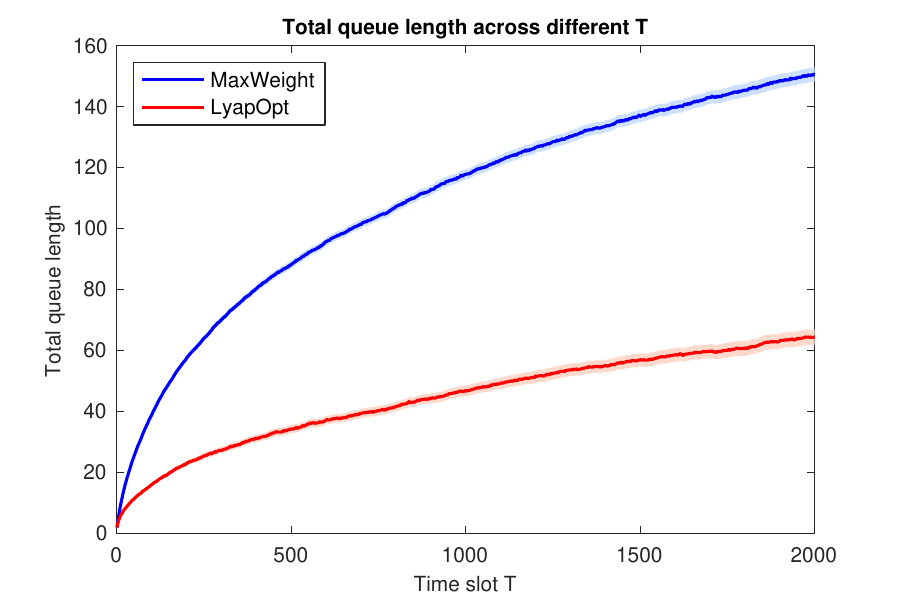}
        \caption{Total queue length when $B=10$, $C_{\mathrm{a}}=C_{\mathrm{d}}=1$}
        \label{fig_app:performance_B=10_C=1}
    \end{subfigure}%
    \hfill
    \begin{subfigure}[t]{0.5\textwidth}
        \includegraphics[width=\linewidth]{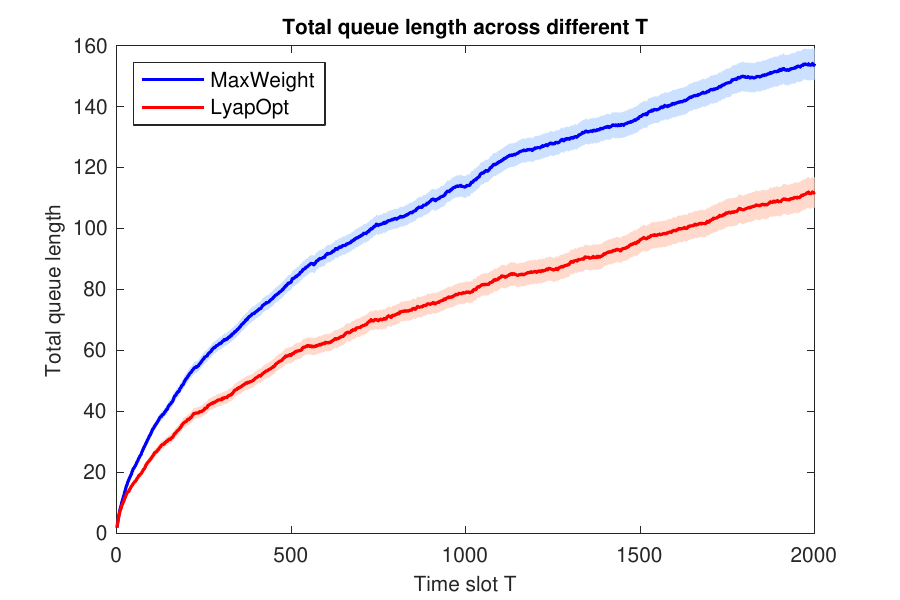}
        \caption{Total queue length when $B=10$, $C_{\mathrm{a}}=C_{\mathrm{d}}=2$}
        \label{fig_app:performance_B=10_C=2}
    \end{subfigure}
    \caption{Performance comparison of MaxWeight and LyapOpt policies when $C_{\mathrm{a}}, C_{\mathrm{d}}>0$} 
    \label{fig_appendix:C>0}
\end{figure}

\subsection{Comprehensive  Experiments}
\label{subsec:add_exp}


We compare the finite-time performance of $\LyapOpt$ and MaxWeight as the number of parallel queues $n$ increases from 2 to 8. For each $n$, the time-invariant scheduling set $\mathcal{D}_\star$ is generated by uniformly sampling $10n$ vectors with integer entries between 1 and 10. Subsequently, 2000 arrival rate vectors are uniformly drawn from the boundary of the capacity region $\rho\Pi(\mathcal{D}_\star)$, each representing a distinct scenario (a total of 2000 scenarios per $n$). In each scenario, arrivals evolve i.i.d.\ over time according to Gamma distributions with fixed variance $1$ per queue, calibrated to match the corresponding sampled arrival rate vector, while departures are modeled as Gamma random variables with fixed variance $1$ per queue, calibrated to match the departure vectors selected by the policy.

Figures~\ref{fig:total3} and~\ref{fig:sq3} show a representative case with $n=8$ and an arrival rate on the boundary of the capacity region $\Pi(\mathcal{D}_\star)$, corresponding to the heavy-traffic regime. The results are averaged over 1000 independent realizations, and the accompanying error bars indicate $95\%$ confidence intervals. Both figures demonstrate that $\LyapOpt$ significantly outperforms MaxWeight. In Figure~\ref{fig:total3}, both policies appear to exhibit $\sqrt{T}$ growth, but $\LyapOpt$ maintains consistently lower total queue length. Figure~\ref{fig:sq3} further shows reduced queue imbalance under $\LyapOpt$, as reflected in the lower sum of squared queue lengths.
\begin{figure}[ht]
\centering
\begin{subfigure}{0.48\linewidth}
    \includegraphics[width=\linewidth]{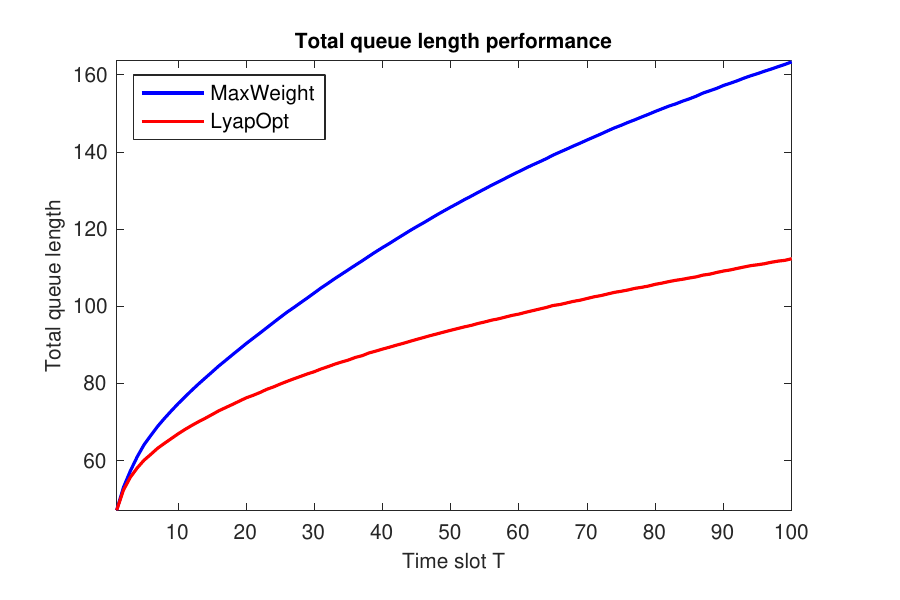}
    \caption{Total queue length}\label{fig:total3}
\end{subfigure}
\hfill
\begin{subfigure}{0.48\linewidth}
    \includegraphics[width=\linewidth]{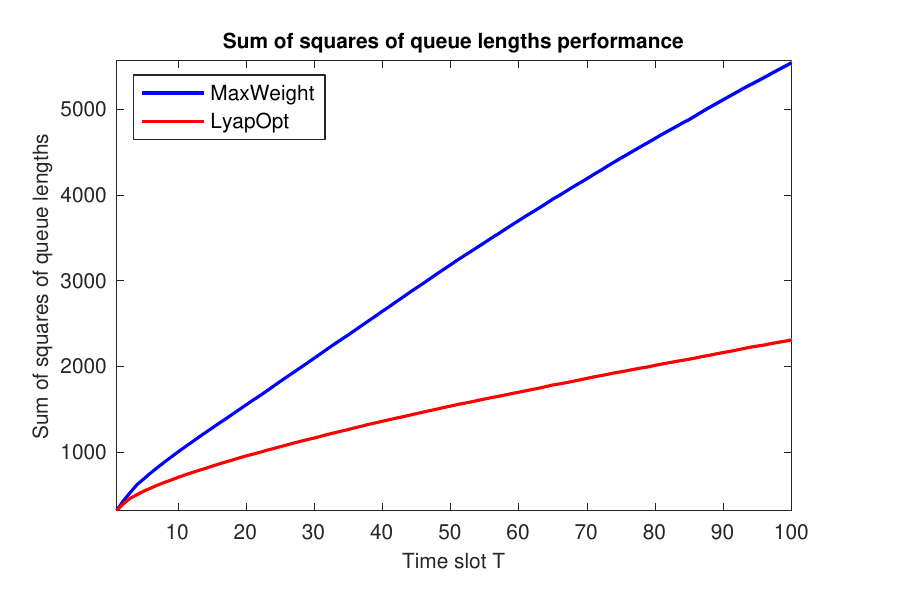}
    \caption{Sum of squares of queue lengths.}\label{fig:sq3}
\end{subfigure}
\caption{Finite-time comparison of MaxWeight and LyapOpt policies ($n = 8$, $\rho=1$).}
\end{figure}

Figures~\ref{fig:total_rho0.9} and~\ref{fig:before_arrivals} illustrate a representative case with $n=8$ and an arrival rate on the boundary of the capacity region $\rho\Pi(\mathcal{D}_\star)$ for $\rho=0.9$, corresponding to the interior regime. Specifically, the vertical axis in Figure~\ref{fig:total_rho0.9} is $\sum_{i=1}^n Q_i(T)$, while in Figure~\ref{fig:before_arrivals} it is $\sum_{i=1}^n(Q_i(T)- A_i(T))$. The results are averaged over 10,000 independent realizations, and the accompanying error bars indicate $95\%$ confidence intervals. Both figures demonstrate that $\LyapOpt$ outperforms MaxWeight and that the system reach stability quickly. In Figure~\ref{fig:before_arrivals}, after removing the common new arrivals, we can see that the queue length under the $\LyapOpt$ policy is noticeably smaller than that under MaxWeight.

\begin{figure}[ht]
\centering
\begin{subfigure}{0.48\linewidth}
    \includegraphics[width=\linewidth]{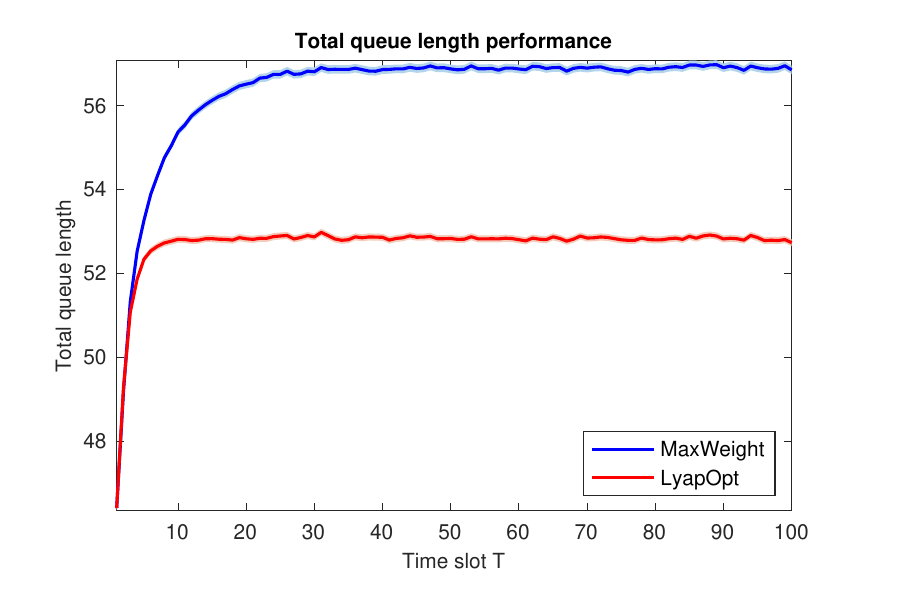}
    \caption{Total queue length}\label{fig:total_rho0.9}
\end{subfigure}
\hfill
\begin{subfigure}{0.48\linewidth}
    \includegraphics[width=\linewidth]{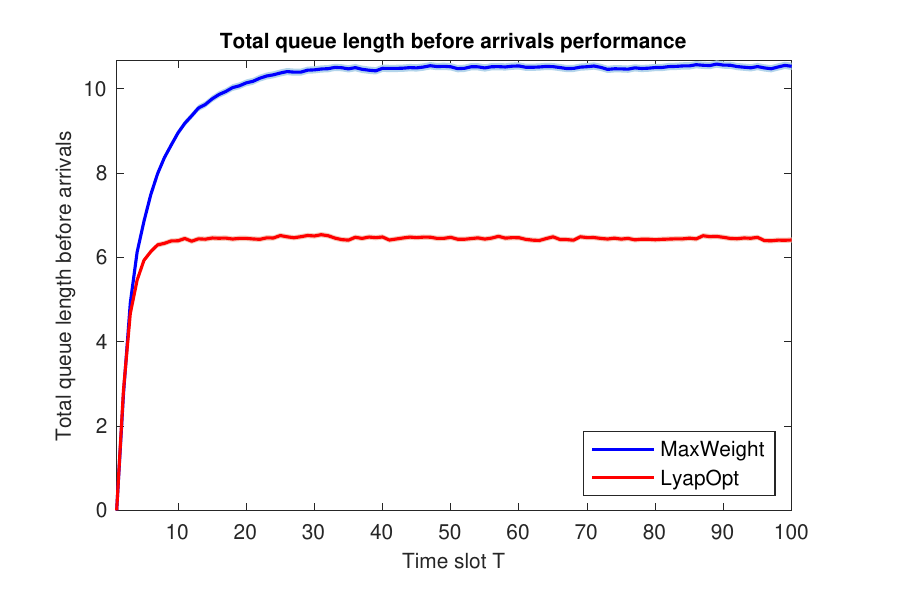}
    \caption{Total queue length before arrivals.}\label{fig:before_arrivals}
\end{subfigure}
\caption{Finite-time comparison of MaxWeight and LyapOpt policies ($n = 8$, $\rho=0.9$).}
\end{figure}

Table~\ref{tab:lyapopt_vs_maxweight} reports, for each number of queues $n$ ($2 \le n \le 8$), the proportion of scenarios where the ratio
\begin{align*}
   \frac{\text{Expected Total Queue Length ($\LyapOpt$) }}{\text{Expected Total Queue Length (MaxWeight)  }}
\end{align*} 
at $t=1000$ falls below 1, 0.9, and 0.8, respectively—indicating cases where $\LyapOpt$ outperforms MaxWeight. The expected total queue length is averaged over 100 independent realizations in each scenario. 
 The accompanying error bars represent $95\%$ confidence intervals based on 2000 independently generated arrival scenarios. The last two columns report the runtimes of $\LyapOpt$ and MaxWeight, measured by the average total computation time across realizations and scenarios. These results show that, for a wide range of queues, $\LyapOpt$ consistently outperforms MaxWeight, often by a substantial margin, while incurring a moderate additional computational cost. 



\begin{table}[h]
\centering
\caption{Proportion of scenarios with ratio below 1, 0.9, and 0.8, and runtime comparison $(\rho=1)$}
\label{tab:lyapopt_vs_maxweight}
\begin{tabular}{cccccc}
\toprule
$n$ & ratio $\leq 1$ & ratio $\leq 0.9$ & ratio $\leq 0.8$ & $\LyapOpt$ (s) & MaxWeight (s) \\
\midrule
2 & 63.30\% \textsubscript{\tiny$\pm$ 2.11\%}   & 9.90\% \textsubscript{\tiny$\pm$ 1.31\%} & 0.10\% \textsubscript{\tiny$\pm$ 0.14\%}&0.086 &0.066\\
3 & 97.30\% \textsubscript{\tiny$\pm$ 0.71\%} & 78.75\% \textsubscript{\tiny$\pm$ 1.79\%}  & 38.30\% \textsubscript{\tiny$\pm$ 2.13\%} &0.090&0.066\\
4 & 99.70\% \textsubscript{\tiny$\pm$ 0.24\%}  & 95.60\% \textsubscript{\tiny$\pm$ 0.90\%}  & 85.85\% \textsubscript{\tiny$\pm$ 1.53\%}&0.094&0.070 \\
5 & 99.35\% \textsubscript{\tiny$\pm$ 0.35\%}  & 92.35\% \textsubscript{\tiny$\pm$ 1.16\%} & 61.25\% \textsubscript{\tiny$\pm$ 2.14\%}&0.103& 0.074\\
6 & 97.40\% \textsubscript{\tiny$\pm$ 0.70\%} & 87.40\% \textsubscript{\tiny$\pm$ 1.45\%} & 72.00\% \textsubscript{\tiny$\pm$ 1.97\%}&0.110& 0.075\\
7 & 100.00\% \textsubscript{\tiny$\pm$ 0.00\%} & 100.00\% \textsubscript{\tiny$\pm$ 0.00\%} & 100.00\% \textsubscript{\tiny$\pm$ 0.00\%}&0.117&0.079 \\
8 & 100.00\%  \textsubscript{\tiny$\pm$ 0.00\%} & 99.50\% \textsubscript{\tiny$\pm$ 0.31\%}  & 92.00\% \textsubscript{\tiny$\pm$ 1.19\%}&0.126 &0.084 \\
\bottomrule
\end{tabular}
\end{table}


\subsection{Experiments on Multi-Server Models}
\label{subsec: exp_nmodel}
In this subsection, we present additional experiments on the N-model and M-model \citep{gurvich2009queue}, two standard benchmarks in the study of multi-server models.  The model parameters are as follows. Table~\ref{tab:service-tables} lists the service rates, where the entry in row $i$ and column $j$ denotes the service rate of server $j$ for queue $i$; a zero entry means that server $j$ cannot serve queue $i$. Each server can serve at most one queue per slot. The arrival rate vector is $\lambda=(2,3.6)^\top$ for the N-model and $\lambda=(1.3,1,3.99)^\top$ for the M-model.
\begin{table}[htbp]
  \centering
  \begin{subtable}{0.48\linewidth}
    \centering
    \caption{N-model}
    \label{table:single}
    \begin{tabular}{@{}lcc@{}}
      \toprule
      & Server 1 & Server 2 \\ \midrule
      Queue 1 & 5 & 0 \\
      Queue 2 & 1 & 3 \\ \bottomrule
    \end{tabular}
  \end{subtable}\hfill
  \begin{subtable}{0.48\linewidth}
    \centering
    \caption{M-model}
    \label{table:multi}
    \begin{tabular}{@{}lccc@{}}
      \toprule
      & Server 1 & Server 2 &Server 3\\ \midrule
      Queue 1 & 1 & 1 & 0\\
      Queue 2 & 0 & 3 & 0\\
      Queue 3 & 0 & 6 & 2\\ \bottomrule
    \end{tabular}
  \end{subtable}
  \caption{Service tables}
  \label{tab:service-tables}
\end{table}

In each scenario, arrivals evolve i.i.d.\ over time according to Gamma distributions with variances fixed at $1$ for each queue and mean calibrated to the corresponding arrival rate vector. Likewise, conditioned on the departure vector selected by the policy at each time slot, the realized departures are generated independently according to Gamma distributions with variance fixed at $1$ and mean matching the chosen departure vector.

For each model, we simulate the queue-length process under both $\LyapOpt$ and MaxWeight over the time horizon $T \leq 500$. Figures~\ref{fig:n_model} and~\ref{fig:M_model} present the results averaged over $1000$ independent sample paths, with error bars corresponding to $95\%$ confidence intervals. In both cases, $\LyapOpt$ achieves smaller queue lengths than MaxWeight.

\begin{figure}[ht]
\centering
\begin{subfigure}{0.48\linewidth}
    \includegraphics[width=\linewidth]{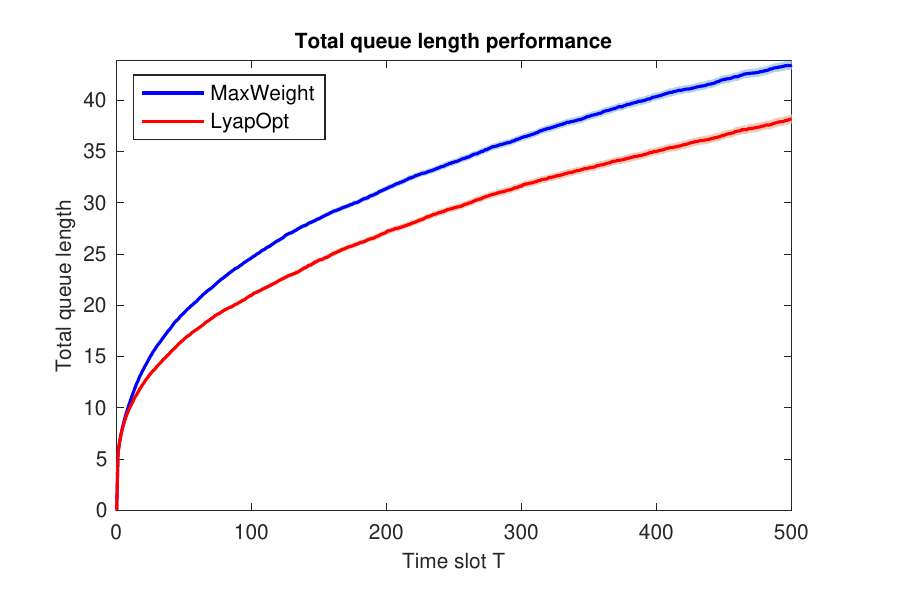}
    \caption{N-model}\label{fig:n_model}
\end{subfigure}
\hfill
\begin{subfigure}{0.48\linewidth}
    \includegraphics[width=\linewidth]{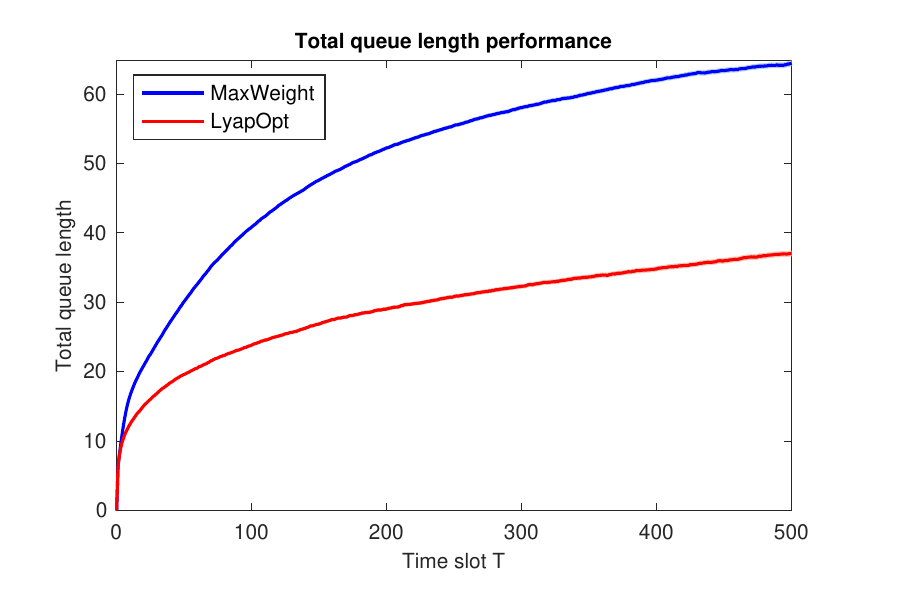}
    \caption{M-model}\label{fig:M_model}
\end{subfigure}
\caption{Finite-time comparison of MaxWeight and LyapOpt policies.}
\end{figure}

\section{Extensions}
\label{sec:extension}

\subsection{Minimax Lower Bound--Independent Case}
\label{subsec:lower-indep}
Some classical queueing models assume independence across queues in both arrivals and departures to facilitate asymptotic analyses~\citep{tassiulas1992stability,harrison1997dynamic,stolyar2004maxweight}. To complement our general analysis, we also present the bound under this assumption. 

We consider the subclass capturing independence across queues: 
\begin{align*}
      \mathcal{M}_{\mathcal{I}}^{\rho}(B,C_{\mathrm{a}},C_{\mathrm{d}}) = \left\{ (\mathcal{D},A,S) \in \mathcal{M}^{\rho}(B,C_{\mathrm{a}},C_{\mathrm{d}})\,: \, \parbox[c]{2.85in}{$\forall\, t\in \mathbb{N},\, \{A_i(t)\}_{i=1}^n \, \text{ indpt across}\; i$ \vspace{0.04 in} \\ $\forall\, t\in \mathbb{N},\, \{S_d(t)_i\}_{i=1}^n \text{ indpt across}\; i, \forall\, d\in \mathcal{D}_t$}
       \right\}.
 \end{align*}
Let $\Gamma^{\rho}_{\mathcal{I}}(B,C_{\mathrm{a}}, C_{\mathrm{d}},T)$ denote the minimax expected total
queue length
\begin{align*}
   \inf_{\Phi} \sup_{(\mathcal{D},A,S) \in \mathcal{M}^ {\rho}_{\mathcal{I}}(B,C_{\mathrm{a}}, C_{\mathrm{d}})} \mathbb{E}^\Phi_{(A, S)}\left[\sum_{i=1}^{n} Q_i(T)\right]. 
\end{align*}
The following proposition characterizes the minimax expected total queue length $\Gamma^{\rho}_{\mathcal{I}}(B,C_{\mathrm{a}}, C_{\mathrm{d}},T)$. The proof is presented in Appendix~\ref{app: thm_lower_bound_indep}.
\begin{proposition}[Minimax Lower Bound--Independent Case]
\label{thm:lower_bound_indep}
For any scheduling policy, and for scheduling sets, arrival processes and random departure fields within the model class $\mathcal{M}_{\mathcal{I}}^{\rho}(B,C_{\mathrm{a}}, C_{\mathrm{d}})$, there exist an absolute constant $c_1>0$ and a constant $\psi_2(n,B,C_{\mathrm{a}},C_{\mathrm{d}})>0$ such that, for
\begin{small}
\begin{align*}
\rho \in
\begin{cases}
\left[\big(1-\tfrac{\sqrt{C_{\mathrm{a}}^2+C_{\mathrm{d}}^2}}{B\psi_2(n,B,C_{\mathrm{a}},C_{\mathrm{d}})}\big)\vee \frac12, 1\right],
& \text{if } C_{\mathrm{a}}^2 + C_{\mathrm{d}}^2 > 0,\\
(0, 1], & \text{if } C_{\mathrm{a}}^2 + C_{\mathrm{d}}^2 = 0,
\end{cases}
\end{align*}
\end{small}
the following lower bound holds for all $T \geq \psi_2(n,B,C_{\mathrm{a}},C_{\mathrm{d}})^2+1$, 
\begin{align}\label{ieq:lower_bound_independent_rho}
  \Gamma^ {\rho}_{\mathcal{I}}(B,C_{\mathrm{a}}, C_{\mathrm{d}},T) \geq 
  c_1\min\left \{\sqrt{n(C_{\mathrm{a}}^2+C_{\mathrm{d}}^2)(T-1)}, \frac{C_{\mathrm{a}}^2+C_{\mathrm{d}}^2}{B(1-\rho)}\right\} +n\rho B.
\end{align}
\end{proposition}

Compared with Theorem~\ref{thm_lower_bound}, the two lower bounds exhibit the same dependence on $T$, $\rho$, $B$, $C_{\mathrm{a}}$, and $C_{\mathrm{d}}$, the only difference lies in their scaling with respect to the number of queues~$n$.


This lower bound relates to the asymptotic results of \citet{stolyar2004maxweight} through its dependence on $n$, $C_{\mathrm{a}}$, and $C_{\mathrm{d}}$, although the two are not directly comparable. 
In \citet{stolyar2004maxweight}, the analysis is performed for individual problem instances. 
By contrast, our results address the worst-case performance over an entire model class, yielding a minimax characterization. 
Moreover, the two works differ in the underlying performance metrics.

\subsection{Other Model Classes}
\label{sec: other model classes}
The previous sections considered the model class $\mathcal{M}^{\rho}(B, C_{\mathrm{a}}, C_{\mathrm{d}})$, in which the scheduling set is allowed to vary over time to capture a more general and challenging problem. In this section, we focus on the time-invariant scheduling set, as is common in the classical queueing literature~\citep{tassiulas1992stability, mandelbaum2004scheduling}. This setting simplifies the model and further highlights that the independence of the lower bound from the scheduling set capacity parameter $B$ in the finite-time regime, together with the dependence of the MaxWeight bound on $B$, are both critical and fundamental characteristics of the system.

Fix the scheduling set $\mathcal{D}_\star$ for all $t$, and consider a reduced model class that does not involve scheduling sets explicitly:
\begin{align*}
     \mathcal{M}_{\mathcal{D}_\star}^{\rho}(C_{\mathrm{a}},C_{\mathrm{d}}) = \left\{ (A,S)\,: \, \parbox[c]{3.3in}{$\lambda(t)\in \rho\Pi(\mathcal{D}_\star)$, $\frac{1}{n}\sum_{i=1}^{n} \mathrm{Var}(A_i(t)) \leq C_{\mathrm{a}}^2, \,\,\,\forall\, t\in \mathbb{N};$ \vspace{0.04 in} \\  $\frac{1}{n}\sum_{i=1}^{n} \mathrm{Var}\!\left((S_d(t))_i\right) \leq C_{\mathrm{d}}^2, 
    \,\,\, \forall\, d\in\mathcal{D}_\star,  t\in \mathbb{N}_0$ }
      \right\}
\end{align*}
where the arrival process $A$, and the random departure field
$S$ are those defined in Section~\ref{sec: problem_setup}. 

Our objective is to characterize the fundamental minimax expected total queue length at time $T$: 
\begin{align*}
   \inf_{\Phi} \sup_{(A,S) \in \mathcal{M}^\rho_{\mathcal{D}_\star}(C_{\mathrm{a}}, C_{\mathrm{d}})} \mathbb{E}^\Phi_{ (A, S)}\left[\sum_{i=1}^{n} Q_i(T)\right],
\end{align*}
which we denote by $\Gamma^ {\rho}_{\mathcal{D}_\star}(C_{\mathrm{a}}, C_{\mathrm{d}},T)$.
The following proposition presents the minimax lower bound for the model class $\mathcal{M}^\rho_{\mathcal{D}_\star}(C_{\mathrm{a}}, C_{\mathrm{d}})$. Fixing $\mathcal{D}_\star$ allows the lower bound to be characterized in terms of the geometric properties of $\mathcal{D}_\star$. In particular, we introduce three parameters, $M_1$, $M_2$, and $M_3$, associated with $\mathcal{D}_\star$, which represent the maximum total capacity, the maximum $\ell_2$ capacity norm, and the minimum $\ell_2$ norm of the vector aligned with the capacity, respectively.
\begin{align*}
&M_1:= \max_{d \in \mathcal{D}_\star} \sum_{i=1}^n d_i;\quad 
M_2 := \max_{d \in \mathcal{D}_\star} \sqrt{\sum_{i=1}^n d_i^2};\\
&M_3:= \min \left\{ \sqrt{\sum_{i=1}^n \gamma_i^2}:
   \gamma \in \Pi(\mathcal{D}_\star), \sum_{i=1}^n \gamma_i = M_1 \right\}.
\end{align*}

\begin{proposition}[Minimax Lower Bound--Fixed Scheduling Set]\label{ext:thm_lower_bound}
Fix a scheduling set $\mathcal{D}_\star$. For any scheduling policy, and for arrival processes and random departure fields in the model class $\mathcal{M}_{\mathcal{D}_\star}^\rho(C_{\mathrm{a}}, C_{\mathrm{d}})$, there exist an absolute constant $c_1>0$ and a finite constant $\psi_3(n,C_{\mathrm{a}},C_{\mathrm{d}},\mathcal{D}_\star)>0$ such that, for
\begin{small}
\begin{align*}
\rho \in
\begin{cases}
\left[\big(1-\tfrac{\sqrt{C_{\mathrm{a}}^2+C_{\mathrm{d}}^2}}{B\psi_3(n,C_{\mathrm{a}},C_{\mathrm{d}},\mathcal{D}_\star)} \big)\vee \frac12, 1 \right],
& \text{if } C_{\mathrm{a}}^2 + C_{\mathrm{d}}^2 > 0,\\
(0, 1], & \text{if } C_{\mathrm{a}}^2 + C_{\mathrm{d}}^2 = 0,
\end{cases}
\end{align*}
\end{small}
the following lower bound holds for all $T \geq \psi_3(n,C_{\mathrm{a}},C_{\mathrm{d}},\mathcal{D}_\star)^2\mathbf{1}_{\{C_{\mathrm{a}}^2+C_{\mathrm{d}}^2>0\}}+1$, 
\begin{align*}
\Gamma^ {\rho}_{\mathcal{D}_\star}(C_{\mathrm{a}}, C_{\mathrm{d}},T)\geq c_1\min\left \{\sqrt{\left(\frac{M_1^2}{M_3^2}C_{\mathrm{a}}^2+\frac{M_1^2}{M_2^2}C_{\mathrm{d}}^2\right)n(T-1)},\ \frac{n}{1-\rho}\left(\frac{M_1}{M_3^2}C_{\mathrm{a}}^2+\frac{M_1}{M_2^2}C_{\mathrm{d}}^2\right)\right\} +\rho M_1.
\end{align*}
\end{proposition}
\begin{proof}
    By repeating the proof of Theorem~\ref{thm_lower_bound}, the desired result follows immediately.
\end{proof}

This proposition reveals that the lower bound for the model class $\mathcal{M}^\rho_{\mathcal{D}_\star}(C_{\mathrm{a}}, C_{\mathrm{d}})$ depends on the geometric discrepancy between different norms associated with the scheduling set $\mathcal{D}_\star$.
To simplify this, we introduce the following variant of the model class:
\begin{align*}
     \widetilde{\mathcal{M}}_{\mathcal{D}_\star}^\rho(C_{\mathrm{a}},C_{\mathrm{d}}) = \left\{ (A,S)\,: \, \parbox[c]{3.5in}{$\lambda(t)\in \rho\Pi(\mathcal{D}_\star)$, $\frac{1}{n}\sum_{i=1}^{n} \sqrt{\mathrm{Var}(A_i(t))} \leq C_{\mathrm{a}}, \,\,\,\forall\, t\in \mathbb{N};$ \vspace{0.04 in} \\  $\frac{1}{n}\sum_{i=1}^{n} \sqrt{\mathrm{Var}\!\left((S_d(t))_i\right)} \leq C_{\mathrm{d}}, 
    \,\,\, \forall\, d\in\mathcal{D}_\star,  t\in \mathbb{N}_0$ }
      \right\}
\end{align*}
where the arrival process $A$, and the random departure field
$S$ are those defined in Section~\ref{sec: problem_setup}.
This variant differs from $\mathcal{M}_{\mathcal{D}_\star}^\rho(C_{\mathrm{a}},C_{\mathrm{d}})$ in that the constraints are placed on average {\em standard deviations} rather than {\em variances}.

Let $\widetilde{\Gamma}^ {\rho}_{\mathcal{D}_\star}(C_{\mathrm{a}}, C_{\mathrm{d}},T)$ denote the minimax expected total queue length at time $T$ under the model class $\widetilde{\mathcal{M}}_{\mathcal{D}_\star}^\rho(C_{\mathrm{a}}, C_{\mathrm{d}})$. The following proposition establishes a lower bound of $\widetilde{\Gamma}^ {\rho}_{\mathcal{D}_\star}(C_{\mathrm{a}}, C_{\mathrm{d}},T)$. Within the model class $\widetilde{\mathcal{M}}_{\mathcal{D}_\star}^\rho(C_{\mathrm{a}},C_{\mathrm{d}})$, the analysis simplifies and yields a cleaner expression, while remaining consistent with the lower bound in Theorem~\ref{thm_lower_bound}.

\begin{proposition}[Minimax Lower Bound--Variant Model Class]
Fix a scheduling set $\mathcal{D}_\star$. For any scheduling policy, and for arrival processes and random departure fields in the model class $\widetilde{\mathcal{M}}_{\mathcal{D}_\star}^\rho(C_{\mathrm{a}}, C_{\mathrm{d}})$,  there exist an absolute constant $c_1>0$ and a finite constant $\psi_4(n,C_{\mathrm{a}},C_{\mathrm{d}},\mathcal{D}_\star)>0$ such that, for
\begin{small}
\begin{align*}
\rho \in
\begin{cases}
\left[\Big(1-\tfrac{\sqrt{C_{\mathrm{a}}^2+C_{\mathrm{d}}^2}}{B\psi_4(n,C_{\mathrm{a}},C_{\mathrm{d}},\mathcal{D}_\star)}\Big)\vee \frac12, 1\right],
&\text{if } C_{\mathrm{a}}^2 + C_{\mathrm{d}}^2 > 0,\\
(0, 1], &\text{if } C_{\mathrm{a}}^2 + C_{\mathrm{d}}^2 = 0,
\end{cases}
\end{align*}
\end{small}
the following lower bound holds for all $T \geq \psi_4(n,C_{\mathrm{a}},C_{\mathrm{d}},\mathcal{D}_\star)^2\mathbf{1}_{\{C_{\mathrm{a}}^2+C_{\mathrm{d}}^2>0\}}+1$, 
\begin{align}\label{ieq:fix_d}
\widetilde{\Gamma}^ {\rho}_{\mathcal{D}_\star}(C_{\mathrm{a}}, C_{\mathrm{d}},T)
\geq c_1\min\left \{n\sqrt{\left(C_{\mathrm{a}}^2+C_{\mathrm{d}}^2\right)(T-1)},\ \frac{n^2(C_{\mathrm{a}}^2+C_{\mathrm{d}}^2)}{M_1(1-\rho)}\right\} +\rho M_1.
\end{align}
\end{proposition}
\begin{proof}
The argument follows the proof of Theorem~\ref{thm_lower_bound}, 
with the tuning parameters in the arrival and departure distributions chosen as
\[
   K_{\mathrm{a}} := \left(\frac{nC_{\mathrm{a}}}{\rho M_1}\right)^2 + 1, 
   \qquad 
   K_{\mathrm{d}} := \left(\frac{nC_{\mathrm{d}}}{M_1}\right)^2 + 1.
\]
\end{proof}

In Corollary~\ref{coro:gap}, we consider an instance with a time-invariant scheduling set and $C_{\mathrm{a}} = C_{\mathrm{d}} = 0$. In this case, the lower bound for MaxWeight still depends on the capacity parameter $B$ within a finite time horizon. This dependence shows that the MaxWeight policy cannot achieve the lower bound given in~\eqref{ieq:fix_d}. To close this gap and attain the lower bound, we introduce the following Empirical Projection (EP) policy:
\begin{align*}
    d^{\mathrm{EP}}(t) \in \argmin_{d\in {\cal D}_\star} \sum_{i=1}^n \max \{ \widehat{A}_i(t)- d_i,0\},
\end{align*}
where $\widehat{A}(t) = \frac{1}{t}\sum_{s=1}^{t} A(s)$.
The EP policy leverages empirical averages for scheduling. 
By projecting the empirical arrival rates onto the feasible scheduling set, it balances service with observed arrivals.
The following theorem establishes its finite-time performance guarantee.

\begin{proposition}[Finite-Time Performance of the $\mathrm{EP}$ Policy]
\label{prop:new_policy}
Assume that $A(t)$ is i.i.d. across $t$ with mean $\mathbb{E}\![A(t)] = \lambda \in \mathcal{D}_\star$. 
Within the model class $\widetilde{\mathcal{M}}_{\mathcal{D}_\star}^\rho(C_{\mathrm{a}}, C_{\mathrm{d}})$ for $\rho \in (0,1]$,  the $\mathrm{EP}$ policy guarantees the following bound on the expected total queue length: for all $T\geq 1$,
  \begin{align}
   \mathbb{E}^{\mathrm{EP}}_{ (A, S)}\bigg[\sum_{i=1}^n Q_i(T) \bigg]\leq n\sqrt{T-1}(6C_{\mathrm{a}}+2C_{\mathrm{d}})+\sum_{i=1}^n \mathbb{E}_{(A, S)}[A_i(T)]. 
 \end{align}
\end{proposition}
The proof is provided in Appendix~\ref{app: other model classes}. 
This result shows that the EP policy achieves the same $O\big(n(C_{\mathrm{a}}+C_{\mathrm{d}})\sqrt{T}\big)$ 
 scaling as the first term in
the min in the lower bound~\eqref{ieq:fix_d}, which is therefore tight within the model class $\widetilde{\mathcal{M}}_{\mathcal{D}_\star}^\rho(C_{\mathrm{a}},C_{\mathrm{d}})$.  

However, minimax optimality of the EP policy (up to universal  factors) in the finite-time regime requires the strong assumption that the arrival process $A(t)$ is i.i.d. across $t$ with mean $\mathbb{E}\![A(t)] = \lambda \in \mathcal{D}_\star$.  When this condition fails, the policy can perform poorly (see Figure~\ref{fig:lambda_in_D}); and even under the assumption, it performs no better than $\LyapOpt$ or $\MaxWeight$ (see Figure~\ref{fig:lambda_notin_D}), under the same experimental setting as in Section~\ref{subsec:add_exp}.
\begin{figure}[H]
    \centering
    \begin{subfigure}[t]{0.5\textwidth}
        \includegraphics[width=\linewidth]{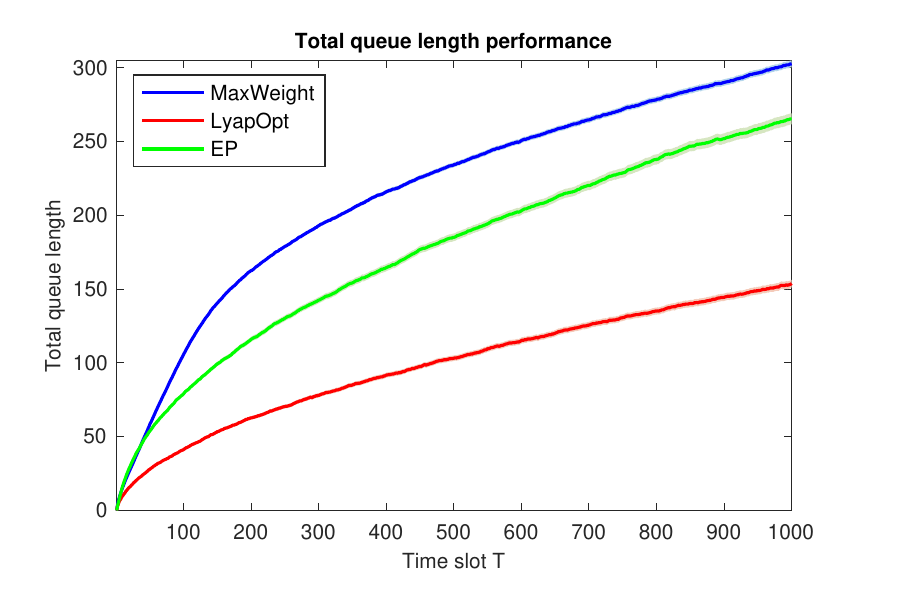}
        \caption{Total queue length when $\lambda \in \mathcal{D}_\star$}
        \label{fig:lambda_in_D}
    \end{subfigure}%
    \hfill
    \begin{subfigure}[t]{0.5\textwidth}
        \includegraphics[width=\linewidth]{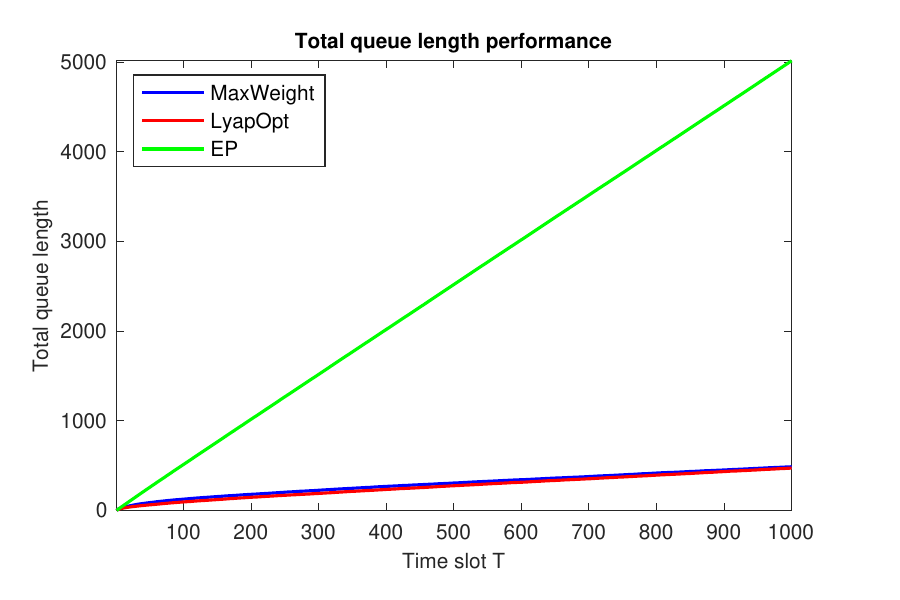}
        \caption{Total queue length when $\lambda \notin \mathcal{D}_\star$}
        \label{fig:lambda_notin_D}
    \end{subfigure}
    \caption{Performance comparison of EP, MaxWeight and LyapOpt policies} 
    \label{fig:all_results2}
\end{figure}

\section{Conclusion and Future Directions}
\label{sec:conclusion}
In this paper, we developed a finite-time minimax framework for scheduling in parallel queueing systems. Within this framework, we identified a fundamental gap between the performance of MaxWeight and the minimax lower bound, with its magnitude depending explicitly on the capacity of the scheduling set. To address this gap, we introduced \(\LyapOpt\), a second-order Lyapunov policy that matches the lower bound under suitable structural conditions while preserving stability guarantees.

Our results suggest several future directions. Because the paper is primarily theoretical and begins with minimax optimality, we focus here on algorithmic and application-oriented directions. Beyond its finite-time guarantees, \(\LyapOpt\) is a surprisingly simple and effective policy that may also offer advantages in classical asymptotic regimes and more general settings. In particular, we conjecture that \(\LyapOpt\) exhibits state-space collapse in heavy traffic, so that the appropriately scaled high-dimensional queue-length process concentrates near a lower-dimensional manifold. This conjecture parallels the corresponding theory for MaxWeight in generalized switches, where state-space collapse is established through fluid-limit and Lyapunov-drift arguments \citep{bramson1998state,stolyar2004maxweight}. Because \(\LyapOpt\) directly minimizes the full one-step Lyapunov drift, it is a natural candidate for such an analysis. Establishing the requisite invariant-manifold behavior could, in turn, provide a path toward diffusion-scale workload optimality.

Another important direction is to extend the design principle underlying \(\LyapOpt\) to more complex stochastic networks. For example, a backpressure-style generalization to multihop networks could be studied to determine whether the generalized version retains \(\LyapOpt\)'s comparative advantages and achieves finite-time optimal dependence on key problem parameters. Such an extension would broaden the applicability of the design while introducing new challenges arising from routing and interference between queues.

A complementary practical direction is to develop efficient implementations of \(\LyapOpt\) for large-scale systems, particularly when the scheduling constraints possess combinatorial structure, such as that induced by bipartite matching graphs \citep{maguluri2016heavytraffic,maguluri2018incompletely}. Such structures have long been central to Internet routing and network resource allocation and are increasingly important in data centers and cloud platforms. More broadly, these directions suggest that finite-time minimax analysis can serve not merely as a theoretical benchmark for scheduling policies, but as a design guideline for stable, scalable, and implementable control algorithms in systems of technological and societal significance.

\section{Code and Data Disclosure}\label{sec:Code and Data Disclosure}The code and data to support the numerical experiments in this paper can be found at \url{https://github.com/lyapopt/lyapopt-queueing-code}.


%
%
%


\bibliography{reference}
\bibliographystyle{informs2014}

\appendix

\section{Technical Appendices and Supplementary Material}

\subsection{Summary of Main Notation}
A table summarizing our notation is provided in Table~\ref{tab:notation-quick-reference-short}.
 \begin{table}[htbp]
\centering
\caption{Summary of Main Notation}
\label{tab:notation-quick-reference-short}
\setlength{\tabcolsep}{6pt}
\renewcommand{\arraystretch}{1.15}
\begin{tabular}{@{}ll@{}}
\toprule
Symbol & Meaning \\
\midrule
$Q(t)$ & Queue length vector at time $t$ \\

$A(t)$ & Arrival vector at time $t$ \\

$\lambda(t)=\mathbb{E}[A(t)]$ & Arrival rate vector at time $t$ \\

$\mathcal{D}_t$ & Scheduling set at time $t$ \\

$d(t)\in\mathcal{D}_t$ & Schedule selected at time $t$ \\

$S_d(t)$ & Random departure vector associated with schedule $d$ at time $t$ \\

$D(t)=S_{d(t)}(t)$ & Realized departure vector at time $t$ \\

$\mathcal{F}_t$ & History available before choosing $d(t)$ \\

$\Phi=\{\phi_t\}_{t\in \mathbb{N}_0}$ & Scheduling policy \\

$\rho$ & Traffic intensity parameter \\

$B$ & Capacity of the scheduling set  \\

$C_{\mathrm a}$, $C_{\mathrm d}$ & Arrival and departure variance parameters \\

$\mathcal{M}^{\rho}(B,C_{\mathrm a},C_{\mathrm d})$ & Model class \\

$\Gamma^{\rho}(B,C_{\mathrm a},C_{\mathrm d},T)$ & Minimax expected total queue length at time $T$ \\
\bottomrule
\end{tabular}
\end{table}

\subsection{Quadratic Objective for Queue Lengths}
\label{app_sec:square_queue_length}
We now show that all the results established for the total queue length objective can be extended to the quadratic objective
\begin{align*}
   \inf_{\Phi}\ \sup_{(\mathcal{D},A,S) \in \mathcal{M}^ {\rho}(B,C_{\mathrm{a}}, C_{\mathrm{d}})}
   \mathbb{E}^\Phi_{(A,S)}\left[\sum_{i=1}^{n} Q_i(T)^2\right].
\end{align*}
For the lower bound, applying the Cauchy–Schwarz and Jensen’s inequalities yields
\begin{align*}
\mathbb{E}^\Phi_{(A,S)}\left[\sum_{i=1}^{n} Q_i(T)^2\right]
&\geq \frac{1}{n}\,\mathbb{E}^\Phi_{(A,S)}\left[\sum_{i=1}^{n} Q_i(T)\right]^2 \geq \frac{1}{n}\left(\mathbb{E}^\Phi_{(A,S)}\left[\sum_{i=1}^{n} Q_i(T)\right]\right)^2.
\end{align*}
Thus, the corresponding lower bound follows directly from the results for the total queue-length objective. 
For the upper bound, both the $\LyapOpt$ and $\MaxWeight$ policies are analyzed using a uniform Lyapunov method with the quadratic Lyapunov function $V(x)=\|x\|_2^2$. 
Hence, the corresponding results are directly obtained for $\sum_{i=1}^n Q_i(T)^2$. 

Consequently, even under the quadratic objective, the $\LyapOpt$ policy achieves the finite-time lower bound up to a constant factor, provided that the conditions of Theorem~\ref{thm:lyapopt} are satisfied, thereby confirming that the optimality of $\LyapOpt$ extends beyond the total queue-length objective.

\subsection{On the Extension of the Pointwise Condition to Window-Based Conditions}\label{app:disc_pointwise_con}
The present results in Sections \ref{sec:lower} and \ref{sec:upper} are proved under the {\em pointwise constraint} $\lambda(t)\in \rho\Pi(\mathcal{D}_t)$ for every $t$. Extending the analysis to the weaker {\em window-based condition} used in prior work (e.g., \citet{yang2023learning}) is not straightforward. For a window-length parameter $\tau>1$, the window-based condition requires that for each $t$ there exists an integer $\omega(t)$ with $1\le \omega(t)\le \tau$ such that
$$\sum_{s=t}^{t+\omega(t)-1}\lambda(s)\in \rho \sum_{s=t}^{t+\omega(t)-1}\Pi(\mathcal{D}_s).$$
 Under this alternative assumption the lower bounds remain valid, but the analysis to obtain an upper bound  becomes problematic.
 
 The proof of the upper bound  relies on the Lyapunov drift argument in \eqref{equ:f_drift+quadratic}, whose first-order term  explicitly depends on the current queue length $Q(t)$. Summing this drift over a window $\{t,t+1,\ldots, t+\omega(t)-1\}$ gives rise to terms of the form 
\begin{align}\label{sum_window}
  \sum_{s=t}^{t+\omega(t)-1}\sum_{i=1}^n Q_i(s)\big(\lambda_i(s) - d_i(s)\big). 
\end{align}
 Unlike several existing papers \citep{yang2023learning,liang2018minimizing}, we do not assume bounded arrivals or departures, and hence, we do not have a uniform bound on $Q(s)$. Moreover, even if such boundedness assumptions were imposed, the resulting bound would likely be too coarse to recover the sharp finite-time guarantees established in \eqref{ieq:upp_bound_lyapopt} and~\eqref{upp_bound:achieve_low_bound}. As a result, because each summand depends on $Q(s)$, the term in \eqref{sum_window} cannot be directly controlled under the window-based condition. Hence, our current proof technique does not readily generalize to this weaker setting. Existing works under such window-based conditions typically establish stability results or derive comparatively coarse finite-time guarantees, rather than the sharper finite-time upper bounds obtained here.

\subsection{Lower Bound Inequalities}
This section presents several auxiliary results supporting the analysis in Section~\ref{sec:lower}.
Among them, Lemma~\ref{lemma:bound_negative_drift} is used directly in the main proof, while the first two lemmas are utilized in its derivation.

We begin with a lower bound for the expected maximum of a drifted Brownian motion.
\begin{lemma}[Lower Bound for the Maximum of a Drifted Brownian Motion]\label{lem:sharp-lb}
Let $Y_t=\sigma W_t-\theta t$ with $\theta>0$, $\sigma>0$, and $W$ a standard Brownian motion.
For $t>0$, define $M_t:=\sup_{0\le s\le t} Y_s$. Then
\begin{equation}\label{eq:sharp-min}
\mathbb{E}[M_t]
 \geq 
\underline{c}\min\left\{\sigma\sqrt{t}, \sigma^2/ \theta \right\},
\qquad
\underline{c}=f_{\mathcal{N}}(1)+2F_{\mathcal{N}}(1)-\tfrac{3}{2}\approx0.425,
\end{equation}
where $f_{\mathcal{N}}$ and $F_{\mathcal{N}}$ denote the standard normal density and cumulative distribution functions, respectively.
\end{lemma}

\begin{proof}
The distribution of the maximum of a Brownian motion with drift (see, e.g., Section~1.9 of~\citet{harrison2013brownian}) is given by 
\[
\mathbb{P}(M_t\leq x)
=F_{\mathcal{N}}\Big(\frac{x+\theta t}{\sigma\sqrt{t}}\Big)
-e^{-\frac{2\theta x}{\sigma^2}}
F_{\mathcal{N}}\Big(\frac{-x+\theta t}{\sigma\sqrt{t}}\Big), \quad x\geq 0.
\]
Integrating the tail probability yields, for $r=\theta\sqrt{t}/\sigma$,
\begin{align*}
\mathbb{E}[M_t]
=\sigma\sqrt{t}A(r)+\frac{\sigma^2}{\theta}B(r),
\ \text{where} \
A(r):=f_{\mathcal{N}}(r)-r(1-F_{\mathcal{N}}(r)),\
B(r):=F_{\mathcal{N}}(r)-1/2.
\end{align*}
Define the functions
\[
g(r):=\frac{\mathbb{E}[M_t]}{\sigma\sqrt{t}}
= A(r)+\frac{B(r)}{r},\qquad 
h(r):=\frac{\mathbb{E}[M_t]}{\sigma^2/\theta}
= r\,A(r)+B(r),\qquad r>0.
\]
We will show that $g(r)$ is decreasing on the interval $(0,1]$ and $h(r)$ is increasing on the interval $[1,\infty)$.

Since $f_{\mathcal{N}}$ decreases on $[0,\infty)$, 
\begin{equation}\label{eq:int-phi}
B(r)=\int_{0}^{r}f_{\mathcal{N}}(u)\mathrm{d}u\geq rf_{\mathcal{N}}(r), \qquad r>0.
\end{equation}
Differentiating $g(r)$ gives
\[
g'(r)=A'(r)+\frac{B'(r)r-B(r)}{r^{2}}
=-(1-F_{\mathcal{N}}(r))+\frac{f_{\mathcal{N}}(r)r-B(r)}{r^{2}},
\]
where the last equality follows from the identity $f_{\mathcal{N}}'(r)=-rf_{\mathcal{N}}(r)$. By \eqref{eq:int-phi}, $f_{\mathcal{N}}(r)r-B(r)\le 0$.
Thus $g(r)$ is decreasing on $(0,1]$, and $\inf_{0<r\le 1} g(r)=g(1)$.

Let the Mills ratio be $m(r):=\frac{1-F_{\mathcal{N}}(r)}{f_{\mathcal{N}}(r)}$ for $r>0$. Then
\[
h'(r)=A(r)+rA'(r)+B'(r)
=2f_{\mathcal{N}}(r)(1-rm(r)).
\]
By the Mills inequality $rm(r)\le 1$, we have $h'(r)\geq 0$ for all $r>0$.
Hence $h(r)$ is increasing on $[1,\infty)$, and $\inf_{r\ge 1} h(r)=h(1)$.
A direct computation gives
\[
g(1)=h(1)=A(1)+B(1)
=f_{\mathcal{N}}(1)+2F_{\mathcal{N}}(1)-\tfrac{3}{2}.
\]
Set
\[
\underline{c}:=g(1)=h(1)=f_{\mathcal{N}}(1)+2F_{\mathcal{N}}(1)-\tfrac{3}{2}\approx 0.425.
\]
For $r\leq 1$, i.e., $t\leq (\sigma/\theta)^2$,
\[\mathbb{E}[M_t]\geq \sigma\sqrt{t}g(r)\geq \underline{c}\sigma\sqrt{t}.\]
For $r\geq 1$, i.e., $t\geq (\sigma/\theta)^2$,
\[\mathbb{E}[M_t]\geq (\sigma^2/\theta)h(r)\geq \underline{c}\sigma^2/\theta.\]
Combining the two regimes yields \eqref{eq:sharp-min}.
\end{proof}

The following lemma follows from the Koml\'os–Major–Tusn\'ady (KMT) coupling theorem, which couples partial sums of i.i.d. random variables with a Brownian motion. For a random variable $Z$, define
\begin{align}\label{def_Sakhanenko}
    \bar{\delta}:=\sup\left\{\delta\geq 0:
\delta\mathbb{E}\left[|Z|^3 e^{\delta |Z|}\right]\leq \text{Var}(Z) \right\}.
\end{align}
If $\bar{\delta}>0$, we say that $Z$ admits a Sakhanenko parameter $\bar{\delta}$.  

\begin{lemma}[KMT Coupling]\label{lem:KMT}
Let $\{Z(k)\}_{k\in \mathbb{N}}$ be a sequence of i.i.d.\ random variables with a non-normal distribution satisfying
\[
\mathbb{E}[Z(1)] = 0, \quad \mathbb{E}[Z(1)^2] = \sigma^2>0,
\]
and suppose that $Z(1)$ admits a Sakhanenko parameter $\bar{\delta}>0$.
 Then there exist i.i.d.\ random variables 
 $\{\widetilde{Z}(k)\}_{k\in \mathbb{N}}$ each with the same distribution as $Z(1)$, and a standard Brownian motion $W=\{W_t\}_{t\geq 0}$ such that the corresponding partial sum $\widetilde{S}(k)=\sum_{i=1}^{k}\widetilde{Z}(i)$ satisfy,
\begin{align}\label{ieq:KMT}
   \mathbb{E}\left[\max_{0\le k\le t}\big|\widetilde{S}(k)- \sigma W_k\big|\right]\le \bar{c}\log (1+t), 
\end{align}
where $\bar{c}$ is a positive constant depending only on $\sigma$ and $\bar{\delta}$. 

Moreover, if for each $\alpha\in \mathcal{A}$, the sequence $Z^{\alpha}=\{Z^{\alpha}(k)\}_{k\in \mathbb{N}}$ satisfies the above conditions with uniformly bounded variance and a uniformly positive lower bound on the Sakhanenko parameter across $\alpha$, then the same constant $\bar{c}$ can be chosen uniformly over $\alpha$.
\end{lemma}
\begin{remark}[Scaling of the KMT constant]
The constant $\bar c$ can be taken to scale on the order
\[
\bar c = O\left(\frac{1}{\bar\delta} + \sigma \right).
\]
\end{remark}

\begin{proof}[Proof of Lemma~\ref{lem:KMT}]
By Theorem~2.4 in~\cite{waudby2025nonsasymptotic}, there exist random variables $\{\widetilde{Z}(k)\}_{k\in \mathbb{N}}$ and i.i.d. standard normal random variables $\{V(k)\}_{k\in \mathbb{N}}$ such that, 
 for any $z>0$ and  $0<\delta<\bar{\delta}$,
\begin{equation}\label{eq:KMTtail}
\mathbb{P}\left(\sup_{k\ge 4}\frac{|\widetilde{S}(k)-\sum_{i=1}^k\sigma V(k)|}{\log k} \geq 4z\right)
\leq 2\sum_{n\geq 1}(1+\delta\sigma2^{2^{n}-1})
\exp\{-c\delta z2^n\log 2\},
\end{equation}
where $c>0$ is an absolute constant.
Define 
\[
\Lambda:=\sup_{k\ge 4}\frac{|\widetilde{S}(k)-\sum_{i=1}^k\sigma V(k)|}{\log k}.
\]
Let $z_0=\tfrac{2}{c\delta}$.
Then
\[
\mathbb{E}[\Lambda]
=4\int_0^\infty \mathbb{P}(\Lambda\geq 4z)\,\mathrm{d}z
\leq 4z_0+4\int_{z_0}^{\infty}\mathbb{P}(\Lambda\geq 4z)\,\mathrm{d}z.
\]
By \eqref{eq:KMTtail} and an application of Tonelli's theorem,
\begin{align*}
4\int_{z_0}^{\infty}\mathbb{P}(\Lambda\ge 4z)\,\mathrm{d}z
&\leq 8\sum_{n\geq 1}\int_{z_0}^{\infty}(1+\delta\sigma2^{2^{n}-1})
e^{-c\delta z2^n\log 2}\,\mathrm{d}z\\
&= \frac{8}{c\delta\log 2}\sum_{n\geq 1}\frac{1}{2^n}
(1+\delta\sigma2^{2^{n}-1})\,e^{-c\delta z_02^n\log 2}.
\end{align*}
With $z_0=\tfrac{2}{c\delta}$, we have $e^{-c\delta z_0\,2^n\log 2}
= e^{-2\cdot 2^n\log 2}=2^{-2^{n+1}}$. Hence,
\begin{align}
4\int_{z_0}^{\infty}\mathbb{P}(\Lambda\ge 4z)\mathrm{d}z
&\leq \frac{8}{c\delta\log 2}\sum_{n\geq 1}
\left(\frac{2^{-2^{n+1}}}{2^n}
+ \frac{\delta\sigma}{2}\cdot \frac{2^{2^{n}}2^{-2^{n+1}}}{2^n}\right)\nonumber\\
&= \frac{8}{c\delta\log 2}\sum_{n\geq 1}\frac{2^{-2^{n+1}}}{2^n}
+ \frac{4\sigma}{c\log 2}\sum_{n\geq 1}\frac{2^{-2^{n}}}{2^n}.
\label{eq:largez}
\end{align}
Both series converge due to the super-exponential decay of $2^{-2^{n}}$. 
Therefore, there exists a
constant $\widetilde{c}>0$ (depending only on $\sigma$ and $\delta$) such that
\begin{equation*}
\mathbb{E}[\Lambda]\leq \widetilde{c}.
\end{equation*}
Since $\log k\le\log t$ for all $4\le k\le t$,
\[
\sup_{4\leq k\leq t} \bigg|\widetilde{S}(k)-\sum_{i=1}^k\sigma V(k)\bigg|
\leq \Lambda\log t.
\]
Taking expectations, we obtain
\[
\mathbb{E}\left[\sup_{4\leq k\leq t} \bigg|\widetilde{S}(k)-\sum_{i=1}^k\sigma V(k)\bigg|\right]\leq \widetilde{c} \log t.
\]
We now control the first three terms. Note that
\[
\mathbb{E}\left[\max_{0\le k\le 3}\bigg|\widetilde{S}(k)-\sum_{i=1}^k\sigma V(k)\bigg|\right]
\le
\mathbb{E}\left[\max_{0\le k\le 3}|\widetilde{S}(k)|\right]
+\mathbb{E}\left[\max_{0\le k\le 3}\bigg|\sum_{i=1}^k\sigma V(i)\bigg|\right].
\]
The partial sums $\{\widetilde{S}_k\}_{k\in \mathbb{N}}$ and    $\{\sum_{i=1}^k V(i)\}_{k\in \mathbb{N}}$  form square-integrable martingales. Applying Jensen’s inequality and Doob’s maximal inequality yields
\begin{align*}
  &\mathbb{E}\left[\max_{0\le k\le 3}|\widetilde{S}(k)|\right]
\leq 
\sqrt{\mathbb{E}\left[\max_{0\le k\le 3} \widetilde{S}(k)^2\right]}\leq \sqrt{4\mathbb{E}[\widetilde{S}(3)^2]}
= 2\sqrt{3}\sigma,\\
&\mathbb{E}\left[\max_{0\le k\le 3}\bigg|\sum_{i=1}^k V(i)\bigg|\right]
\leq 
\sqrt{\mathbb{E}\left[\max_{0\le k\le 3} \bigg(\sum_{i=1}^k V(i)\bigg)^2\right]}\leq \sqrt{4\mathbb{E}\bigg[\sum_{i=1}^3 V(i)^2\bigg]}
= 2\sqrt{3}.
\end{align*}
Summing the two parts gives
\[
\mathbb{E}\left[\max_{0\le k\le 3}\bigg|\widetilde{S}_k-\sum_{i=1}^k\sigma V(i)\bigg|\right]
\le
4\sqrt{3}\sigma.
\]
Thus, there exists a constant $\bar{c}$ depending only on $\sigma$ and $\bar{\delta}$ such that \eqref{ieq:KMT} holds,  using the fact that the sequence $\{V(k)\}_{k\in \mathbb{N}}$ can be embedded into a standard Brownian motion $W=\{W_t\}_{t\geq 0}$. The second part follows directly from the choice of $\bar{c}$.
\end{proof}

Finally, combining Lemmas~\ref{lem:sharp-lb} and~\ref{lem:KMT}, we derive the following lower bound for the maximum of a drifted random walk. This is the key result used in the main proof of Section~\ref{sec:lower}.

\begin{lemma}[Lower Bound for the Maximum of a Drifted Random Walk]\label{lemma:bound_negative_drift}
Let $\{Z(k)\}_{k\in \mathbb{N}}$ satisfy the same conditions as in Lemma~\ref{lem:KMT}.
Define the partial sum $S(t)=\big(\sum_{k=1}^{t}Z(k)\big)-\theta t$ with $\theta\geq 0$.
Then, for all $t \geq \psi^2$ and for all $\theta$ satisfying $ \theta \leq \sigma/\psi$, the expected maximum satisfies
\begin{align}\label{negative_drift_bound}
\mathbb{E}\left[\max_{0\leq k\leq t}{S(k)}\right]\geq 
  \frac{\underline{c}}{2}\min\left\{\sigma\sqrt{t}, \frac{\sigma^2}{\theta}\right\}, 
\end{align}
where $\psi$ is the unique positive solution to $\frac{4(\bar{c}/\sigma+1)}{\underline c}\log(1+x)+\frac{8}{\underline c}= x$; $\underline{c}$  and $\bar{c}$ are the constants defined in Lemmas~\ref{lem:sharp-lb} and~\ref{lem:KMT} respectively.
\end{lemma}
\begin{remark}
The equation $\frac{4(\bar{c}/\sigma+1)}{\underline c}\log(1+x)+\frac{8}{\underline c}= x$ admits a unique positive root.
\end{remark}

\begin{proof}[Proof of Lemma~\ref{lemma:bound_negative_drift}]
Let $X_s:=\sigma W_s-\theta s$
and $M_s:=\sup_{0\le u\le s} X_u$ for $s>0$, where $W$ denotes a standard Brownian motion. In the following, the index $k$ ranges over integers, while $s$ and $u$ vary over real numbers. For an integer $t\geq 1$, note that
\[
\max_{0\le k\le t} S(k)
\geq
\max_{0\le k\le t} X_k-\max_{0\le k\le t}\big|S(k)-X_k\big|
=\max_{0\le k\le t} X_k - \max_{0\le k\le t}\left|\sum_{i=1}^kZ(i)-\sigma W_k\right|.
\]
Taking expectations and using   Lemma~\ref{lem:KMT}  yields
\begin{align*}
\mathbb{E}\left[\max_{0\le k\le t} S(k)\right]
\geq
\mathbb{E}\left[\max_{0\le k\le t} X_k\right] - \bar{c}\log(1+t).
\end{align*}
Moreover, for all $t \geq 1$,
\begin{align}\label{discrete_difference}
  \mathbb{E}\left[\max_{0\le k\le t} X_k\right] \geq  \mathbb{E}\left[\max_{0\le s\le t} X_s\right]-\sigma\left(\sqrt{2\log(4t)}+\frac{1}{\sqrt{2\log(4t)}}\right),
\end{align}
whose proof will be completed later. 
Since $\sqrt{2y} \leq y + 1$ for all $y \ge 0$, we obtain
\begin{align*}
    \mathbb{E}\left[\max_{0\le k\le t} S(k)\right]\geq\mathbb{E}[M_t] - ((\bar{c}+\sigma)\log(1+t)+4\sigma).
\end{align*}
We first consider $\theta>0$. Let $u:=t\wedge(\sigma/\theta)^2$. Then
\begin{equation}\label{eq:additive}
\mathbb{E}\left[\max_{0\le k\le t} S(k)\right]
\geq \mathbb{E}\left[\max_{0\le k\le u} S(k)\right]\geq
\mathbb{E}[M_u] - ((\bar{c}+\sigma)\log(1+u)+4\sigma),
\quad u=t\wedge(\sigma/\theta)^2.
\end{equation}
By Lemma~\ref{lem:sharp-lb}, for every $u>0$,
\begin{equation}\label{eq:bm-min}
\mathbb{E}[M_u]\geq\underline{c}\,\min\left\{\sigma\sqrt{u},\frac{\sigma^2}{\theta}\right\}=\underline{c}\min\left\{\sigma\sqrt{t},\frac{\sigma^2}{\theta}\right\}.
\end{equation}
Combining \eqref{eq:additive} and \eqref{eq:bm-min} gives the additive-error inequality
\begin{align*}
\mathbb{E}\left[\max_{0\le k\le t} S(k)\right]
\geq
\underline{c}\min\left\{\sigma\sqrt{t},\frac{\sigma^2}{\theta}\right\}-\bigg((\bar{c}+\sigma)\log\Big(1+t\wedge\Big(\frac{\sigma}{\theta}\Big)^2\Big)+4\sigma\bigg).
\end{align*}
To obtain \eqref{negative_drift_bound}, it suffices to ensure
\[
(\bar{c}/\sigma+1)\log\Big(1+t\wedge(\frac{\sigma}{\theta})^2\Big)+4\leq \frac{\underline c}{2}\min\left\{\sqrt{t},\frac{\sigma}{\theta}\right\}.
\]
 By the definition of $\psi$, for $t \geq \psi^2$, we have 
\begin{align}\label{sqrt_ieq}
\frac{4(\bar{c}/\sigma+1)}{\underline c}\log(1+\sqrt{t})+\frac{8}{\underline c}\leq \sqrt{t}.  
\end{align}
Hence,
\[
(\bar{c}/\sigma+1)\log\Big(1+t\wedge(\tfrac{\sigma}{\theta})^2\Big)+4\leq (\bar{c}/\sigma+1)\log(1+t)+4 \leq 2(\bar{c}/\sigma+1)\log(1+\sqrt{t})+4\leq \frac{\underline c}{2}\sqrt{t}.
\]
where the second inequality follows from $(1+\sqrt t)^2\geq 1+t$.

Let $r:=\sigma/\theta$. Under the assumption on $\theta$,   we have $r\geq \psi$, and thus
\[
\frac{4(\bar{c}/\sigma+1)}{\underline c}\log(1+r)+\frac{8}{\underline c}\leq r.
\]
Then
\[
(\bar{c}/\sigma+1)\log\Big(1+t\wedge(\tfrac{\sigma}{\theta})^2\Big)+4\leq (\bar{c}/\sigma+1)\log(1+r^2)+4 \leq 2(\bar{c}/\sigma+1)\log(1+r)+4\leq \frac{\underline c}{2}r.
\]
Combining the two cases completes the proof for case $\theta>0$.
When $\theta=0$, 
\begin{align*}
    \mathbb{E}[M_t] =\sigma\sqrt{2t/\pi}\geq \sigma\underline{c}\sqrt{t},
\end{align*}
and, by applying~\eqref{sqrt_ieq}, the proof proceeds in the same manner.

Now we prove \eqref{discrete_difference}. Let $\widetilde{X}_s:=X_k$ and $\widetilde{W}_s:=W_k$
 for $k\leq s<k+1$. Note that
\begin{align*}
\max_{0\le k\le t} \widetilde{X}_k=\max_{0\le s\le t} \widetilde{X}_s
\geq
\max_{0\le s\le t} X_s-\max_{0\le s\le t}\big|X_s-\widetilde{X}_s\big|. 
\end{align*}
It therefore suffices to show that
\begin{align*}
\max_{0\leq s\leq t}\big|W_s-\widetilde{W}_s\big| \leq\sqrt{2\log(4t)}+\frac{1}{\sqrt{2\log(4t)}}.
\end{align*}
 On each interval $[k, k+1)$, write $s = k + u$ with $u \in [0,1)$. Let $t\geq 1$ be an integer.
By the stationary and independent increments of Brownian motion,
\[
\sup_{0 \leq s\leq t} |\widetilde{W}_s- W_s|
= \max_{0 \leq k \leq t-1}\sup_{0 \leq u \leq 1}|W_{k+u} - W_k|
\stackrel{d}{=}\max_{1 \leq j \leq t}Y_j,
\]
where $Y_1, \dots, Y_t$ are i.i.d.\ copies of $Y :=\sup_{0 \leq u \le 1} |W_u|.$
By the reflection principle and the standard Gaussian tail bound $1-F_{\mathcal{N}}(x)\le e^{-x^2/2}$  which holds for all $x\ge0$, we have
\[
\mathbb{P}(Y>x) 
= \mathbb{P}\left(\sup_{0 \leq u\leq 1} |W_u|>x\right)
\leq 2\mathbb{P}\left(\sup_{0\leq u\leq 1} W_u>x\right)
= 4(1-F_{\mathcal{N}}(x))\leq 4e^{-x^2/2}, \quad x \geq 0.
\]
Let $ \widetilde{Y}_t:= \max_{1 \leq j \leq t} Y_j$. 
For any $a>0$, we decompose
\[
\mathbb{E}[\widetilde{Y}_t]
=\int_{0}^{\infty}\mathbb{P}(\widetilde{Y}_t\geq x)\mathrm{d}x
\leq a+\int_{a}^{\infty}\mathbb{P}(\widetilde{Y}_t\geq x)\mathrm{d}x. 
\]
Since
\[
\mathbb{P}(\widetilde{Y}_t\ge x)
\leq \sum_{k=1}^t \mathbb{P}(Y_k\geq x)
= t\mathbb{P}(Y\geq x)
\leq 4te^{-x^2/2},
\]
we have
\[
\mathbb{E}[\widetilde{Y}_t]
\le a+4t\int_{a}^{\infty} e^{-x^2/2}\mathrm{d}x\leq a+\frac{4t}{a}e^{-a^2/2},
\]
where the last inequality follows from integration by parts.
Choosing $a:=\sqrt{2\log(4t)}$ yields
\[
\mathbb{E}[\widetilde{Y}_t] \leq \sqrt{2\log(4t)}+\frac{1}{\sqrt{2\log(4t)}}.
\]
Therefore, we complete the proof.\end{proof}

\subsection{Proofs of Theorems and Supporting Results for Section~\ref{sec:lower}}\label{appendix_thm_lower_bound}
\subsubsection{Auxiliary Lemma}\label{appendix_lemmas}
\begin{lemma}[Lower Bound on Total Queue Length]\label{lem: queuelength_lower_bound}
    Given the arrival process $\{A(t)\}_{t\in \mathbb{N}}$ and the departure process $\{D(t)\}_{t\in \mathbb{N}_0}$, the total queue length at time $T\geq 1$ satisfies
    \begin{align}
    \sum_{i=1}^n Q_i(T) & \geq\max\bigg\{\sum_{t=1}^{T-1}\Big(\sum_{i=1}^nA_i(t)-\sum_{i=1}^nD_i(t)\Big),\sum_{t=2}^{T-1}\Big(\sum_{i=1}^nA_i(t)-\sum_{i=1}^nD_i(t)\Big),\dots,\nonumber\\
    & \quad \quad \quad \quad \sum_{i=1}^nA_i(T-1)-\sum_{i=1}^nD_i(T-1),0\bigg\}+\sum_{i=1}^nA_i(T).\label{appen:lower_queue1}
    \end{align}
\end{lemma}
\begin{proof}
Let $\Delta(t)=A(t)-D(t)$ for all $t\geq 1$.  By \eqref{queue_recursion}, for $1\leq i\leq n$ and $T\geq 1$, we have
\begin{align*}
    &Q_i(T)-A_i(T)\nonumber\\
    &=\max\{Q_i(T-1)-D_i(T-1),0\} \nonumber \\
          &=\max\{\max\{Q_i(T-2)-D_i(T-2),0\}+A_i(T-1)-D_i(T-1),0\} \nonumber \\
          &=\max\{Q_i(T-2)-D_i(T-2)+\Delta_i(T-1),\Delta_i(T-1),0\}\nonumber \\
          &=\max\bigg\{Q_i(1)-D_i(1)+\sum_{t=2}^{T-1}\Delta_i(t),\sum_{t=2}^{T-1}\Delta_i(t),\dots,\Delta_i(T-1),0\bigg\}\nonumber \\
          &=\max\bigg\{\sum_{t=1}^{T-1}\Delta_i(t),\sum_{t=2}^{T-1}\Delta_i(t),\dots,\Delta_i(T-1),0\bigg\},
    \end{align*}
where the last equation follows that $Q_i(0)=0$ and then $Q_i(1)=A_i(1)$. Therefore, we have
\begin{align*}
    \sum_{i=1}^n Q_i(T)&=\sum_{i=1}^n\max\bigg\{\sum_{t=1}^{T-1}\Delta_i(t),\sum_{t=2}^{T-1}\Delta_i(t),\dots,\Delta_i(T-1),0\bigg\}+\sum_{i=1}^nA_i(T)\\
    &\geq \max\bigg\{\sum_{t=1}^{T-1}\sum_{i=1}^n\Delta_i(t),\sum_{t=2}^{T-1}\sum_{i=1}^n\Delta_i(t),\dots,\sum_{i=1}^n\Delta_i(T-1),0\bigg\}+\sum_{i=1}^nA_i(T),
    \end{align*}
    which completes the proof.
\end{proof}

\subsubsection{Proof of Theorem \ref{thm_lower_bound}}\label{appendix_thm_lower_bound_general}

\begin{proof}[Proof of Theorem \ref{thm_lower_bound}.]
Fix a time-invariant scheduling sequence $\mathcal{D}$ with $\mathcal{D}_t=\mathcal{D}_\star$ for all $t\in \mathbb{N}_0$, and assume that $\mathcal{D}_\star$ satisfies the condition in the model class $\mathcal{M}^ {\rho}(B,C_{\mathrm{a}}, C_{\mathrm{d}})$.
Let 
$$M=\max\left\{\sum_{i=1}^n d_i:d\in \mathcal{D}_\star\right\}$$
denote the maximum total service capacity in a single time slot.
In the following, we derive a lower bound for the right-hand-side of inequality~\eqref{appen:lower_queue1}  by constructing a family of hard instances $(A,S)$. 
Here the arrival process $A$ and the random departure field $S$ are constructed independently. We focus on the case $C_{\mathrm{a}}^2 + C_{\mathrm{d}}^2 > 0$; as the case $C_{\mathrm{a}}^2 + C_{\mathrm{d}}^2 = 0$ is straightforward.

 {\bf Arrival Process}: We construct an i.i.d.\ sequence $\{A(t)\}_{t\in \mathbb{N}}$ over time. At time $t$, 
\begin{align*}
    &\mathbb{P}(A_1(t)=K_{\mathrm{a}}\lambda_1, A_2(t)=K_{\mathrm{a}}\lambda_2,\dots,A_n(t)=K_{\mathrm{a}}\lambda_n)=1/K_{\mathrm{a}},\\
    &\mathbb{P}(A_1(t)=0, A_2(t)=0,\dots,A_n(t)=0)=(K_{\mathrm{a}}-1)/K_{\mathrm{a}},
\end{align*}
where 
\begin{align} \label{choice_lambda_Ka}
     \frac{\lambda}{\rho}\in \argmin_{\substack{
    \gamma \in \Pi(\mathcal{D}_\star) \\
    \sum_{i=1}^n \gamma_i = M
}} \left\{  \sum_{i=1}^n \gamma_i^2 \right\}\quad \mbox{and} \quad K_{\mathrm{a}}=\frac{nC_{\mathrm{a}}^2}{\sum_{i=1}^n\lambda_i^2}+1.
\end{align}
Note that
\begin{align}\label{eq:mean_variance_a}
   \mathbb{E}[A(t)]=\lambda, \quad \text{and} \quad \text{Var}(A_i(t))=\frac{(K_{\mathrm{a}}-1)^2\lambda_i^2+\lambda_i^2(K_{\mathrm{a}}-1)}{K_{\mathrm{a}}}=(K_{\mathrm{a}}-1)\lambda_i^2.
\end{align}
The choice of $K_{\mathrm{a}}$ ensures that
  $\frac{1}{n}\sum_{i=1}^n\text{Var}(A_i(t))= C_{\mathrm{a}}^2$. 

{\bf Random Departure Field}:  We construct an independent sequence $\{S_t\}_{t\in \mathbb{N}_0}$ over time. At time $t$, for each $d \in \mathcal{D}_t$,
    \begin{align*}
    \mathbb{P}(S_d(t)=K_{\mathrm{d}} d)=1/K_{\mathrm{d}},\ \text{and} \ \mathbb{P}(S_d(t)=0)=(K_{\mathrm{d}}-1)/K_{\mathrm{d}},
\end{align*}
where
\begin{align}\label{choice_Kd}
       K_{\mathrm{d}}=\frac{nC_{\mathrm{d}}^2}{\max_{d\in \mathcal{D}_\star}\sum_{i=1}^n d_i^2}+1.
\end{align}
If the scheduling policy selects $d(t)\in \mathcal{D}_t$, the actual departure $D(t)$ is distributed as
\begin{align*}
    &\mathbb{P}(D_1(t)=K_{\mathrm{d}} d_1(t), D_2(t)=K_{\mathrm{d}} d_2(t),\dots,D_n(t)=K_{\mathrm{d}} d_n(t))=1/K_{\mathrm{d}},\\
    &\mathbb{P}(D_1(t)=0, D_2(t)=0,\dots,D_n(t)=0)=(K_{\mathrm{d}}-1)/K_{\mathrm{d}}.
\end{align*}
Similarly, the choice of $K_{\mathrm{d}}$ ensures $S$ satisfies the conditions of the model class $\mathcal{M}^ {\rho}(B, C_{\mathrm{a}}, C_{\mathrm{d}})$. 

The departures satisfy $\mathbb{E}[D(t)]=d(t)$, and $\{\sum_{i=1}^nD_i(t)\}_{t \in \mathbb{N}_0}$ forms a sequence of binary random variables:
\begin{align*}
    \mathbb{P}\left(\sum_{i=1}^nD_i(t)=K_{\mathrm{d}} \sum_{i=1}^nd_i(t)\right)=1/K_{\mathrm{d}},\ \text{and} \ 
    \mathbb{P}\left(\sum_{i=1}^nD_i(t)=0\right)=(K_{\mathrm{d}}-1)/K_{\mathrm{d}}.
\end{align*}
Let $\{G(t)\}_{t\in \mathbb{N}}$ be an i.i.d. sequence of random variables defined by
\begin{align*}
    G(t)=\frac{M}{\sum_{i=1}^n d_i(t)}\sum_{i=1}^n D_i(t).
\end{align*}
 Since $\sum_{i=1}^nd_i(t) \leq M$ for any scheduling policy, then $G(t)\geq \sum_{i=1}^n D_i(t)$. Let $X(t)=\sum_{i=1}^n A_i(t)$. By \eqref{appen:lower_queue1}, the queue length satisfies
\begin{align*}
    \sum_{i=1}^n Q_i(T) & \geq\max\left\{\sum_{t=1}^{T-1}\Big(X(t)-G(t)\Big),\sum_{t=2}^{T-1}\Big(X(t)-G(t)\Big),\dots, X(T-1)-G(T-1),0\right\}+X(T).
    \end{align*}
 Observe that
 \[
 X(t)-G(t)=(X(t)-\rho M)+(M-G(t))-(1-\rho)M.
 \]
Note that $\{(X(t)-\rho M)-(M-G(t))\}_{t \in \mathbb{N}}$ forms an i.i.d.\ sequence of random variables, satisfying
\begin{align*}
\mathbb{E}[(X(t)-\rho M)-(M-G(t))]&=0, \\
\text{Var}((X(t)-\rho M)+(M-G(t)))&=M^2(\rho^2(K_{\mathrm{a}}-1)+(K_{\mathrm{d}}-1)),
\end{align*}
where the variance follows from the independence between $A$ and $S$. Thus, by Lemma \ref{lemma:bound_negative_drift}, we obtain the following lower bound: for all $T \geq \psi_1^2+1$ and $(1-\rho)M \leq \sigma/\psi_1$, 
\begin{align}
    &\mathbb{E}\left[\max\left\{\sum_{t=1}^{T-1}\Big(X(t)-G(t)\Big),\sum_{t=2}^{T-1}\Big(X(t)-G(t)\Big),\dots, X(T-1)-G(T-1),0\right\}\right]\nonumber\\
    &=\mathbb{E}\left[\max\left\{0,X(1)-G(1),\sum_{t=1}^{2}\Big(X(t)-G(t)\Big),\dots,\sum_{t=1}^{T-1}\Big(X(t)-G(t)\Big) \right\}\right]\nonumber\\
    &\geq \frac{\underline{c}}{2}\min\left\{M\sqrt{(\rho^2(K_{\mathrm{a}}-1)+(K_{\mathrm{d}}-1))(T-1)}, \frac{M(\rho^2(K_{\mathrm{a}}-1)+(K_{\mathrm{d}}-1))}{(1-\rho)}\right\},\label{lower_queue}
\end{align}
where $\psi_1$ is the unique positive solution to $\frac{4(\bar{c}/\sigma+1)}{\underline c}\log(1+x)+\frac{8}{\underline c}= x$; $\sigma=M\sqrt{\rho^2(K_{\mathrm{a}}-1)+(K_{\mathrm{d}}-1)}$, $\underline{c}$ and $\bar{c}$ are the constants defined in Lemmas~\ref{lem:sharp-lb} and~\ref{lem:KMT} respectively. 

We now construct a specific scheduling set:
\begin{align}\label{construction_D}
\mathcal{D}_\star=\left\{(x_1,x_2,\dots,x_n) :  0\leq x_i\leq B \ \text{for} \ 1\leq i\leq n \right\}.
\end{align}
Note that this construction of $\mathcal{D}_\star$ ensures that scheduling sequence $\mathcal{D}$ satisfies the condition in the model class $\mathcal{M}^ {\rho}(B,C_{\mathrm{a}}, C_{\mathrm{d}})$. In this case,  we choose $\lambda=(\rho B, \rho B,\dots,\rho B)$ so that 
\begin{align}\label{choice_parameter}
    M=nB,\ K_{\mathrm{a}}-1=\frac{C_{\mathrm{a}}^2}{\rho^2B^2}, \ \text{and} \ K_{\mathrm{d}}-1=\frac{C_{\mathrm{d}}^2}{B^2}.
\end{align}
Substitute \eqref{choice_parameter} into \eqref{lower_queue}. Observe that for $\rho \in [1/2,1]$, the random variables $(X(t)-\rho M)+(M-G(t))$ have uniformly bounded support across all $\rho$. Indeed,
\[
|(X(t)-\rho M)+(M-G(t))|\leq \big(\rho((K_{\mathrm{a}}-1) \vee 1)+(K_{\mathrm{d}}-1)\vee 1\big)M\leq nB\left(\frac{2C_{\mathrm{a}}^2}{B^2} \vee 1+\frac{C_{\mathrm{d}}^2}{B^2}\vee 1\right)
\]
and the variance of $(X(t)-\rho M)+(M-G(t))$ is  $\sigma^2=n^2(C_{\mathrm{a}}^2+C_{\mathrm{d}}^2)$. Consequently, their Sakhanenko parameters are uniformly bounded away from zero by a positive constant depending only on $n$, $C_{\mathrm{a}}$, $C_{\mathrm{d}}$ and $B$. By Lemma~\ref{lem:KMT}, the constant $\bar{c}$ can therefor be chosen to depend only on $n$, $C_{\mathrm{a}}$, $C_{\mathrm{d}}$ and $B$.

Let $c_1=\underline{c}/2$ and define $\psi_1(n,B,C_{\mathrm{a}},C_{\mathrm{d}})$ as the unique positive solution to $\frac{4(\bar{c}/\sigma+1)}{\underline c}\log(1+x)+\frac{8}{\underline c}= x$, we recover the result in \eqref{ieq:lower_bound_dependent_rho}.
\end{proof}

\subsubsection{Proof of Remark \ref{rem:psi1-scaling}}\label{appendix_rem_scaling}

\begin{proof}[Proof of Remark \ref{rem:psi1-scaling}.]
Recall that Remark~\ref{rem:psi1-scaling} claims that the quantity $\psi_1(n,B,C_{\mathrm a},C_{\mathrm d})$ can be chosen independent of $n$. In fact, we show that it satisfies
\[
\psi_1(n,B,C_{\mathrm a},C_{\mathrm d})
=O\left((1+R^3)\log(1+R^3)\right) \quad\mbox{where}\quad R=\frac{B}{\sqrt{C_{\mathrm a}^2+C_{\mathrm d}^2}}{\Big(\frac{2C_{\mathrm a}^2}{B^2}\vee 1+\frac{C_{\mathrm d}^2}{B^2}\vee 1\Big)}.
\]
Moreover, in the regimes in which  $\sqrt{C_{\mathrm a}^2+C_{\mathrm d}^2}\ll B$ and $\sqrt{C_{\mathrm a}^2+C_{\mathrm d}^2} \gg B$, the above scaling simplifies to
\[
\psi_1(n,B,C_{\mathrm a},C_{\mathrm d})
=O\left(\tilde R^3\log\tilde R^3\right),
\]
where 
\begin{align*}
   \tilde R=\begin{cases}
   \frac{B}{\sqrt{C_{\mathrm a}^2+C_{\mathrm d}^2}}, & \sqrt{C_{\mathrm a}^2+C_{\mathrm d}^2} \ll B,\\
       \frac{\sqrt{C_{\mathrm a}^2+C_{\mathrm d}^2}}{B} & \sqrt{C_{\mathrm a}^2+C_{\mathrm d}^2} \gg B.
   \end{cases} 
\end{align*}
 Now we  prove the above claims. Let
\[
\chi =nB\left(\frac{2C_{\mathrm{a}}^2}{B^2} \vee 1+\frac{C_{\mathrm{d}}^2}{B^2}\vee 1\right).
\]In fact, the Sakhanenko parameters of random variables $(X(t)-\rho M)+(M-G(t))$ are lower bounded by $\delta_0$, where $\delta_0$, by definition of the Sakhanenko parameter given in \eqref{def_Sakhanenko}, is the unique solution to
\begin{align*}
    \delta \chi^3 e^{\delta \chi}=\sigma^2.
\end{align*}
Solving this equation yields $\delta_0= \frac{1}{\chi} W\left(\frac{\sigma^2}{\chi^2}\right)$, where $W(\cdot)$ denotes the Lambert $W$ function defined by $W(u)e^{W(u)}=u$. By Lemma~\ref{lem:KMT}, the constant $\bar{c}$ can therefor be chosen as 
\begin{align*}
    \bar{c}=\frac{\omega}{\delta_0}+\omega\sigma
\end{align*}
for some absolute constant $\omega>0$. Consequently, 
$\psi_1$ is the unique solution of
$$
\psi=H \log (1+\psi)+\frac{8}{\underline{c}}, 
$$
where
$$
H=\frac{4}{\underline{c}}\left(1+\omega+\frac{\omega}{(\sigma / \chi) W\left(\sigma^2 / \chi^2\right)}\right).
$$
In fact, we can upper bound $\psi_1$ as  
\begin{align}\label{bound_Lambert}
    \psi_1 \leq 2H \log H+ \frac{8}{\underline{c}}.
\end{align}
Let
 \[x_0:=2H \log H+\frac{8}{\underline{c}}.
 \]
We claim that for all $H\geq e^{1 / 2}$,
\begin{align}\label{bound_x0}
    x_0 \geq H \log (1+x_0)+\frac{8}{\underline{c}}.
\end{align}
Since $\psi_1$ is the unique solution to $\psi=H \log (1+\psi)+\frac{8}{\underline{c}}$, it is the unique zero of
\[
F(\psi):=\psi-H \log (1+\psi)-\frac{8}{\underline{c}}.
\]
Moreover, $F$ is increasing on $\left[\psi_1, \infty\right)$. Hence, \eqref{bound_x0} implies $F(x_0)\geq 0$ and therefore $\psi_1 \leq x_0$, which establishes \eqref{bound_Lambert}.
We now verify \eqref{bound_x0}. Using the inequality $1+2s \log s\leq s^2$ for $s\geq 1$, we obtain
\[
1+2H \log H\leq  H^2.
\]
Therefore,
\[
H\log(1+2H \log H)\leq 2H\log H.
\]
Next, we use the bound $\log (a+t)= \log a+\log(1+t/a)\leq \log a+t/a$ for $a,t >0$. Take $a=1+2H\log H$ and $t=\frac{8}{\underline{c}}$, we have
\begin{align*}
 H \log (1+x_0) &\leq H\log(1+2H\log H)+\frac{8 H}{\underline{c}(1+2 H \log H)}  \\
 &\leq 2H\log H+\frac{4}{\underline{c}\log H}\leq  x_0.
\end{align*}
where the last inequality uses $\log H \geq 1/2$ (e.g., $ H \geq e^{1 / 2} $). This proves \eqref{bound_x0} and completes the proof of~\eqref{bound_Lambert}.

It suffices to characterize the order of $H$. Let $u=\frac{\sigma^2}{\chi^2}$. We first consider the case $u \in(0,1]$ (i.e., $\sigma \leq \chi)$.
Since $W(u) \leq 1$ for $u \leq 1$, we have $u=W(u) e^{W(u)} \leq e W(u)$,
which implies $W(u) \geq \frac{u}{e}$. Therefore,
$$
\frac{1}{(\sigma / \chi) W(u)} \leq e\left(\frac{\chi}{\sigma}\right)^3 .
$$
Next, consider $u \geq 1$ (i.e., $\sigma \geq \chi$).
The Lambert $W$ function is increasing and satisfies $W(1) \approx 0.567>\frac{1}{2}$. Hence $W(u) \geq \frac{1}{2}$. Consequently,
$$
\frac{1}{(\sigma / \chi) W(u)} \leq \frac{2 \chi}{\sigma} \leq 2 .
$$
Combining the two cases yields
\[H=O(1+(\chi / \sigma)^3).
\]
Consequently,
\begin{align*}
\psi_1=O\left(\left(1+(\chi / \sigma)^3\right) \log \left(1+(\chi / \sigma)^3\right)\right).
\end{align*}
Substituting the definitions of $\chi$ and $\sigma$ then yields  the claim stated in Remark~\ref{rem:psi1-scaling}. 
 
\end{proof}

\subsection{Proofs of Theorems and Supporting Results for Section~\ref{sec:upper}}
\subsubsection{Formal Proof of the Upper Bound in (\ref{ieq:upp_bound})}\label{appendix_thm_upper}

For any scheduling sets, arrival processes and random departure fields in $\mathcal{M}^ {\rho}(B,C_{\mathrm{a}}, C_{\mathrm{d}})$, and any scheduling policy $\Phi$, consider the one-step Lyapunov drift. Here for notational simplicity and without risk of ambiguity, we write $\mathbb{E}^\Phi_{(A, S)}$ as $\mathbb{E}$. First, observe that
\begin{align}\label{ieq:derandom_d}
    \mathbb{E}\left[\sum_{i=1}^n(\max\{Q_i(t)-D_i(t),0\})^2\middle| \mathcal{F}_t\right]\leq \mathbb{E}\left[\sum_{i=1}^n(\max\{Q_i(t)-d_i(t),0\})^2\middle| \mathcal{F}_t\right]+\sum_{i=1}^n\text{Var}(D_i(t)).
\end{align}
To see this, define $g(d_i)=(\max\{x_i-d_i,0\})^2$. Since $|g'(d_i)-g'(s_i)|\leq 2|d_i-s_i|$, the descent lemma yields for any $D_i$,
\begin{align*}
   g(D_i)\leq g(d_i)+g'(d_i)(D_i-d_i)+(D_i-d_i)^2. 
\end{align*}
Taking conditional expectations given $\mathcal{F}_t$ gives
\begin{align*}
   \mathbb{E}[g(D_i)|\mathcal{F}_t]\leq \mathbb{E}[g(d_i)|\mathcal{F}_t]+\text{Var}(D_i). 
\end{align*}
Taking the summation over $i$ completes the proof of~\eqref{ieq:derandom_d}.

The one-step Lyapunov drift satisfies
\begin{align*}
    &\mathbb{E}\left[V\big(Q(t+1)-A(t+1)\big)-V\big(Q(t)-A(t)\big)\mid \mathcal{F}_t\right]\\
    &=\mathbb{E}\left[\sum_{i=1}^n(\max\{Q_i(t)-D_i(t),0\})^2\middle| \mathcal{F}_t\right]-\sum_{i=1}^n(Q_i(t)-A_i(t))^2\\
    &\leq \sum_{i=1}^n\mathbb{E}\left[(\max\{Q_i(t)-d_i(t),0\})^2\middle| \mathcal{F}_t\right]+\sum_{i=1}^n\text{Var}(D_i(t))-\sum_{i=1}^n(Q_i(t)-A_i(t))^2\\
    &\leq \sum_{i=1}^n\left(\mathbb{E}\left[(Q_i(t)-d_i(t))^2\middle| \mathcal{F}_t\right]-Q_i(t)^2+2Q_i(t)A_i(t)-A_i(t)^2\right)+\sum_{i=1}^n\text{Var}(D_i(t))\\
    &= \sum_{i=1}^n\mathbb{E}\left[d_i(t)^2-A_i(t)^2+2Q_i(t)\big(A_i(t)-d_i(t)\big)\middle| \mathcal{F}_t\right]+\sum_{i=1}^n\text{Var}(D_i(t))\\
    &= f(Q(t),d(t))+r(Q(t), A(t))+\sum_{i=1}^n\text{Var}(D_i(t)),
\end{align*}
where
\begin{align*}
f(Q(t),d) &:= \underbrace{\mathbb{E}\left[2 \sum_{i=1}^n Q_i(t)\big(\lambda_i(t) - d_i\big)\middle| \mathcal{F}_t\right]}_{\text{first-order term}}+\underbrace{\sum_{i=1}^n\left(d_i^2 - \lambda_i(t)^2\right)}_{\text{second-order term}},
\end{align*}
and 
\begin{align*}
r(Q(t),A(t))=&\sum_{i=1}^n \left[2Q_i(t)(A_i(t)-\lambda_i(t))+\lambda_i(t)^2-A_i(t)^2\right].
\end{align*}
Note that
\begin{align*}
   &\mathbb{E}\big[Q_i(t)A_i(t)\big] \\
   &=\mathbb{E}\left[\big(Q_i(t)-A_i(t)\big)A_i(t)\right]+\mathbb{E}\big[A_i(t)^2\big]\\
   &\overset{(a)}{=}\mathbb{E}\big[Q_i(t)-A_i(t)\big]\mathbb{E}\big[A_i(t)\big]+\mathbb{E}\big[A_i(t)^2\big]\\
   &\overset{(b)}{=}\mathbb{E}\big[Q_i(t)\big]\lambda_i(t)+\text{Var}(A_i(t)),
\end{align*}
where ($a$) follows from the independence between $Q_i(t)-A_i(t)$ and $A_i(t)$, and ($b$) follows from the relation $\mathbb{E}\big[A_i(t)^2\big]=\lambda_i(t)^2+\text{Var}(A_i(t))$.

Then we have
\[
\mathbb{E}[r(Q(t),A(t))]=\sum_{i=1}^n\text{Var}(A_i(t)).
\]

By taking expectation and summing over for $1\leq t\leq T-1$, we have
\begin{align}
  &\mathbb{E}[V(Q(T)-A(T)]\nonumber \\
  &\leq  \mathbb{E}[V(Q(1)-A(1))]+\sum_{t=1}^{T-1}\mathbb{E}[f(Q(t),d(t))]+\sum_{t=1}^{T-1}\sum_{i=1}^{n}(\text{Var}(A_i(t))+\text{Var}(D_i(t)))\nonumber\\
  &\leq \sum_{t=1}^{T-1}\mathbb{E}[f(Q(t),d(t))]+(T-1)n(C_{\mathrm{a}}^2+C_{\mathrm{d}}^2),\label{ieq:drift_bound_app}
\end{align}
where the last inequality follows from the relation $Q(1)=A(1)$.
Note that
\begin{align*}
  \bigg(\mathbb{E}\Big[\sum_{i=1}^n\big(Q_i(T)-A_i(T)\big)\Big]\bigg)^2&\overset{(c)}{\leq} \mathbb{E}\bigg[\Big(\sum_{i=1}^n\big(Q_i(T)-A_i(T)\big)\Big)^2\bigg]\\
  &\overset{(d)}{\leq} n\mathbb{E}\bigg[\sum_{i=1}^n \big(Q_i(T)-A_i(T)\big)^2\bigg]\\&=n\mathbb{E}\big[V\big(Q(T)-A(T)\big)\big].
\end{align*}
Here ($c$) follows from Jensen's inequality and ($d$) follows from the Cauchy--Schwartz inequality.
Substituting this expression into \eqref{ieq:drift_bound_app}, we have
\begin{align*}
   \mathbb{E}\Bigg[\sum_{i=1}^n Q_i(T) \Bigg]\leq 
   n\sqrt{\sum_{t=1}^{T-1} \frac{1}{n}\ \mathbb{E} [f(Q(t),d(t))] +(T-1)(C_{\mathrm{a}}^2+C_{\mathrm{d}}^2)}+\sum_{i=1}^n \mathbb{E}[A_i(T)].
\end{align*}

\subsubsection{Formal Proof of Theorem \ref{thm:lower_bound_maxweight}}\label{appendix_thm_lower_maxweight}

\begin{proof}[Proof of Theorem~\ref{thm:lower_bound_maxweight}]

\begin{enumerate}[label={\bf Step \arabic*.}, leftmargin=*, align=left]
\item 
Fix an instance $(\mathcal{D},A,S)$ in the model class $\mathcal{M}^{1}(B,C_{\mathrm{a}},C_{\mathrm{d}})$. For $\rho \in (0,1]$, let $\{Q^{(\rho)}(t)\}_{t\in \mathbb{N}_0}$ and $\{d^{(\rho)}(t)\}_{t\in \mathbb{N}_0}$ denote the queue length process and the corresponding MaxWeight decisions when the arrival process $\{A(t)\}_{t\in \mathbb{N}}$ is replaced by the scaled process $\{\rho A(t)\}_{t\in \mathbb{N}}$. For brevity, when $\rho=1$, we write $\{Q(t)\}_{t\in \mathbb{N}_0}$ and $\{d(t)\}_{t\in \mathbb{N}_0}$ in place of $\{Q^{(1)}(t)\}_{t\in \mathbb{N}_0}$ and $\{d^{(1)}(t)\}_{t\in \mathbb{N}_0}$, respectively. Suppose that each scheduling set $\mathcal{D}_t$ is a polytope and that the MaxWeight policy yields a unique maximizer of $\langle Q(t),d\rangle$ at every time (with the exception of $t=0$, where we fix an arbitrary schedule). Under this assumption, we establish that
\[
\mathbb{E}_{(A, S)}\left[\sum_{i=1}^{n} Q_i^{(\rho)}(T)\right]
\]
is continuous in $\rho$ on an interval $(1-\delta,1]$ for some $\delta>0$. 
 
By definition, the $\MaxWeight$ policy always selects an extreme point of the scheduling set. 
Since the extreme points of the polytope is finite, there exists $\delta>0$  such that for $\rho \in (1-\delta,1]$, the maximizer of $\langle Q^{(\rho)}(t),d\rangle$ remains unchanged, that is,
$d^{(\rho)}(t) = d(t)$ for all $t \le T$ when $\rho \in (1-\delta,1]$. 
Consequently, for every sample path, $\sum_{i=1}^n Q_i^{(\rho)}(t)$ is continuous in~$\rho$ on the interval $(1-\delta,1]$. 
Moreover, because
\[
\sum_{i=1}^n Q_i^{(\rho)}(T) \leq \sum_{t=1}^T\sum_{i=1}^n A_i(T),
\]
the family $\{\sum_{i=1}^nQ_i^{(\rho)}(T)\}_{\rho\in(1-\delta,1]}$ is uniformly integrable. 
By the dominated convergence theorem, the expectation $\mathbb{E}_{(A, S)}[\sum_{i=1}^n Q^{(\rho)}_i(T)]$ is continuous in~$\rho$ on the interval $(1-\delta,1]$.

\item
We begin by considering a family of instances $\mathcal{I}_{b,1}$, which is given in~\eqref{exp_qb} and illustrated in Figure~\ref{fig:maxweight_lowerbound_q} for the case $q=1$.
For the instance $\mathcal{I}_{b,1}$ with $b\geq 6$, we can show (as done in the proof following Step~5) that for all time horizons $T$ satisfying $\left\lceil \frac{b^2}{b-1}\right\rceil \leq T\leq \left\lceil (\frac{b}{2}-1)^3\right\rceil+1$, 
\begin{align}\label{lower_maxweight:C=0}
 \mathbb{E}_{(A,S)}\left[\sum_{i=1}^2 Q_i(T)\right] \geq \frac{\sqrt{bT}}{2\sqrt{2}}.
\end{align}
\item We now consider the family of instances $\mathcal{I}_{b,q}$, which is given in~\eqref{exp_qb} and illustrated in Figure~\ref{fig:maxweight_lowerbound_q}.
To ensure that the instance lies within the model class $\mathcal{M}^ {1}(B,C_{\mathrm{a}}, C_{\mathrm{d}})$,  we require
\begin{align*}
  \frac{q^2}{2} \leq B^2, \ \text{and} \ \frac{(qb)^2}{2}\leq B^2.
\end{align*}
Recalling that $b > 1$, both conditions are satisfied by setting $q = \frac{\sqrt{2}B}{b}$.   It then follows that $\mathcal{I}_{b, \frac{\sqrt{2}B}{b}} \in \mathcal{M}^ {1}(B,C_{\mathrm{a}}, C_{\mathrm{d}})$ for all $b > 1$. Consequently, the queue lengths under the MaxWeight policy for the instance $\mathcal{I}_{b, \frac{\sqrt{2}B}{b}}$ are exactly the queue lengths for the instance $\mathcal{I}_{b, 1}$ multiplied by the scaling factor $q = \frac{\sqrt{2}B}{b}$, which can be shown by induction. Suppose $\widetilde{Q}(t)=qQ(t)$. Since MaxWeight selects $d(t) \in \arg \max_{d\in \mathcal{D}_t} \langle Q(t), d \rangle$, it follows that 
$$qd(t) \in \argmax_{d\in q\mathcal{D}_t} \langle Q(t), d \rangle = \argmax_{d\in q\mathcal{D}_t} \langle \widetilde{Q}(t), d \rangle.$$ 
Therefore, recalling $D(t)=d(t)$, the queue dynamics satisfy
\begin{align*}
   \widetilde{Q}(t+1)= \max\{\widetilde{Q}(t) - qD(t), \mathbf{0}\} + qA(t)=q (\max\{Q(t) - D(t), \mathbf{0}\} + A(t))=qQ(t+1).
\end{align*}
By induction, this holds for all $t\in \mathbb{N}$.
As a result, applying inequality~\eqref{lower_maxweight:C=0}, for any instance $\mathcal{I}_{b, \frac{\sqrt{2}B}{b}}$ with $b\geq 6$, and for all time horizons satisfying $\left\lceil \frac{b^2}{b-1}\right\rceil \leq T\leq \left\lceil (\frac{b}{2}-1)^3\right\rceil+1$, we have
\begin{align}\label{lower_maxweight:q}
 \mathbb{E}_{(A,S)}\left[\sum_{i=1}^2 Q_i(T)\right] \geq \frac{B\sqrt{T}}{2\sqrt{b}}.
\end{align}

\item For any $T\geq 9$, we can choose the instance $\mathcal{I}_{b, \frac{\sqrt{2}B}{b}}$ with $b\geq 6$ such that $T=\left\lceil (\frac{b}{2}-1)^3\right\rceil+1$. For this instance, the total queue length at time $T$ satisfies inequality~\eqref{lower_maxweight:q}.
Moreover, since $T=\left\lceil (\frac{b}{2}-1)^3\right\rceil+1$, it follows that $b\leq 3T^{\frac{1}{3}}$, and hence
$$\mathbb{E}_{(A,S)}\left[\sum_{i=1}^2 Q_i(T)\right] \geq \frac{BT^{\frac{1}{3}}}{2\sqrt{3}}.$$
Therefore, for all $T\geq 9$,
$$\sup_{\{\mathcal{I}_{b,\frac{\sqrt{2}B}{b}}: b\geq 6\}} \mathbb{E}_{(A,S)}\left[\sum_{i=1}^2 Q_i(T)\right]\geq \frac{BT^{\frac{1}{3}}}{2\sqrt{3}}.$$

By the continuity of $\mathbb{E}_{(A,S)}\left[\sum_{i=1}^2 Q_i^{(\rho)}(T)\right]$ in $\rho$ on the interval $(1-\delta,1]$ for each instance $\mathcal{I}_{b, \frac{\sqrt{2}B}{b}}$, we further obtain
\[
\liminf_{\rho\uparrow 1}\sup_{\{\mathcal{I}_{b,\frac{\sqrt{2}B}{b}}: b\geq 6\}}\mathbb{E}_{(A,S)}\left[\sum_{i=1}^2 Q_i^{(\rho)}(T)\right]\geq\sup_{\{\mathcal{I}_{b,\frac{\sqrt{2}B}{b}}: b\geq 6\}} \lim_{\rho\uparrow 1}\mathbb{E}_{(A,S)}\left[\sum_{i=1}^2 Q_i^{(\rho)}(T)\right]\geq \frac{BT^{\frac{1}{3}}}{2\sqrt{3}}.
\]
Since $\mathcal{I}_{b, \frac{\sqrt{2}B}{b}} \in \mathcal{M}^ {1}(B,C_{\mathrm{a}}, C_{\mathrm{d}})$ for all $b > 1$, we conclude that for all $T\geq 9$,
\begin{align}\label{lower_maxweight:BT^1/3}
 \liminf_{\rho\uparrow 1}\sup_{\mathcal{M}^ {\rho}(B,C_{\mathrm{a}}, C_{\mathrm{d}})}\mathbb{E}_{( A,S)}\left[\sum_{i=1}^2 Q_i(T)\right] \geq  \liminf_{\rho\uparrow 1}\sup_{\mathcal{M}^ {1}(B,C_{\mathrm{a}}, C_{\mathrm{d}})}\mathbb{E}_{(A,S)}\left[\sum_{i=1}^2 Q_i^{(\rho)}(T)\right] \geq \frac{BT^{\frac{1}{3}}}{2\sqrt{3}},
\end{align}
where the first inequality uses that scaling the arrival process of any instance $(\mathcal{D}, A,S) \in \mathcal{M}^{1}(B,C_{\mathrm{a}}, C_{\mathrm{d}})$ by $\rho$ yields an instance in $\mathcal{M}^{\rho}(B,C_{\mathrm{a}}, C_{\mathrm{d}})$.
\item 
Combining the bound in~\eqref{lower_maxweight:BT^1/3} with the lower bound in~\eqref{ieq:lower_bound_dependent_rho} from Theorem~\ref{thm_lower_bound}, we obtain the following:
 for all $T \geq \psi_1(n,B,C_{\mathrm{a}},C_{\mathrm{d}})^2+9$, we have
\begin{align*}
 \liminf_{\rho\uparrow 1}\sup_{\mathcal{M}^ {\rho}(B,C_{\mathrm{a}}, C_{\mathrm{d}})}\mathbb{E}_{(A,S)}\left[\sum_{i=1}^2 Q_i(T)\right]
&\geq  \max\left\{2c_1\sqrt{(C_{\mathrm{a}}^2+C_{\mathrm{d}}^2)(T-1)}, \frac{BT^{\frac{1}{3}}}{2\sqrt{3}}\right\}\\
&\geq   c_1\sqrt{(C_{\mathrm{a}}^2+C_{\mathrm{d}}^2)(T-1)}+\frac{BT^{\frac{1}{3}}}{4\sqrt{3}},
\end{align*}
which completes the proof of the theorem.
\end{enumerate}
We now prove inequality~\eqref{lower_maxweight:C=0}. Before presenting the formal argument, we first outline the main idea.

\paragraph{Proof Idea for Inequality~\eqref{lower_maxweight:C=0}.}
The choices of $b$ and $\varepsilon$ ensure that no two schedules tie under MaxWeight at finite time, so the policy always selects either $(0,1)$ or $(b,0)$.
Under MaxWeight, one can show: At each time \(t\), MaxWeight selects \((b,0)\) unless 
  \(\frac{Q_2(t)}{Q_1(t)}\ge b\).  
  As a result, $Q_2(t)$ must accumulate to $b$ before it forces the policy to use $(0,1)$; however, doing so causes $Q_1(t)$ to increase to 2, requiring $Q_2(t)$ to build up to $2b$ before $(0,1)$ can be used again more frequently. More precisely, let $t_0=0$ and
  \begin{align*}
      t_k=\min\{t: Q_2(t)\geq kb\} \ \text{for} \ 1\leq k\leq \left\lfloor \frac{b}{2}\right\rfloor.
  \end{align*}
  Under the MaxWeight policy, we observe the following behavior:  when $t_0\leq t< t_1$, $d(t)=(b,0)$, so the average growth rate of $Q_2$ is $\frac{b-1}{b}$; when $t_1\leq t< t_2$, $Q_2(t)\geq b-1$, and $d(t)$ alternates between $(b,0)$ and $(0,1)$  until $Q_2(t)=2b$, the average growth rate of $Q_2$ is $\frac{b-1}{b}-\frac{1}{2}$; when $t_2\leq t< t_3$, $Q_2(t)\geq 2b-1$, and each use of $(b,0)$ is followed by two uses of $(0,1)$ until $Q_2(t)=3b$, the average growth rate of $Q_2$ is $\frac{b-1}{b}-\frac{2}{3}$.

  More generally, for all $k\leq \lfloor b/2\rfloor$, during the interval $t_{k-1}\leq t< t_k$, $Q_2(t)\geq (k-1)b-1$, and each use of $(b,0)$ is followed by $k$ uses of $(0,1)$ until $Q_2(t)=kb$, the average growth rate of $Q_2$ is $\frac{b-1}{b}-\frac{k-1}{k}$. Moreover, since the interval length satisfies  $t_k-t_{k-1}=\Theta(kb)$, it follows that $t_{k}=\Theta(k^2b)$. Note that for $t_{k} \leq t< t_{k+1}$, $Q_2(t)\geq kb-1$. Combining these observations yields $Q_2(t)= \Theta(\sqrt{bt})$.

\paragraph{ Formal Proof for Inequality~\eqref{lower_maxweight:C=0}.}
  Now we present the formal proof. Consider the MaxWeight policy under the instance~\eqref{exp_qb} with $q=1$,
where $b\geq 6$ is rational and we choose $0<\varepsilon<\frac{b-1}{b}$ to be a small irrational constant. 


  By the queue dynamics, $Q(1)=(1,\frac{b-1}{b}-\varepsilon)$. MaxWeight repeatedly selects $(b,0)$ until $Q_2(t)\geq b$. Thus $Q(t)=(1, \frac{b-1}{b}t-\varepsilon)$ for $1\leq t < \lceil \frac{b(b+\varepsilon)}{b-1}\rceil$. When $t=\lceil \frac{b(b+\varepsilon)}{b-1}\rceil$, we have $Q_2(t)>b=bQ_1(t)$ so MaxWeight choose $(0,1)$. There is no tie in this decision, since $Q_2(t)$ is irrational while $Q_1(t)$ is rational. Note that once $Q_2(t)$ falls below $b$, MaxWeight will always choose $(b,0)$. Consequently,  for all $t \geq \lceil \frac{b(b+\varepsilon)}{b-1}\rceil$,  $Q_2(t)$ can never drop below $b-1$. This ensures that the $-\varepsilon$ term is always preserved in $Q_2(t)$, making it irrational for all $t$. In the following, we consider only time horizons where $Q_1(t) < b$ holds throughout, ensuring that $Q_1(t)$ always remains an integer.
 Consequently, there is never a tie in MaxWeight decisions, and it either selects $(b,0)$ or $(0,1)$. Therefore, we henceforth restrict our analysis to the MaxWeight policy under 
 \begin{align}
    &\mathcal{D}_\star=\{(b,0),(0,1)\}; \nonumber\\
    &A(t)=\Big(1, \frac{b-1}{b}\Big) \ \text{for all} \ t \geq 2, \ \text{and} \ A(1)=\Big(1, \frac{b-1}{b}-\varepsilon\Big);\nonumber \\
    & S_d(t)=d, \qquad \forall d \in \mathcal{D}_\star,\label{appendix_reduced_model}
\end{align}
where $b\geq 6$ is rational and $0<\varepsilon<\frac{b-1}{b}$ is an irrational constant. In addition, we choose $\varepsilon$ small enough so that $\left\lceil \frac{b(b+\varepsilon)}{b-1}\right\rceil=\left\lceil \frac{b^2}{b-1}\right\rceil$. This is possible because $\frac{b^2}{b-1}$ is not an integer for any $b\geq 6$.

  At each time $t$, the MaxWeight policy will choose $(0,1)$ only when $Q_2(t)>bQ_1(t)$ and $(b,0)$ when $bQ_1(t)>Q_2(t)$.
  Let $t_0=0$ and
  \begin{align*}
      t_k=\min\{t: Q_2(t)\geq kb\} \ \text{for} \ 1\leq k\leq \left\lfloor \frac{b}{2}\right\rfloor.
  \end{align*}
We will show, by induction, that 
\begin{enumerate}
    \item[(i)] $t_k$ is well defined;
    \item[(ii)] $Q_1(t_k)=1$;
    \item[(iii)] For $t_{k-1}\leq t< t_k$, $1\leq Q_1(t)\leq k$; $(k-1)b-(k-1)/b\leq Q_2(t)<kb$.
\end{enumerate}
For $k=1$, by the above discussion,  $t_1=\lceil \frac{b(b+\varepsilon)}{b-1}\rceil$ is well defined, $Q_1(t_1)=1$, and (iii) naturally holds. 

Now suppose that (i)-(iii) hold for some $k-1$ with $k\geq 2$. We show (i)-(iii) also hold for $k$. Note the fact that at any time $t$, choosing $(0,1)$ gives
\[
Q_1(t+1) = Q_1(t) + 1, Q_2(t+1) = Q_2(t) - \frac{1}{b},
\]
whereas choosing $(b,0)$ yields
\[
Q_1(t+1) = \max\{Q_1(t)-b,0\} + 1 = 1,
Q_2(t+1) = Q_2(t) + 1 - \frac{1}{b}.
\]
Thus $Q_1$ increases strictly only when $(0,1)$ is used, and $Q_2$ increases strictly only when $(b,0)$ is used. Starting from $t\geq t_{k-1}$, and noting $Q_1(t_{k-1})=1$,  consider the next $k$ time slots. There must be at least one use of $(b,0)$, because if the first $k-1$ departures were all $(0,1)$, then $Q_1(t_{k-1}+k-1)= k$, while $Q_2(t_{k-1}+k-1)<   (k-1)b+1-(k-1)/b<kb$, forcing the $k$th departure to be $(b,0)$. Moreover, if only consider $k-1$ time slots after $t_{k-1}$, at most one $(b,0)$ can be used, because for $s\leq k$, $Q_2(t_{k-1}+s)\geq (k-1)b-s/b>(k-2)b$. Note that $Q_2(t_{k-1}+k)\geq (k-1)b+1-k/b>(k-1)b$.
By the same reasoning and the fact that $Q_2(t_{k-1}+k+s)\geq (k-1)b+1-(k+s)/b>(k-1)b$ for $s\leq k$, in each subsequent block of $k$ slots, $Q_2$ increases by exactly $1-k/b>0$. Hence $t_k$ is well defined, and since $Q_2$ only increases when $(b,0)$ is applied, we conclude $Q_1(t_k)=1$. Finally, (iii) follows directly from the analysis above. And we complete the induction. 
In addition, we also have the following properties:
\begin{enumerate}
    \item For all $t_{k-1}\leq t\leq t_k$, 
    \begin{align}\label{app_ieq:kb}
     Q_1(t)+Q_2(t)\geq (k-1)b. 
    \end{align}
    \item Each interval length satisfies
    \begin{align*}
   & t_1 -t_0 =\left\lceil \frac{b(b+\varepsilon)}{b-1}\right\rceil,\\
  &\left\lceil\frac{k(b-1)}{1-k/b}\right\rceil-1 \leq t_k-t_{k-1}\leq \left\lceil \frac{kb}{1-k/b}\right\rceil, \ 2 \leq k \leq \left\lfloor \frac{b}{2}\right\rfloor.
\end{align*}
\end{enumerate}
By inequality~\eqref{app_ieq:kb}, for all $t_{k-1}\leq t\leq t_k$,
\begin{align*}
    Q_1(t)+Q_2(t)\geq(k-1)b=\sqrt{t}b\frac{k-1}{\sqrt{t}}\geq \sqrt{t}b\frac{k-1}{\sqrt{t_k}}. 
\end{align*}
Note that
\begin{align*}
    t_k=\sum_{i=1}^{k}(t_i-t_{i-1})\leq \left\lceil \frac{b(b+\varepsilon)}{b-1}\right\rceil+\sum_{i=2}^{k}\left\lceil \frac{ib}{1-i/b}\right\rceil\leq \sum_{i=1}^{k}\left\lceil \frac{ib}{1-i/b}\right\rceil+1.
\end{align*}
The remainder of the proof relies on two technical claims.
\paragraph{Claim 1.} For $b\geq 4$ and $k\geq 2$,
\begin{align*}
    \frac{k-1}{\sqrt{t_k}}\geq \frac{1}{2\sqrt{2b}}.
\end{align*}
\paragraph{Claim 2.} For $b\geq 4$,
\begin{align*}
    t_{\lfloor\frac{b}{2}\rfloor} \geq (\frac{b}{2}-1)^3+1.
\end{align*}
With the help of Claim 1, for $t_1\leq t \leq t_{\lfloor\frac{b}{2}\rfloor} $,
\begin{align*}
    Q_1(t)+Q_2(t)\geq  \frac{\sqrt{bt}}{2\sqrt{2}}.
\end{align*}
Since $t_1= \left\lceil \frac{b(b+\varepsilon)}{b-1}\right\rceil\leq \frac{b^2}{b-1}+1$, and when $b\geq 6$, we have $\frac{b^2}{b-1}+1\leq \left\lceil (\frac{b}{2}-1)^3\right\rceil+1$. Thus, for $b\geq 6$, for all $\left\lceil \frac{b^2}{b-1}\right\rceil \leq t\leq \left\lceil (\frac{b}{2}-1)^3\right\rceil+1$, 
\begin{align*}
    Q_1(t)+Q_2(t)\geq \frac{\sqrt{bt}}{2\sqrt{2}}.
\end{align*}
 Substituting $b=\sqrt{2}B$ completes the proof.

\end{proof}
We now proceed to prove these two claims.

\begin{proof}[Proof of Claim 1.]

For $2\le k\le \lfloor b/2\rfloor$ and each $1\le i\le k$, we have $1-i/b \geq 1/2$. Hence
\[
\left\lceil \frac{ib}{1 - i/b}\right\rceil
\leq \frac{ib}{1 - i/b} + 1 \leq 2ib + 1.
\]
Summing over \(i=1,2,\dots,k\) yields
\begin{align*}
t_k \leq \sum_{i=1}^{k}(2ib + 1)+1=bk(k + 1)+k+1\leq 2bk^2.
\end{align*}
The last inequality holds under the assumptions $b\geq 4$ and $k\geq 2$.
Hence
\[
\frac{k-1}{\sqrt{t_k}}\geq
\frac{k-1}{k\,\sqrt{2b}}=\frac{1-\tfrac1k}{\sqrt{2b}}\geq \frac{1}{2\sqrt{2b}},
\]
where the final inequality uses $k\ge2$, which implies $1-\tfrac1k\ge\tfrac12$.
\end{proof}

\begin{proof}[Proof of Claim 2.]
   Note that
   \begin{align*}
       t_{\lfloor\frac{b}{2}\rfloor}= \sum_{i=1}^{\lfloor\frac{b}{2}\rfloor}(t_i-t_{i-1})\geq \sum_{i=1}^{\lfloor\frac{b}{2}\rfloor}\left(\left\lceil \frac{i(b-1)}{1-i/b}\right\rceil-1\right)+1.
   \end{align*}
And 
\begin{align*}
\sum_{i=1}^{k}\Bigl(\left\lceil \tfrac{i(b-1)}{1 - i/b}\right\rceil - 1\Bigr)+1
&\ge
\sum_{i=1}^k\Bigl(\frac{ib(b-1)}{b-i} - 1\Bigr)+1
\geq\frac{bk(k+1)}{2} - k+1.
\end{align*}
Then 
\[
\sum_{i=1}^{\lfloor b/2\rfloor}\Bigl(\left\lceil \tfrac{i(b-1)}{1 - i/b}\right\rceil - 1\Bigr)+1 \geq \frac{b\lfloor b/2\rfloor(\lfloor b/2\rfloor+1)}{2} - \lfloor b/2\rfloor+1 \geq (\frac{b}{2}-1)^3+1.
\]
\end{proof}

\subsubsection{Formal Proof of Corollary~\ref{coro:gap}}
\label{app:coro_gap}
\begin{proof}[Proof of Corollary~\ref{coro:gap}]
    We continue to use the same construction as in \eqref{exp_qb} with $q=1$ and $B=b/\sqrt{2}$, under which the lower bound for MaxWeight is given by \eqref{lower_maxweight:C=0}.  For $\LyapOpt$, since the arrival rate lies in the scheduling set $\mathcal{D}_\star$ and  $C_{\mathrm{a}}=C_{\mathrm{d}} = 0$, Theorem~\ref{thm:lyapopt} and \eqref{upp_bound:achieve_low_bound} yield
  $$\sum_{i=1}^2 Q_i^{\LyapOpt}(T)=\sum_{i=1}^2 A_i(T)=2-\frac{1}{b}-\varepsilon\leq 2-\frac{1}{\sqrt{2}B},$$
  which completes the proof.
\end{proof}

\subsection{Proofs of Theorems and Supporting Results for Section~\ref{sec:stability}}

\subsubsection{Proof of Theorem~\ref{thm:stability_of_lyapopt}}\label{app:stability}
We invoke the following Foster–Lyapunov lemma (see, e.g., \citet{meyn2012markov}).
\begin{lemma}[Foster--Lyapunov Criterion]
\label{lem:foster-lyapunov}
Let $\{X_t\}_{t\in\mathbb{N}_0}$ be a Markov chain on a countable state space $\mathcal{X}$. 
Suppose there exists a function $V:\mathcal{X}\to [0,\infty)$, a finite subset $W\subset \mathcal{X}$, 
and constants $\epsilon>0$ and $b<\infty$ such that
    \[
       \mathbb{E}[V(X_{t+1})-V(X_t)\mid X_t=x] \le -\epsilon\mathbbm{1}_{\{x\notin W\}}+b\mathbbm{1}_{\{x\in W\}}.
    \]
Then the Markov chain $\{X_t\}_{t\in\mathbb{N}_0}$ is stochastically stable in the sense that it reaches the recurrent states with probability~1, and every recurrent state is positive recurrent.
\end{lemma}
\begin{proof}[Proof of Theorem~\ref{thm:stability_of_lyapopt}]
Under the $\LyapOpt$ policy, the one-step Lyapunov drift
\begin{align*}
    &\mathbb{E}^\LyapOpt_{(A,S)}\left[V(Q(t+1))-V(Q(t))\mid Q(t)=x\right]\\
   & =\mathbb{E}_{(A,S)}\left[\sum_{i=1}^n(\max\{x_i(t)-D_i^\LyapOpt(t),0\}+A_i(t+1))^2\, \bigg| \, Q(t)=x\right]-\sum_{i=1}^nx_i(t)^2\\
    &\leq \sum_{i=1}^n\left(\mathbb{E}\left[(\max\{x_i(t)-D_i^\LyapOpt(t),0\})^2\middle| Q(t)=x\right]+2x_i(t)\mathbb{E}[A_i(t+1)]+\mathbb{E}[A_i(t+1)^2]-x_i(t)^2\right),
\end{align*}
where the inequality uses the independence between $A(t+1)$ and $Q(t)$.
Use the same argument as in~\eqref{ieq:derandom_d} of Appendix \ref{appendix_thm_upper}, we have 
\begin{align*}
    &\sum_{i=1}^n\mathbb{E}\left[(\max\{x_i(t)-D_i^\LyapOpt(t),0\})^2\middle| Q(t)=x\right]\\
    &\leq \sum_{i=1}^n\mathbb{E}\left[(\max\{x_i(t)-d_i^\LyapOpt(t),0\})^2\middle| Q(t)=x\right]+\sum_{i=1}^n \text{Var}(D_i^\LyapOpt(t))\\
    &=\min_{d\in \mathcal{D}_\star} \sum_{i=1}^n(\max\{x_i(t)-d_i,0\})^2 + \sum_{i=1}^n \text{Var}(D_i^\LyapOpt(t))\\
    &\leq \min_{d\in \mathcal{D}_\star} \sum_{i=1}^n(x_i(t)-d_i)^2 + \sum_{i=1}^n \text{Var}(D_i^\LyapOpt(t))\\
    &\leq \min_{d\in \mathcal{D}_\star} \sum_{i=1}^n-2x_i(t)d_i + \max_{d\in \mathcal{D}_\star} \sum_{i=1}^n d_i^2+\sum_{i=1}^n x_i^2+\sum_{i=1}^n \text{Var}(D_i^\LyapOpt(t)),
\end{align*}
where the equality uses the definition of $\LyapOpt$, and the last inequality is by the fact that $\min(f+g)\leq \min f+\max g$.
Therefore,
\begin{align*}
 &\mathbb{E}^\LyapOpt_{(A,S)}\left[V(Q(t+1))-V(Q(t))\mid Q(t)=x\right]\\
    &\leq \min_{d\in \mathcal{D}_\star}\sum_{i=1}^n2x_i(t)(\lambda_i(t+1)-d_i) +2nB^2+nC_{\mathrm{a}}^2+nC_{\mathrm{d}}^2\\
    &\leq -2(1-\rho) \max_{d\in \mathcal{D}_\star}\sum_{i=1}^n x_i(t)d_i +2nB^2+nC_{\mathrm{a}}^2+nC_{\mathrm{d}}^2,
\end{align*}
where the last inequality follows from  $\lambda(t+1) \in \rho \Pi(\mathcal{D}_\star)$, which implies $$\sum_{i=1}^n x_i(t)\lambda_i(t+1)\leq \max_{d\in \mathcal{D}_\star}\rho\sum_{i=1}^n x_i(t)d_i.$$
 Then we choose
\begin{align*}
    W=\left\{x\in \mathbb{Z}_+^n: \max_{d\in \mathcal{D}_t}\sum_{i=1}^n x_id_i\leq \frac{\epsilon+n(2B^2+C_{\mathrm{a}}^2+C_{\mathrm{d}}^2)}{2(1-\rho)}\right\}.
\end{align*}
Since $W$ is finite, it satisfies the condition in Lemma~\ref{lem:foster-lyapunov}, which completes the proof.
\end{proof}

\subsubsection{Proof of Proposition~\ref{prop:stability_of_lyapopt_adv} and Related Lemma}\label{app:proof_thm6}

The next lemma follows directly from Assumption~\ref{assump:ucc}.
\begin{lemma}
\label{lem:ucc_implies_star}
Under Assumption~\ref{assump:ucc}, for all $t\in\mathbb{N}_0$ and all $x\in\mathbb{R}_+^n$,
\begin{equation}
\label{eq:star}
\max_{d\in \mathcal{D}_t} \langle x,d\rangle \geq \frac{\alpha}{n}\sum_{i=1}^n x_i.
\end{equation}
\end{lemma}

\begin{proof}[Proof of Lemma~\ref{lem:ucc_implies_star}]
Fix $t\in\mathbb{N}_0$ and $x\in\mathbb{R}_+^n$. By Assumption~\ref{assump:ucc}, for each $j$ there exists a schedule vector $d^{(j)}(t)\in D_t$ such that its $j$-th component $d^{(j)}_j(t)\geq\alpha$.  Hence
\[
\max_{d\in \mathcal{D}_t} \langle x,d\rangle
\geq
\max_{1\leq j\leq n} \langle x,d^{(j)}(t) \rangle
\geq
\max_{1\leq j\leq n} x_j d^{(j)}_j(t)
\geq
\alpha\,\max_{1\le j\le n} x_j .
\]
Since $\max_i x_i \ge \frac{1}{n}\sum_{i=1}^n x_i$, inequality~\eqref{eq:star} follows.
\end{proof}

\begin{proof}[Proof of Proposition~\ref{prop:stability_of_lyapopt_adv}]
 By inequality~\eqref{eqn:Delta_Vt} and the definition of $\LyapOpt$, the one-step Lyapunov drift under the $\LyapOpt$ policy satisfies
    \begin{align*}
    &\mathbb{E}^\LyapOpt_{(A,S)}\left[V(Q(t+1)-A(t+1))-V(Q(t)-A(t))\mid \mathcal{F}_t\right]\\
    & \leq \min_{d\in \mathcal{D}_{t}} f(Q(t),d) + r(Q(t),A(t))+\sum_{i=1}^n\text{Var}(D_i^\LyapOpt(t))\\
    &= \min_{d\in \mathcal{D}_{t}}\left(2 \sum_{i=1}^n Q_i(t)\big(\lambda_i(t) - d_i\big)+\sum_{i=1}^n\left(d_i^2 - \lambda_i(t)^2\right)\right) + r(Q(t),A(t))+\sum_{i=1}^n\text{Var}(D_i^\LyapOpt(t))\\
    &\leq \min_{d\in \mathcal{D}_{t}}2 \sum_{i=1}^n Q_i(t)\big(\lambda_i(t) - d_i\big)+nB^2 + r(Q(t),A(t))+\sum_{i=1}^n\text{Var}(D_i^\LyapOpt(t))\\
    &\leq -2(1-\rho) \max_{d\in \mathcal{D}_t}\sum_{i=1}^n Q_i(t)d_i+nB^2 + r(Q(t),A(t))+\sum_{i=1}^n\text{Var}(D_i^\LyapOpt(t)),
\end{align*}
where the last inequality follows from  $\lambda(t) \in \rho \Pi(\mathcal{D}_t)$, which implies $$\sum_{i=1}^n Q_i(t)\lambda_i(t)\leq \max_{d\in \mathcal{D}_t}\rho\sum_{i=1}^n Q_i(t)d_i.$$
Taking expectations and summing over $t$ yields
\begin{align*}
    \mathbb{E}^\LyapOpt_{(A,S)}[V(Q(T)-A(T))]\leq -2(1-\rho) \sum_{t=1}^{T-1}\max_{d\in \mathcal{D}_t}\sum_{i=1}^n \mathbb{E}^\LyapOpt_{(A,S)}[Q_i(t)]d_i+n(T-1)(B^2 + C_{\mathrm{a}}^2+C_{\mathrm{d}}^2).
\end{align*}
By Lemma~\ref{lem:ucc_implies_star}, 
\begin{align*}
    \max_{d\in \mathcal{D}_t}\sum_{i=1}^n \mathbb{E}^\LyapOpt_{(A,S)}[Q_i(t)]d_i\geq \frac{\alpha}{n}\sum_{i=1}^n \mathbb{E}^\LyapOpt_{(A,S)}[Q_i(t)].
\end{align*}
Then we have the uniform stability bound \eqref{uni_upper_bound}.
\end{proof}

\subsection{Supplementary Experiments in Section~\ref{subsec:experiments_gap}}
\label{sec_appendix: supple_exp}
The following two figures confirm that, in the setting of the Corollary \ref{coro:gap}, the total queue length under MaxWeight eventually becomes bounded (i.e., $O(1)$ in $T$) as $T$ increases.
\begin{figure}[ht]
    \centering
    \begin{subfigure}[t]{0.48\textwidth}
        \includegraphics[width=\linewidth]{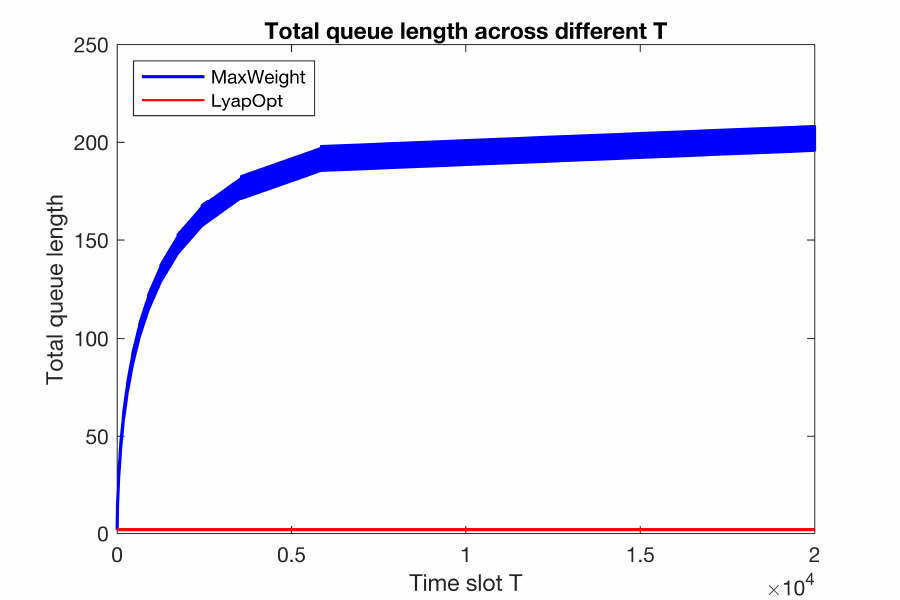}
        \caption{Total queue length when $B=10$, $C_{\mathrm{a}}=C_{\mathrm{d}}=0$}
        \label{fig_app:performance_b10}
    \end{subfigure}%
    \hfill
    \begin{subfigure}[t]{0.48\textwidth}
        \includegraphics[width=\linewidth]{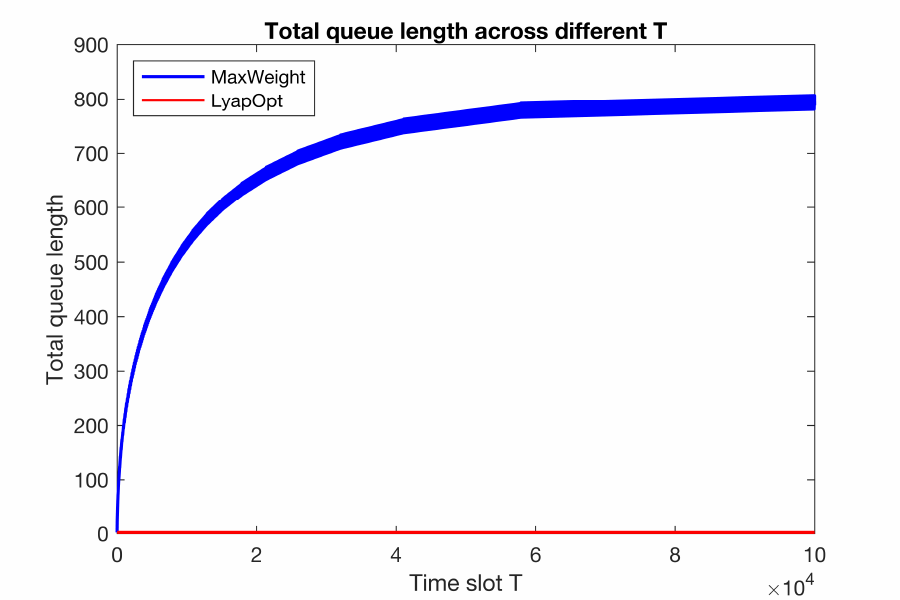}
        \caption{Total queue length when $B=20$, $C_{\mathrm{a}}=C_{\mathrm{d}}=0$}
        \label{fig:performance_b20}
    \end{subfigure}
    \caption{Performance comparison of MaxWeight and LyapOpt policies versus $B$} 
    \label{fig_appendix:all_results}
\end{figure}

\subsection{Formal Proofs of Results  in Section~\ref{sec:extension}}
\subsubsection{Proof of Proposition \ref{thm:lower_bound_indep}}\label{app: thm_lower_bound_indep}
\begin{proof}[Proof of Proposition \ref{thm:lower_bound_indep}.]
We focus on the case $C_{\mathrm{a}}^2 + C_{\mathrm{d}}^2 > 0$; as the case $C_{\mathrm{a}}^2 + C_{\mathrm{d}}^2 = 0$ is straightforward.
 Fix a time-invariant scheduling sequence $\mathcal{D}$ with $\mathcal{D}_t=\mathcal{D}_\star$ for all $t\in \mathbb{N}$, where $\mathcal{D}_\star$ is defined in~\eqref{construction_D}. Let 
$$M=\max\left\{\sum_{i=1}^n d_i : d\in \mathcal{D}_\star\right\}=nB$$
denote the maximum total service capacity in a single time slot.
  We construct arrival process $A$ and random departure field $S$ are mutually independent. For each $t$, the components $A_i(t)$ are independent across $i$ and distributed as
\begin{align*}
    &\mathbb{P}\Big(A_i(t)=K_{\mathrm{a}}\lambda_i\Big)=\frac{1}{K_{\mathrm{a}}},\quad \mathbb{P}\Big(A_i(t)=0\Big)=\frac{K_{\mathrm{a}}-1}{K_{\mathrm{a}}},
\end{align*}
where the choices of $\lambda$ and $K_{\mathrm{a}}$ follow those specified in \eqref{choice_lambda_Ka}. Thus, the mean and variance conditions for $A$ in \eqref{eq:mean_variance_a} still holds.
In particular, the total variance across all queues is given by
\begin{align*}
    \sum_{i=1}^n\text{Var}(A_i(t))=\sum_{i=1}^n(K_{\mathrm{a}}-1)\lambda_i^2=nC_{\mathrm{a}}^2.
\end{align*}
For each $t$ and $d\in \mathcal{D}_t$, the components of $(S_d(t))_i$ are independent across $i$ and distributed as
\begin{align*}
    &\mathbb{P}\Big((S_d(t))_i=K_{\mathrm{d}} d_i\Big)=\frac{1}{K_{\mathrm{d}}},\quad \mathbb{P}\Big((S_d(t))_i=0\Big)=\frac{K_{\mathrm{d}}-1}{K_{\mathrm{d}}},
\end{align*}
where the choice of $K_{\mathrm{d}}$ follows those specified in \eqref{choice_Kd}. This confirms that the constructed pair $(A, S)$ belongs to the model class $\mathcal{M}_{\mathcal{I}}^ {\rho}(B,C_{\mathrm{a}}, C_{\mathrm{d}})$. Recall that $D(t)=S_{d(t)}(t)$. So the components $D_i(t)$ are independent across $i$ and distributed as
\begin{align*}
    &\mathbb{P}\Big(D_i(t)=K_{\mathrm{d}} d_i(t)\Big)=\frac{1}{K_{\mathrm{d}}},\quad \mathbb{P}\Big(D_i(t)=0\Big)=\frac{K_{\mathrm{d}}-1}{K_{\mathrm{d}}}.
\end{align*}
Let $\{G(t)\}_{t\in \mathbb{N}}$ be an i.i.d. sequence of random variables defined by $G(t)=\sum_{i=1}^n Y_i(t)$, where $Y_i(t)$ are independent across $i$ and distributed as
\begin{align*}
    &\mathbb{P}\Big(Y_i(t)=K_{\mathrm{d}}B\Big)=\frac{1}{K_{\mathrm{d}}},\quad \mathbb{P}\Big(Y_i(t)=0\Big)=\frac{K_{\mathrm{d}}-1}{K_{\mathrm{d}}},
\end{align*}
so that
\begin{align*}
 \mathbb{P}(G(t)\geq y)\geq \mathbb{P}(\sum_{i=1}^n D_i(t)\geq y) \ \text{ for each $y$}, \ \mathbb{E}[G(t)]=M, \ \text{and} \ \text{Var}(G(t))=nC_{\mathrm{d}}^2 .
\end{align*} 
The following argument follows the same idea as in the proof of Theorem~\ref{thm_lower_bound}.
Let $X(t)=\sum_{i=1}^n A_i(t)$. By \eqref{appen:lower_queue1}, the queue length satisfies
\begin{align*}
    \sum_{i=1}^n Q_i(T) & \geq\max\left\{\sum_{t=1}^{T-1}\Big(X(t)-G(t)\Big),\sum_{t=2}^{T-1}\Big(X(t)-G(t)\Big),\dots, X(T-1)-G(T-1),0\right\}+X(T).
    \end{align*}
 Observe that
 \[
 X(t)-G(t)=(X(t)-\rho M)+(M-G(t))-(1-\rho)M.
 \]
Note that $\{(X(t)-\rho M)-(M-G(t))\}_{t \in \mathbb{N}}$ forms an i.i.d. sequence of random variables, satisfying
\begin{align*}
\mathbb{E}[(X(t)-\rho M)-(M-G(t))]=0, \text{Var}((X(t)-\rho M)+(M-G(t)))=nC_{\mathrm{a}}^2+nC_{\mathrm{d}}^2,
\end{align*}
where the variance follows from the independence of all components. Thus, by Lemma \ref{lemma:bound_negative_drift}, we obtain the following lower bound: for all $T \geq \psi_2^2+1$ and $(1-\rho)M \leq \sigma/\psi_2$, 
\begin{align*}
    &\mathbb{E}\left[\max\left\{\sum_{t=1}^{T-1}\Big(X(t)-G(t)\Big),\sum_{t=2}^{T-1}\Big(X(t)-G(t)\Big),\dots, X(T-1)-G(T-1),0\right\}\right]\nonumber\\
    &=\mathbb{E}\left[\max\left\{0,X(1)-G(1),\sum_{t=1}^{2}\Big(X(t)-G(t)\Big),\dots,\sum_{t=1}^{T-1}\Big(X(t)-G(t)\Big) \right\}\right]\nonumber\\
    &\geq \frac{\underline{c}}{2}\min\left\{\sqrt{n(C_{\mathrm{a}}^2+C_{\mathrm{d}}^2)(T-1)}, \frac{C_{\mathrm{a}}^2+C_{\mathrm{d}}^2}{B(1-\rho)}\right\},
\end{align*}
where $\psi_2$ is the unique positive solution to $\frac{4(\bar{c}/\sigma+1)}{\underline c}\log(1+x)+\frac{8}{\underline c}= x$; $\sigma=\sqrt{nC_{\mathrm{a}}^2+nC_{\mathrm{d}}^2}$; $\underline{c}$ and $\bar{c}$ are the constants defined in Lemmas~\ref{lem:sharp-lb} and~\ref{lem:KMT} respectively. 

 Observe that for $\rho \in [1/2,1]$, the random variables $(X(t)-\rho M)+(M-G(t))$ have uniformly bounded support across $\rho$,
\[
|(X(t)-\rho M)+(M-G(t))|\leq nB\left(\frac{2C_{\mathrm{a}}^2}{B^2} \vee 1+\frac{C_{\mathrm{d}}^2}{B^2}\vee 1\right)
\]
and variance $\sigma^2=n(C_{\mathrm{a}}^2+C_{\mathrm{d}}^2)$. Consequently, their Sakhanenko parameters are uniformly bounded away from zero by a positive constant depending only on $n$, $C_{\mathrm{a}}$, $C_{\mathrm{d}}$ and $B$. By Lemma~\ref{lem:KMT}, the constant $\bar{c}$ can therefore be chosen to depend only on $n$, $C_{\mathrm{a}}$, $C_{\mathrm{d}}$ and $B$.

Let $c_1=\underline{c}/2$ and define $\psi_2(n,B,C_{\mathrm{a}},C_{\mathrm{d}})$ as the unique positive solution to $\frac{4(\bar{c}/\sigma+1)}{\underline c}\log(1+x)+\frac{8}{\underline c}= x$, we recover the result in \eqref{ieq:lower_bound_independent_rho}.
 
\end{proof}
\subsubsection{Proof of Proposition~\ref{prop:new_policy}}
\label{app: other model classes}
\begin{proof}[Proof of Proposition~\ref{prop:new_policy}]
Under the assumption that $\mathbb{E}[A(t)] = \lambda \in \mathcal{D}_\star$, we have 
\begin{align}\label{assum_lambda_in_d}
\sum_{i=1}^n \max \{ \widehat{A}_i(t)- d^{\mathrm{EP}}_i(t),0\} \leq \sum_{i=1}^n \max \{ \widehat{A}_i(t)- \lambda_i,0\}.
\end{align}
 Let $\Delta(t)=A(t)-D^{\mathrm{EP}}(t)$. Then
\begin{align*}
    \Delta(t) &= (A(t) - \lambda) + (\lambda - \widehat{A}(t)) + (\widehat{A}(t) - d^{\mathrm{EP}}(t))+(d^{\mathrm{EP}}(t)-D^{\mathrm{EP}}(t)) \\
    &\overset{\triangle}{=} \Delta^{(1)}(t) + \Delta^{(2)}(t)+\Delta^{(3)}(t)+\Delta^{(4)}(t).
\end{align*}
Under this policy, for each $i$,
\begin{align*}
    Q_i(T)-A_i(T)
          &=\max\bigg\{\sum_{t=1}^{T-1}\Delta_i(t),\sum_{t=2}^{T-1}\Delta_i(t),\dots,\Delta_i(T-1),0\bigg\}\\
          &\leq \sum_{j=1}^4\max\bigg\{\sum_{t=1}^{T-1}\Delta_i^{(j)}(t),\sum_{t=2}^{T-1}\Delta_i^{(j)}(t),\dots,\Delta_i^{(j)}(T-1),0\bigg\}.
\end{align*}
For the first term, let $S^{(1)}_i(t) = \sum_{s=T-t}^{T-1}\Delta^{(1)}_i(s)$ and $M^{(1)}_i(t)= \max_{1\leq s\leq  t} S^{(1)}_i(s)$ for $t\in \mathbb{N}$. Specifically, let
$S^{(1)}_i(0)=M^{(1)}_i(0)=0$. since the variables $\{\Delta^{(1)}_i(s)\}_{s\geq 1}$ are independent across $s$ and satisfy $\mathbb{E}[\Delta^{(1)}_i(s)]=0$, $\{S^{(1)}_i(t)\}_{t=0}^{T-1}$ is a martingale. Thus, we have
\begin{align*}
    &\mathbb{E}\left[\max\bigg\{\sum_{t=1}^{T-1}\Delta_i^{(1)}(t),\sum_{t=2}^{T-1}\Delta_i^{(1)}(t),\dots,\Delta_i^{(1)}(T-1),0\bigg\}\right]\\
    &\ = \mathbb{E}\Big[M^{(1)}_i(T-1)\Big]   \leq \sqrt{\mathbb{E}\Big[\big(M^{(1)}_i(T-1)\big)^2\Big]}\\
    &\ \leq \sqrt{4 \mathbb{E}\Big[\big(S^{(1)}_i(T-1)\big)^2\Big]}=2\sqrt{\sum_{t=1}^{T-1}\operatorname{Var}(A_i(1))},
\end{align*}
where the first inequality follows from the Jensen's inequality and the second inequality follows from   Doob's maximal inequality. Using the same argument, for the forth term,
we also have
\begin{align*}
   \mathbb{E}\left[\max\bigg\{\sum_{t=1}^{T-1}\Delta_i^{(4)}(t),\sum_{t=2}^{T-1}\Delta_i^{(4)}(t),\dots,\Delta_i^{(4)}(T-1),0\bigg\}\right] \leq 2\sqrt{\sum_{t=1}^{T-1}\operatorname{Var}(D_i(t))}.
\end{align*}
For the second term,
\begin{align*}
    &\max\bigg\{\sum_{t=1}^{T-1}\Delta^{(2)}_i(t),\sum_{t=2}^{T-1}\Delta^{(2)}_i(t),\dots,\Delta^{(2)}_i(T-1),0\bigg\}\leq \sum_{t=1}^{T-1}\max\{\Delta_i^{(2)}(t),0\}\leq \sum_{t=1}^{T-1}\big|\Delta_i^{(2)}(t)\big|.
\end{align*}
Note that
\begin{align*}
 \mathbb{E}\Big[\big|\Delta_i^{(2)}(t)\big|\Big] \leq \sqrt{\mathbb{E}\Big[\big(\Delta_i^{(2)}(t)\big)^2\Big]}= \sqrt{\frac{\operatorname{Var}(A_i(1))}{t}}.
\end{align*}
Then
\begin{align*}
  \sum_{t=1}^{T-1}\big|\Delta_i^{(2)}(t)\big|\leq  \sqrt{\text{Var}(A_i(1))} \sum_{s=1}^{T-1}\frac{1}{\sqrt{t}}&\leq \sqrt{\text{Var}(A_i(1))}\int_{s=1}^{T-1} \frac{1}{\sqrt{t}} \mathrm{d}t \leq 2\sqrt{(T-1)\operatorname{Var}(A_i(1))}.
\end{align*}
For the third term,
\begin{align*}
    &\sum_{i=1}^n\max\bigg\{\sum_{t=1}^{T-1}\Delta^{(3)}_i(t),\sum_{t=2}^{T-1}\Delta^{(3)}_i(t),\dots,\Delta^{(3)}_i(T-1),0\bigg\}\\
     & \ \leq \sum_{i=1}^n\sum_{t=1}^{T-1}\max\{\Delta_i^{(3)}(t),0\}= \sum_{t=1}^{T-1}\sum_{i=1}^n\max\{\widehat{A}_i(s)-d^{\mathrm{EP}}_i(t), 0\}\\
    & \ \leq  \sum_{t=1}^{T-1}\sum_{i=1}^n\max\{\widehat{A}_i(s)-\lambda_i, 0\}\leq  \sum_{t=1}^{T-1}\sum_{i=1}^n|\widehat{A}_i(s)-\lambda_i| \\
    & \ = \sum_{t=1}^{T-1}\sum_{i=1}^n\big|\Delta_i^{(2)}(t)\big|\leq \sum_{i=1}^n2\sqrt{(T-1)\operatorname{Var}(A_i(t))}.
\end{align*}
In sum, we have
\begin{align*}
    \mathbb{E}\bigg[\sum_{i=1}^n Q_i(T)\bigg]-\mathbb{E}\bigg[\sum_{i=1}^n A_i(T)\bigg] &\leq \sum_{i=1}^n 6\sqrt{(T-1)\operatorname{Var}(A_i(1))}+2\sum_{i=1}^n\sqrt{\sum_{t=1}^{T-1}\operatorname{Var}(D_i(t))}\\
    &\leq n\sqrt{T-1}(6C_{\mathrm{a}}+2C_{\mathrm{d}}).
\end{align*}
\end{proof}



\end{document}